\documentclass[12pt]{article}
\usepackage{amssymb,amsmath,amsthm,amsfonts,amscd,latexsym, dsfont, color, xcolor, tipa, tipx, mathrsfs, upgreek, eucal, wasysym, enumerate}
\usepackage{graphics}
\usepackage{hyperref} 
\usepackage{array}
\usepackage{tikz-cd}

\hypersetup{colorlinks}
\usepackage{tikz}
\usetikzlibrary{matrix,arrows,decorations.pathmorphing}

\definecolor{indigo}{RGB}{51,0,102}
\definecolor{brightpurple}{RGB}{102,0,153}

\definecolor{fuchsia}{RGB}{180,51,180}
\definecolor{jolightpurple}{RGB}{188,171,240}

\hypersetup{colorlinks,
linkcolor=brightpurple,
filecolor=brightpurple,
urlcolor=indigo,
citecolor=fuchsia}

\usepackage[margin=1in]{geometry}

\normalbaselines
\newcommand{\C}{\mathbb{C}}
\newcommand{\bbC}{\mathbb{C}}

\newcommand{\co}{\mathbb{C}^2  \setminus \{ \bf{0}\}}
\newcommand{\Z}{\mathbb{Z}}
\newcommand{\R}{\mathbb{R}}
\newcommand{\N}{\mathbb{N}}
\newcommand{\Q}{\mathbb{Q}}

\newcommand{\g}{\Gamma}

\newcommand{\A}{\lambda}

\newcommand{\we}{\wedge}

\newcommand{\bo}{\mathbf{0}}

\newcommand{\pa}{\partial}

\newcommand{\om}{\omega}

\newcommand{\M}{{\mathcal{M}}}

\newcommand{\ds}{\dot{\Sigma}}
\newcommand{\ind}{\mbox{ind}}

\newcommand{\cala}{\mathcal{A}}

\newcommand{\calb}{\mathfrak{B}}

\newcommand{\mhat}{{\mathcal{M}}^{\Jt}}

\newcommand{\calm}{\widehat{\mathcal{M}}^{\Jt}}

\newcommand{\calp}{\mathscr{P}}
\newcommand{\calpa}{\mathscr{P}({\A;\Gamma})}
\newcommand{\baar}{\Gamma}

\newcommand{\calc}{u}

\newcommand{\deebar}{\bar{\pa}_{\Jt}}

\newcommand{\lep}{\lambda_{\varepsilon}}

\newcommand{\czm}{\mu_{CZ}}
\newcommand{\rs}{\mu_{RS}}

\newcommand{\Sp}{\mbox{Sp}}
\newcommand{\id}{\mathds{1}}
\newcommand{\gax}{x}
\newcommand{\up}{y}

\newcommand{\mult}{\mbox{m}}

\newcommand{\ga}{\gamma}

\newcommand{\gpk}{\gamma_p^k}
\newcommand{\glp}{\gamma^\ell_+}

\newcommand{\gdm}{\gamma^d_-}
\newcommand{\gp}{{\gamma_+}}
\newcommand{\gm}{{\gamma_-}}

\newcommand{\lo}{\lambda}

\newcommand{\vepsilon}{\varepsilon}
\newcommand{\veps}{\varepsilon}

\newcommand{\lens}{L(n+1,n)}

\newcommand{\hopf}{S^1 \hookrightarrow S^3  \overset{h}{\longrightarrow} S^2}

\newcommand{\J}{\mathcal{J}}
\newcommand{\Jt}{{J}}

\newcommand{\sign}{\mbox{sign}}
\newcommand{\indx}{\mbox{index}}
\newcommand{\hind}{{\mbox{ind}}}

\makeatletter
\newtheorem*{rep@theorem}{\rep@title}
\newcommand{\newreptheorem}[2]{%
\newenvironment{rep#1}[1]{%
 \def\rep@title{#2 \ref{##1}}%
 \begin{rep@theorem}}%
 {\end{rep@theorem}}}
\makeatother

\newtheorem{theorem}{Theorem}
\newreptheorem{theorem}{Theorem}

\newtheorem{lemma}[theorem]{Lemma}
\newtheorem{corollary}[theorem]{Corollary}
\newtheorem{proposition}[theorem]{Proposition}
\newtheorem{example}[theorem]{Example}
\newtheorem{remark}[theorem]{Remark}

\newtheorem{definition}[theorem]{Definition}
\numberwithin{theorem}{section}
\numberwithin{equation}{subsection}
\numberwithin{figure}{section}

\title{Automatic transversality in contact homology II: filtrations and computations}
\author{Jo Nelson\footnote{Partially supported by NSF grants DMS-1303903 and DMS-184072.}}
\date{}

\begin{document}
\maketitle 
 \begin{abstract}
This paper is the sequel to the previous paper \cite{jo1}, which showed that sufficient regularity exists to define cylindrical contact homology in dimension three for nondegenerate dynamically separated contact forms, a subclass of dynamically convex contact forms.   The Reeb orbits of these so-called dynamically separated contact forms satisfy a uniform growth condition on their Conley-Zehnder indices with respect to a free homotopy class; see Definition \ref{taut}.  {Given a contact form which is dynamically separated up to large action, we demonstrate a filtration by action on the chain complex and show how to obtain the desired cylindrical contact homology by taking direct limits.}  We give a direct proof of invariance of cylindrical contact homology within the class of dynamically separated contact forms, {and elucidate the independence of the filtered cylindrical contact homology with respect to the choice of the dynamically separated contact form and almost complex structure.} We also show that these regularity results are compatible with geometric methods of computing cylindrical contact homology of prequantization bundles, proving a conjecture of Eliashberg \cite{yashaICM} in dimension three.  \\

   \end{abstract}

\tableofcontents

 \section{Motivation and results}\label{intro}
Cylindrical contact homology, as introduced by Eliashberg-Givental-Hofer \cite{EGH}, is in principle an invariant of contact manifolds that admit a nondegenerate contact form $\lambda$ without Reeb orbits of certain gradings.  The cylindrical contact homology of $(M,\xi)$ is defined by choosing a nondegenerate such contact form and taking the homology of a chain complex over $\Q$ which is generated by ``good'' Reeb orbits, and whose differential $\partial^{EGH} $ counts $J$-holomorphic cylinders in $\R \times M$ for a suitable almost complex structure $J. $ Unfortunately, in many cases there is no way to choose $J$ so as to obtain the transversality for holomorphic cylinders needed to define $\partial^{EGH} $, to show that $\left(\partial^{EGH} \right)^2 = 0$, and to prove that the homology is an invariant of the contact structure $\xi$.

In \cite{jo1}, we gave a rigorous construction of cylindrical contact homology for contact forms in dimension three whose Reeb orbits satisfy a uniform growth condition on their Conley-Zehnder indices with respect to a fixed free homotopy class.  Such contact forms are said to be \emph{dynamically separated}; a precise definition is given in Definition \ref{taut}.  {Given a dynamically separated contact form up to large action, we investigate action (and SFT-grading) filtered cylindrical contact homology.}  Our invariance results are obtained more directly than those which appeared in \cite{HN2} for the hypertight case and those to appear in \cite{HN3} for the class of dynamically convex contact forms.  We also provide computational methods for the class of dynamically separated contact forms associated to prequantization spaces and Seifert fiber spaces.

\begin{remark} [Relationship to the dynamically convex case] \em
In \cite{HN1}, we showed that for ``dynamically convex'' contact forms $\lambda$ in three dimensions, and for generic almost complex structures $J$, one can in fact define the differential $\partial^{EGH} $ by counting $J$-holomorphic cylinders without any abstract perturbation. We also showed that $\left(\partial^{EGH} \right)^2=0$ using a generic  almost complex structure, without breaking the $S^1$ symmetry.     However, this relied on certain technical assumptions, which hold when $\pi_1(M)$ is not torsion;  see (*) in Theorem 1.3 and Remark 1.4 of \cite{HN1}.  We expect these assumptions to be removable.  In the meantime, the dynamically separated case allows us to consider some dynamically convex contact forms which do not satisfy (*).   Obtaining invariance in the dynamically convex case is currently in preparation \cite{HN3}, which involves extending the machinery of \cite{HN2} with obstruction bundle gluing from \cite{HT2}.
\end{remark}

To define cylindrical contact homology in general, some kind of abstract perturbation  of the $J$-holomorphic curve equation is needed, for example using polyfolds or Kuranishi structures.   Hofer, Wysocki, and Zehnder have developed the abstract analytic framework \cite{HWZpoly1}-\cite{HWZgw}, collectively known as \emph{polyfolds}, to systematically resolve issues of regularizing moduli spaces.  Contact homology awaits foundations via polyfolds and the use of abstract perturbations can make computations difficult. 

Pardon \cite{p} has defined full contact homology via virtual fundamental cycles but this approach is not applicable to defining cylindrical contact homology in the presence of contractible Reeb orbits.    In dimension three, in the absence of contractible Reeb orbits, and when paired with the action filtered versions of \cite[Thm 1.6, 1.9]{HN2}, the definition provided by Bao-Honda in \cite{bh} can be shown to be isomorphic to the cylindrical contact homology.  Using virtual techniques, Bao-Honda \cite{bh2} give a definition of the full contact homology differential graded algebra for any closed contact manifold in any dimension.  The approaches of Pardon and the latter of Bao-Honda make use of Kuranishi structures to construct contact and symplectic invariants and while they hold more generally, they are more difficult to work with in computations and applications.

\bigskip 

\noindent \textbf{Organization of the article}.  
The rest of Section \ref{intro} gives an overview of cylindrical contact homology, a discussion  of dynamically separated contact forms, a geometric means of computing cylindrical contact homology for prequantization bundles, and some examples.  We also discuss applications to dynamics.  Regularity results are proven in Section \ref{background}.  Index calculations and related arguments ruling out non cylindrical holomorphic buildings in cobordisms between dynamically separated contact manifolds are given in Section \ref{invariance}, yielding the desired invariance results for filtered cylindrical contact homology.  Conley-Zehnder index calculations associated to perturbations of prequantization bundles are given in Section \ref{cz-section}.  Finally, the proof of the Morse-Bott computational result is given in Section \ref{filtration}. 

\bigskip 

\noindent \textbf{Acknowledgements}. I thank Mohammed Abouzaid, Michael Hutchings, and Dusa McDuff for their interest in my work and insightful discussions. {I also thank the referee for their helpful comments and suggestions. } 

\subsection{Contact forms, Reeb vector fields, and gradings}\label{gradingsec}
 Let $(M^{2n-1}, \xi)$ be a co-oriented closed contact manifold of let $\lambda$ be a contact form such that ker $\lambda = \xi$.  The contact form $\lambda$ uniquely determines the \textbf{Reeb vector field} $R_\lambda$ by 
\[
\iota(R_\lambda)d\lambda=0, \ \ \ \lambda(R_\lambda)=1.   
\]
A (closed) \textbf{Reeb orbit} $\ga$ of period $T$ with $T>0$, associated to $R_\lambda$ is defined to be a map
\[
\ga: \R / T\Z \to M
\]
satisfying
\[
\dot{\ga}(t)= R_\lambda(\ga(t)), \ \ \ \ga(0)= \ga(T).
\]
Two Reeb orbits are considered equivalent if they differ by reparametrization, i.e. precomposition with any translation of $\R / T \Z$ corresponding to the choice of a starting point for the orbit.  

A Reeb orbit is said to be \textbf{simple} or equivalently, \textbf{embedded}, whenever the map $\ga: \R / T\Z \to M$ is injective.   If $\ga\colon \R/T\Z \to M$ is a simple Reeb orbit of period $T$ and $k$ a positive integer, then we denote $\ga^k$ to be the \textbf{$k$-fold cover} or iterate of $\ga$, meaning $\ga^k$ is the composition of $\ga$ with $\R/kT\Z \to \R / T\Z$ and has period $kT$.   We denote the the Reeb flow by $\varphi_t$, i.e. $ \dot{\varphi}_t = R_\lambda(\varphi_t).$

 A Reeb orbit is said to be \textbf{nondegenerate} whenever the linearized return map of the flow along $\gamma$,
\[ 
d\varphi_T: (\xi_{\ga(0)}, d\lambda) \to (\xi_{\ga(T)=\ga(0)},d\lambda)
\]
has no eigenvalue equal to 1.  If all the Reeb orbits associated to $\lambda$ are nondegenerate then $\lambda$ is said to be a \textbf{nondegenerate} contact form. 

The linearized flow of a $T$-periodic Reeb orbit $\ga$ yields a path of symplectic matrices given by
\[
d\varphi_t:\xi_{\gamma(0)} \to \xi_{\gamma(t)}, \ t\in[0,T].
\]
One can compute the Conley-Zehnder index of $d\varphi_t, \ t\in[0,T],$ however this index is typically dependent on the choice of trivialization $\Phi$ of $\xi$ along $\ga$ used in linearizing the Reeb flow. There is, however, always a canonical $\Z_2$-grading due to the axiomatic properties of the Conley-Zehnder index \cite{RS1,SZ}.  For $(M^{2n-1},\xi)$ this grading is obtained via 
\begin{equation}\label{Z2grading}
(-1)^{\czm(\gamma)}=(-1)^{n-1}\sign \det (\id - \Psi(T)),
\end{equation}
where $\Psi(t)_{t \in [0,T]} \in \mbox{Sp}(2n-2)$ is the linearized flow restricted to $\xi$ along a $T$-periodic Reeb orbit $\gamma$ with respect to the choice of symplectic trivialization $\Phi$ of $\xi$.  

In dimension three, one can classify a nondegenerate Reeb orbit $\gamma$ as being one of three types, depending on the eigenvalues $\Lambda$, $\Lambda^{-1}$ of the linearized flow return map $d\varphi_T|_\xi$:
\begin{enumerate}
\item[] $\gamma$ is \textbf{elliptic} if $\Lambda, \Lambda^{-1} := e^{\pm2\pi i \theta}$;
\item[] $\gamma$ is \textbf{positive hyperbolic} if  $\Lambda, \Lambda^{-1} > 0$; 
\item[] $\gamma$ is \textbf{negative hyperbolic} if $\Lambda, \Lambda^{-1} < 0$.
\end{enumerate}
 The parity of the Conley-Zehnder index does not depend on the choice of trivialization and is even when $\gamma$ is positive hyperbolic and odd otherwise, yielding the canonical $\Z_2$ grading in dimension 3. 

 
We will further need to classify Reeb orbits whose Conley-Zehnder index changes parity under iteration, a phenomenon which is always independent of the choice of trivialization.  

\begin{definition}\em
The $m$-fold closed Reeb orbit $\ga^m$ is \textbf{bad} if it is the $m$-fold iterate of a simple Reeb orbit $\ga$ such that the difference $\czm(\ga^{m}) - \czm(\ga)$ of their Conley-Zehnder indices is odd.   If a Reeb orbit is not bad then it is deemed to be a \textbf{good Reeb orbit}.  
 \end{definition}
 
 In dimension three, the set of bad orbits consists solely of the even iterates of simple negative hyperbolic orbits.  In higher dimensions, bad orbits can only arise from even multiple covers of nondegenerate simple orbits whose linearized return flow has an odd number of pairs of negative real eigenvalues $(\lambda, \lambda^{-1})$. 
  The set of all Reeb orbits in the free homotopy class $\Gamma$  is denoted by $\calpa$, and the set of  good Reeb orbits in a free homotopy class $\Gamma$ is denoted by $\mathscr{P}_{\mbox{\tiny good}}(\lambda;\Gamma)$.

In certain cases, one can upgrade the canonical $\Z_2$-grading.  For any $\lambda$-compatible $J$, the symplectic vector bundle $(\xi, d\lambda, J)$ has a natural $U(n-1)$ structure.  Since this bundle is a (almost) complex bundle, we can take its highest exterior power, which is the \textbf{anticanonical bundle} of $M$ and denoted by $\mathcal{K}^*$. The dual of $\mathcal{K}^*$ is the \textbf{canonical bundle}.   If $c_1(\xi; \Z)=0\in H^2(M;\Z)$  then one can trivialize the anticanonical bundle  $\mathcal{K}^*$.  Let 
 \[
 \widetilde{\Phi}: \mathcal{K}^* \to TM \times \mathbb{C}
 \]
 be a choice of such a trivialization.   This amounts to specifying a global complex volume form on $\R \times M$. If $H^1(M;\Q) =0$ then $\widetilde{\Phi}$ (as well as any complex volume form) is unique up to homotopy.  Now we can insist than any local trivialization $\Phi$ of $\xi$, which can be used to linearize the Reeb flow along $\gamma$ must agree with our ``canonically" determined trivialization $\widetilde{\Phi}$.  This gives rise to an absolute $\Z$-grading on the Reeb orbits.  

In this case one can sensibly refer to the Conley-Zehnder index of a Reeb orbit $\ga$, obtaining a $\Z$-grading on the Reeb orbits given by
 \begin{equation}\label{sftgrading}
 |\gamma|=\czm^\Phi(\gamma)+n-3.
 \end{equation}
 Here $\czm^\Phi(\gamma):=\czm(d\varphi_t)\arrowvert_{t\in[0,T]}$ is the Conley-Zehnder index of the path of symplectic matrices obtained from the linearization of the flow along $\ga$, restricted to $\xi$.  If $c_1(\xi; \Z)=0$ and  $H^1(M;\Q)\neq0$ then there is more than one homotopy class of trivializations associated to the complex line bundle that is the canonical representation of $-c_1(\xi)$, resulting in different choices of complex volume forms on  $(\R \times M, d(e^\tau \lambda), \Jt)$.  If $c_1(\xi; \Q)=0$ one can obtain a fractional $\Z$-grading, see  \cite[\S 3-4]{mclean} \cite{sgrade, biased}.  
 
{If we have that $c_1(\xi;\Z)$ vanishes on $\pi_2(M)$ then for each contractible Reeb orbit $\gamma$ we can define the Conley-Zehnder index $\gamma$ by $\czm(\gamma) = \czm^\Phi(\gamma)$, where $\Phi$ is a trivialization of $\xi_\gamma$ which extends to a trivialization of $\xi$ over a disk bounded by $\gamma$ of contractible loops there is a $\Z$-grading.  Should $c_1(\xi;\Q)=0$ and $H^1(M;\Q)=0$ then the trivialization of $\xi$ along a contractible closed Reeb orbit that extends to a capping disk will coincide with the homotopy class of the trivialization induced by a global complex volume form \cite[Lemma 4.3]{mclean}.}

 
 
 It is important to note that our trivializations are fixed up to homotopy; that is trivializations over iterated orbits must be homotopic to the iterated trivializations.  When the trivialization $\widetilde{\Phi}$ is available globally as when $c_1(\xi; \Z)=0$ this is straightforward, otherwise care must be taken in specifying local trivializations.

{We now give the definition of a \textbf{dynamically convex} contact form, a notion due to Hofer, Wysocki, and Zehnder.} This definition necessitates that the Conley-Zehnder index of contractible periodic orbits of the Reeb vector field be well-defined without any reference to a specific homotopy class of discs spanned by the orbits.  
This necessitates for every map $v: S^2 \to M$ that the integer $c_1(v^*\xi)([S^2])$ vanishes.  The stipulation that $c_1(v^*\xi)([S^2])\equiv 0$ is equivalent to $\psi_\xi \equiv 0$, where $\psi_\xi$ is the natural homomorphism defined by
\begin{equation}\label{naturalhomo}
\begin{array}{crcl}
\psi_\xi: & \pi_2(M) & \to & \Z, \\
& [\sigma] & \mapsto& c_1(v^*\xi). \\
\end{array}
\end{equation}

\begin{definition}\label{dyncon}\em
Let $\A$ be a nondegenerate contact form on a closed 3-manifold $M$.  We say that $\lambda$ is \textbf{dynamically convex} whenever
\begin{itemize}
\item $\lambda$ admits no contractible Reeb orbits, or
\item The map from (\ref{naturalhomo}) satisfies $\psi_\xi =0$  and every contractible Reeb orbit $\ga$ satisfies  $ \czm(\ga) \geq 3$.
\end{itemize}
\end{definition}
{ If $M$ is a compact star-shaped hypersurface in $\R^{4}$ then 
\[
\lambda = \frac{1}{2}\sum_{j=1}^2 \left( x_jdy_j - y_jdx_j \right)
\]
restricts to a contact form on $M$.  In \cite{HWZdyn} it is shown that if $M$ convex, and if $\lambda$ is nondegenerate then $\lambda$ is dynamically convex.  This property was used to give a remarkable characterization of the tight 3-sphere \cite[Theorem 1.5]{HWZdyn}.}

We are also interested in contact forms which do not admit contractible Reeb orbits.  A contact form $\lambda$ on $M^{2n-1}$ is said to be \textbf{hypertight} whenever the Reeb vector field associated to $\lambda$ admits no contractible Reeb orbits.  While historically inaccurate, we take the class of dynamically convex contact forms to include the set of hypertight contact forms.

\subsection{Dynamically separated contact forms}\label{dynsep}

The differentials (when well-defined) on the chain complex defining cylindrical contact homology preserve the free homotopy classes of Reeb orbits since they count cylinders which project to homotopies in $M$ between the Reeb orbits.  Furthermore, the chain maps (when well-defined)  also preserve the free homotopy classes of Reeb orbits.  The dynamically separated condition gives control on the Conley-Zehnder index of iterates of Reeb orbits in a specified non-primitive free homotopy class.  This permits us to achieve transversality for certain multiply covered cylinders in cobordisms and esnure that no noncylindrical levels are present in compactifications of curves to pseudoholomorphic buildings.

First we recall some preliminary notions with regard to free homotopy classes of loops.  
Fix a closed contact three manifold $(M,\xi)$. A \textbf{primitive} homotopy class of loops $\Gamma \in \pi_0(\Omega M)$ means that $\Gamma$ is not equal to $k\Gamma'$ for any $\Gamma' \in \pi_0(\Omega M)$ and an integer $k >1$.  As explained in \cite[\S 10]{wendl-sft}, all pseudoholomorphic cylinders interpolating between closed primitive Reeb orbits are somewhere injective, and hence regular provided $J$ is generic.  

One can define and obtain topological invariance of cylindrical contact homology with ``classical" methods for the following important subclass of hypertight contact forms in any dimension.

\begin{definition}[Def. 10.16 \cite{wendl-sft}] \em
Given a contact manifold $(M^{2n-1}, \xi)$ and a primitive homotopy class $\Gamma \in \pi_0(\Omega M)$ we say that a contact form $\lambda$ for $\xi$ is \textbf{$\Gamma$-admissible} if all the Reeb orbits homotopic to $\Gamma$ are nondegenerate and there are no contractible Reeb orbits.
\end{definition}

\begin{remark}\em
Standard SFT compactness \cite{BEHWZ} does not apply for sequences of pseudoholomorphic cylinders in the symplectization of a $\Gamma$-admissible contact manifold.  However, Wendl gives a direct proof of the desired result in \cite[Prop. 10.19]{wendl-sft}.
\end{remark}

In dimension three, this paper provides a means of defining cylindrical contact homology for non-primitive homotopy classes subject to the dynamically separated condition.  The definition of dynamically separated necessitates that $c_1(\ker \lambda)=0$ if there is more than one Reeb orbit in each free non-primitive homotopy class so that a $\Z$-grading is available.  
We first give the definition of dynamically separated when all Reeb orbits are contractible. 

\begin{definition} \em
Let $(M,\lambda)$ be a contact 3-manifold with $c_1(\ker \lambda)=0$ such that all the Reeb orbits of $R_\lambda$ are contractible. Then $\lambda$ is said to be \textbf{dynamically separated} whenever the following conditions hold.
\begin{enumerate}
\item[{(I)  }]  If $\ga$ is a closed simple Reeb orbit then $ 3 \leq  \czm({\ga}) \leq 5$;
\item[{(II)}] If $\ga^k$ is the $k$-fold cover of a simple orbit $\ga$ then $\czm(\ga^k) = \czm(\ga^{k-1})+4.$
\end{enumerate}
\end{definition}

The presence of noncontractible non-primitive Reeb orbits necessitates that we must keep track of the free homotopy class of a non-primitive Reeb orbit after each iteration of the underlying simple orbit.  {This is particularly important if the simple orbit is a torsion element of  $\pi_0(\Omega M)$, and some of if its iterates are contractible, as is the case for lens spaces, see Example \ref{lensex}.  This bookkeeping is important when ruling out breaking phenomena in Section \ref{invariance} and is used to define the following}  analogue of Condition II with respect to a free homotopy class $\baar \in \pi_0(\Omega M)$.

\begin{definition}\label{taut} \em
Let $(M,\lambda)$ be a contact 3-manifold with $c_1(\ker \lambda)=0$. Let $\gamma$ be a simple Reeb orbit. For each free homotopy class $\baar$, let 
\[
1 \leq k_1(\baar, \ga) < k_2(\baar, \ga) < ... <k_i(\baar, \ga) <...
\]
 be the (possibly empty or infinite) list of all integers such that all the $k_i(\baar,\ga)$-fold covers of $\ga$ lie in the same free homotopy class $\baar$.  We will use $\baar=0$ to represent the class of contractible orbits. A contact form $\lambda$ is said to be \textbf{dynamically separated} whenever the following conditions are satisfied. 
\begin{enumerate}
\item[{(I.i)  }]   For the class of contractible orbits, $\baar=0$, we have $ 3 \leq  \czm({\ga^{k_1(0, \ga)}}) \leq 5$;
\item[{(I.ii)}]   For each {non-primitive} $\baar \neq 0$ there exists  $m(\baar, \ga) \in \Z_{>0}$ such that  $2m-1 \leq \czm(\ga^{k_1(\baar, \ga)}) \leq 2m+1$;
\item[{(II)  }]  For each  {non-primitive} free homotopy class $\baar$ we have $\czm(\ga^{k_{i+1}(\baar, \ga)})=\czm(\ga^{k_{i}(\baar, \ga)})+4.$
\end{enumerate}
\end{definition}
{We note that (I.ii) is equivalent to requiring that $\czm(\gamma^{k_1(c,\gamma)})$ is a positive integer for each {non-primitive}  $c \neq 0$.  We have expressed this condition more pedantically to stress that the first iterates of a simple Reeb orbit representing different  {non-primitive} free homotopy classes need not have their Conley-Zehnder index agree. }

For computational methods it is often practical to consider contact forms which will be \textbf{dynamically separated up to (large) action}, which is proportional to the index.  This modification is explained in the following definition and we note that many Morse-Bott contact forms can be made dynamically separated up to large action by a small perturbation.

\begin{definition}\em
A contact form $\lambda$ is said to be $L$-\textbf{dynamically separated} whenever the following conditions are satisfied. 
\begin{enumerate}
\item[{(I.i)  }]   For the class of contractible orbits, $\baar=0$, we have $ 3 \leq  \czm({\ga^{k_1(0, \ga)}}) \leq 5$ and 
\[
 \mathcal{A}(\gamma^{k_1(0, \ga)};\lambda):=\int_{\gamma^{k_1(0, \ga)}} \lambda < {L};
\]
\item[{(I.ii)}]   For each {non-primitive} $\baar \neq 0$ there exists  $m(\baar, \ga) \in \Z_{>0}$ such that  $2m-1 \leq \czm(\ga^{k_1(\baar, \ga)}) \leq 2m+1$ and
\[
\mathcal{A}(\gamma^{k_1(\baar, \ga)};\lambda):=\int_{\gamma^{k_1(\baar, \ga)}} \lambda < {L};
\]
\item[{(II)  }]  For each {non-primitive} free homotopy class $\baar$ we have $\czm(\ga^{k_{i+1}(\baar, \ga)})=\czm(\ga^{k_{i}(\baar, \ga)})+4,$ whenever 
\[
\mathcal{A}(\gamma^{k_{i+1}(\baar, \ga)};\lambda):=\int_{\gamma^{k_{i+1}(\baar, \ga)}} \lambda < {L}.
\]
\end{enumerate}

\end{definition}

Examples of $L$-nondegenerate dynamically separated contact forms arise naturally from prequantization bundles; see Section \ref{examples-intro}.

 \subsection{Cylindrical contact homology}\label{attempt}

We say that an almost complex structure $J$ on $\R\times Y$ is \textbf{$\lambda$-compatible\/} if $J(\xi)=\xi$; $d\lambda(v,Jv)>0$ for nonzero $v\in\xi$; $J$ is invariant under translation of the $\R$ factor; and $J(\frac{\partial}{\partial \tau})=R$, where $\tau$ denotes the $\R$ coordinate.  In the following it should be assumed that we have chosen such a $J$ generically.

If $\gamma_+$ and $\gamma_-$ are Reeb orbits, we consider $J$-holomorphic cylinders between them, namely maps 
\[
u(s,t):=(a(s,t),f(s,t)): (\R \times S^1, j_0) \to (\R \times M, \Jt)
\]
satisfying the Cauchy-Riemann equation, 
\[
\bar{\pa}_{j,J}u:= du + J \circ du \circ j \equiv 0,
\]
such that $\lim_{s\to\pm_\infty}\pi_\R(u(s,t))=\pm\infty$, and $\lim_{s\to\pm\infty}\pi_Y(u(s,\cdot))$ is a parametrization of $\gamma_\pm$. Here $\pi_\R$ and $\pi_Y$ denote the projections from $\R\times Y$ to $\R$ and $Y$ respectively.


We declare two maps to be equivalent if they differ by translation and rotation of the domain $\R \times S^1$ and denote the set of equivalence classes by $\widehat{\mathcal{M}}^J(\gp;\gm)$.  Note that $\R$ acts on $\widehat{\mathcal{M}}^J(\gp;\gm)$ by translation of the $\R$ factor in $\R\times Y$.  We denote
\[
\M^J(\gamma_+,\gamma_-):=\widehat{\mathcal{M}}^J(\gp;\gm)/\R.
\]

Given $u$ as above, with respect to a suitable trivialization $\Phi$ of $\xi$ over $\gamma_+$ and $\gamma_-$, we define the {\bf Fredholm index\/} of $u$ by
\[
\mbox{ind}(u) = \czm^\Phi(\gamma_+) - \czm^\Phi(\gamma_-) + 2c_1^\Phi(u^*\xi).
\]

The significance of the Fredholm index is that if $J$ is generic and $u$ is somewhere injective, then $\widehat{\mathcal{M}}^J(\gamma_+,\gamma_-)$ is naturally a manifold near $u$ of dimension $\mbox{ind}(u)$. Let $\widehat{\mathcal{M}}^J_k(\gamma_+,\gamma_-)$ denote the set of $u\in\widehat{\mathcal{M}}^J(\gamma_+,\gamma_-)$ with $\mbox{ind}(u)=k$.

The cylindrical contact homology \textbf{chain complex} $C^{EGH}_*(M,\lambda, J)$ is generated by all nondegenerate closed \textbf{good} Reeb orbits of  $R_\lambda$ over $\Q$-coefficients, with grading determined by (\ref{sftgrading}).    Bad Reeb orbits must be excluded from the chain group because of issues involving orientations and invariance.   For a more detailed discussion on other choices of coefficients see Remark \ref{coefficients}.     

The chain complex splits over the free homotopy classes $\baar \in \pi_0(\Omega M)$ of Reeb orbits because the \textbf{differentials} are defined via a weighted count of rigid pseudoholomorphic cylinders interpolating between two closed Reeb orbits.  We denote the subcomplex involving Reeb orbits in the class $\Gamma$ by $C^{EGH}_*(M,\lambda, J, \Gamma)$.

The \textbf{differential} is given in terms of a weighted count of the elements of the moduli space of  \textbf{rigid cylinders} $\widehat{\mathcal{M}}^J_1(\gamma_+,\gamma_-)/\R$.  The weights arise because $\gp$ and $\gm$ may be multiply covered Reeb orbits, which means that $\widehat{\mathcal{M}}^J_1(\gamma_+,\gamma_-)/\R$ may consist of multiply covered curves

\begin{definition}[Multiplicities of orbits and curves]\label{multiplicities}\em
If $\widetilde{\ga}$ is a closed Reeb orbit, which is a $k$-fold cover of a simple orbit $\ga$, then the \textbf{multiplicity of the Reeb orbit} $\widetilde{\ga}$ is defined to be
$\mult(\widetilde{\ga})=k$
and $\mult(\ga)=1$.  The {multiplicity of a pseudoholomorphic curve} is 1 if it is somewhere injective.  If the pseudoholomorphic curve $u$ is multiply covered then it factors through a somewhere injective curve $v$ and a holomorphic covering $\varphi: (\R \times S^1,j_0) \to (\R \times S^1, j_0),$ e.g. $u = v \circ \varphi$. The \textbf{multiplicity of $u$} is defined to be $\mult(u):=\mbox{deg}(\varphi).$   If $u \in \mhat(\gp;\gm)$ then $m(\calc)$ divides both $\mult(\gp)$ and $\mult(\gm)$.  
\end{definition}


We define the operators 
\[
\begin{array}{ccll}
\kappa: & C^{EGH}_*(M,\lambda, J) &\to& C^{EGH}_{*}(M,\lambda, J)\\
&x& \mapsto &\mult(x)x \\
\end{array}
\]
and
\begin{equation}\label{defdelta}
\begin{array}{ccll}
\delta: & C^{EGH}_*(M,\lambda, J) &\to& C^{EGH}_{*-1}(M,\lambda, J)\\
&x& \mapsto &\displaystyle \sum_{\substack{ y \in \mathscr{P}_{\mbox{\tiny good}}(\lambda) \\ u \in \calm_1(x; y)/\R}}\dfrac{\epsilon(u)}{\mult(u)}y. \\
\end{array}
\end{equation}
The differentials are defined by
\begin{equation}\label{pa1}
\begin{array}{ccll}
\pa^{EGH}_{-}:=\kappa \circ \delta \ \colon & C^{EGH}_*(M,\lambda, J) &\to& C^{EGH}_{*-1}(M,\lambda, J) \\
&\gax &\mapsto& \displaystyle \sum_{\substack{y \in \mathscr{P}_{\mbox{\tiny good}}(\lambda) \\ u \in\calm_1(x; y)/\R}}\left( \epsilon(u)\frac{\mult({y})}{\mult(u)}\right) \up \\
\end{array}
\end{equation}
and
\begin{equation}\label{pa2}
\begin{array}{ccll}
\pa^{EGH}_{+}:= \delta \circ \kappa \ \colon & C^{EGH}_*(M,\lambda, J) &\to& C^{EGH}_{*-1}(M,\lambda, J) \\
&\gax &\mapsto& \displaystyle \sum_{\substack{y \in \mathscr{P}_{\mbox{\tiny good}}(\lambda) \\ u \in \calm_1(x; y)/\R}} \left( \epsilon(\calc)\frac{\mult({x})}{\mult(u)}\right) \up, \\
\end{array}
\end{equation}
where $\epsilon(\calc) = \pm1$ depends on a choice of \textbf{coherent orientations.}    Coherent orientations for symplectic field theory can be found in \cite{BM}, with additional exposition in \cite[\S A]{HN2}. A different choice of coherent orientations will lead to different signs in the differential, but the chain complexes will be canonically isomorphic.   

\begin{remark}[Well-definedness of the differentials]\em
In order to ensure that both of the expressions (\ref{pa1}) and (\ref{pa2}) are meaningful, i.e. that the counts of curves are finite, one must have proven that all moduli spaces of relevance can be cut out transversely.  
\end{remark}

\begin{remark}[Existence of Orientations]\label{goodbad1}\em
When the moduli space $\calm(x;y)/\R$  is a manifold, it can only be oriented by a choice of coherent orientations as in \cite{BM}, provided both $x$ and $y$ are good orbits.  
\end{remark}

\begin{remark}[Choices of coefficients]\label{coefficients}\em
The homologies, $H_*(C^{EGH}_*(M, \lambda, J), \pa^{EGH}_\pm)$ are equivalent over $\Q$-coefficients, provided sufficient transversality holds to define the chain complexes and obtain invariance.  The isomorphism between these two chain complexes is then given by $\kappa$ because $(\kappa \delta) \kappa = \kappa (\delta \kappa).$  As a result we denote
\[
CH_*^{EGH}(M, \lambda, J): = H_*(C^{EGH}_*(M, \lambda, J), \pa^{EGH}_\pm)
\]

While one can always define either differential for cylindrical contact homology over $\Z_2$ or $\Z$-coefficients because the weighted expression is always integral, one needs to work over $\Q$ in order to define the chain maps between the respective complexes $(C^{EGH}_*(M, \lambda, J), \pa^{EGH}_\pm)$.  In the case of dynamically separated contact forms $\lambda$ we have $\pa_-^{EGH} \equiv \pa_+^{EGH}$ because for any $u \in \calm_1(x;y)/\R$, $\mult(x)=\mult(y)$.  In this case the contact homologies are trivially isomorphic over $\Z_2$ and $\Z$-coefficients.
\end{remark}

\begin{remark}[Exclusion of bad Reeb orbits] \em
One must exclude bad Reeb orbits from the chain complex as their inclusion obstructs the proof of invariance, assuming sufficient transversality existed in the first place; see the period doubling example explained in \cite[\S 6.3]{HN2}.
\end{remark}


Cylindrical contact homology is well-defined for any primitive homotopy class $\Gamma \in \pi_0(\Omega M)$ and closed contact manifold $(M^{2n-1}, \xi)$ which is $\Gamma$-admissible.  It is also invariant under contactomorphisms in the following sense.  Here $CH_*^{EGH}(M,\lambda,J, \Gamma)$ represents the homology of subcomplex generated by the Reeb orbits in the free homotopy class $\Gamma$.

\begin{theorem}[Prop. 10.21, 10.24 \cite{wendl-sft}]
Let $M^{2n-1}$ be a closed manifold and $\Gamma \in \pi_0(\Omega M)$ be a primitive homotopy class of loops.  Then for a $\Gamma$-admissible contact form $\lambda$ and generic $\lambda$-compatible almost complex structure the operator $\delta$ in \eqref{defdelta} is well-defined and satisfies $\delta\kappa\delta=0$.  Suppose $\varphi: (M_0, \xi_0) \to (M_1,\xi_1)$ is a contactomorphism with $\varphi_*\Gamma_0 = \Gamma_1$, where $\Gamma_0$ is a primitive homotopy class of loops and $(M_1,\xi_1)$ is $\Gamma_1$-admissible.  Then $(M_0,\xi_0)$ is $\Gamma_0$-admissible, and 
 $CH_*^{EGH}(M_0, \xi_0, \Gamma_0) \cong  CH_*^{EGH}(M_1,\xi_1, \Gamma_1)$.

\end{theorem}


In \cite[Theorem 1.3]{HN1} we proved the following.

\begin{theorem}
Let $\lambda$ be a nondegenerate dynamically convex contact form on a closed 3-manifold $M$.   Suppose further that:
\begin{description}
\item{(*)}
A contractible Reeb orbit $\gamma$ has $\czm(\gamma)=3$ only if $\gamma$ is embedded.
\end{description}
Then for generic $\lambda$-compatible almost complex structures $J$ on $\R\times M$, the operator $\delta$ in \eqref{defdelta} is well-defined and satisfies $\delta\kappa\delta=0$, so that $(C^{EGH}_*(M,\lambda,J),\partial^{EGH}_\pm)$ is a well-defined chain complex.
\end{theorem}

In \cite{HN2} we establish invariance of cylindrical contact homology in the hypertight case.  This is achieved this by breaking the $S^1$-symmetry and using domain dependent almost complex structures, which necessitates the construction of nonequivariant contact homology $NCH_*(M,\xi; \Z)$ and a family Floer $S^1$-equivariant version of the nonequivaraint theory $CH_*^{S^1}(M,\xi; \Z)$.  We show that these theories do not depend on the choice of contact form or choice of $S^1$-dependent (resp. $S^1$-equivariant $S^1 \times ES^1$-dependent) family of almost complex structures.  More precisely we show the following.
\begin{theorem}[Theorem 1.6 \cite{HN2}]
Let $Y^{2n-1}$ be a closed manifold, and $\lambda$ and $\lambda'$ be nondegenerate hypertight contact forms on $Y$ with $\mbox{\em ker}(\lambda)=\mbox{\em ker}(\lambda')$. Let $\mathfrak{J}$ be a generic $S^1$-equivariant $S^1\times ES^1$-family of $\lambda$-compatible almost complex structures, and let $\mathfrak{J}'$ be a generic $S^1$-equivariant $S^1\times ES^1$-family of $\lambda'$-compatible almost complex structures. Then there is a canonical isomorphism
\[
CH_*^{S^1}(Y,\lambda,\mathfrak{J};\Z) = CH_*^{S^1}(Y,\lambda',\mathfrak{J}';\Z).
\]
\end{theorem}
This will be upgraded to allow for dynamically convex contact forms in dimension three in \cite{HN3}.

Next, suppose that $J$ is a $\lambda$-compatible almost complex structure on $\R\times Y$ which satisfies the transversality conditions needed to define cylindrical contact homology, see \cite[Def. 1.1]{HN2}. We show how to then compute the $S^1$-equivariant contact homology using  an automonomous family of almost complex structures. (In general, a slight perturbation of the autonomous family might be needed to obtain the transversality necessary to define the $S^1$-equivariant differential. See \cite[\S 5.2]{HN2} for details.) We then show that the $S^1$-equivariant theory, when tensored with $\Q$, is isomorphic to the cylindrical contact homology proposed by Eliashberg-Givental-Hofer, when the latter can be defined.

\begin{theorem}[Theorem 1.9 \cite{HN2}]
\label{thm:cchinv}
Let $Y$ be a closed manifold, let $\lambda$ be a nondegenerate hypertight contact form on $Y$, and write $\xi=\mbox{\em ker}(\lambda)$. Let $J$ be an almost complex structure on $\R\times Y$ which is admissible (see \emph{\cite[Def. 5.2]{HN2}}). Then there is a canonical isomorphism
\[
CH_*^{S^1}(Y,\xi;\Z)\otimes \Q = CH_*^{EGH}(Y,\lambda,J).
\]
\end{theorem}

\begin{corollary}
\label{cor:eghinv}
$CH_*^{EGH}$ is an invariant of closed contact manifolds $(Y,\xi)$ for which there exists a pair $(\lambda,J)$ where $\lambda$ is a nondegenerate hypertight contact form with $\mbox{\em ker}(\lambda)=\xi$, and $J$ is an admissible $\lambda$-compatible almost complex structure.
\end{corollary}

Again, we will upgrade these results to hold for dynamically convex contact forms in dimension three in \cite{HN3}.  In contrast to \cite{HN2, HN3}, this paper is concerned with the more restricted class of dynamically separated contact forms which allows us to directly obtain regularity for $S^1$-independent pseudoholomorphic cylinders in cobordisms.

\subsection{Filtered cylindrical contact homology}


The \textbf{action} of a Reeb orbit $\gamma$ is given by $ \mathcal{A}(\gamma):= \int_\gamma \lambda.$. Since $J$ is a $\lambda$-compatible almost complex structure on the symplectization it follows \cite[Lem. 2.18]{jo1} that the cylindrical contact homology differential(s) decreases the action, e.g. if $\langle \partial_\pm \gamma_+, \gamma_- \rangle \neq 0$ then $\mathcal{A}(\gamma_+) > \mathcal{A}(\gamma_-)$.

Thus, given any real number $L$ it makes sense to define the filtered cylindrical contact homology, denoted by $CH_*^{EGH, L}(M,\lambda,J)$, to be the homology of the subcomplex $C_*^{EGH, L}(M,\lambda,J)$ of the chain complex spanned by generators of action less than $L$.  The invariance of these filtered cylindrical contact homology groups is more subtle than in the unfiltered case, as they typically depend on the choice of contact form, cf. \cite[Thm 1.3]{HTarnold}.  We elucidate this point further.  

There are various natural maps defined on filtered cylindrical contact homology, which we will also explore from a computational perspective in  Section \ref{filteredchsec}.  First, if $L < L'$ there is a map
\begin{equation}\label{filteredinclusion}
\iota_J^{L,L'}: CH_*^{EGH, L}(M,\lambda,J) \to CH_*^{EGH, L}(M,\lambda,J)
\end{equation}
induced by the inclusion of chain complexes.   Given sufficient regularity, the cylindrical contact homology can be recovered from the filtered contact homology by taking the direct limit over $L$, 
 \begin{equation}\label{directlimiteq}
 CH_*^{EGH}(M,  \lambda, J) := \lim_{L \to \infty} CH_*^{{EGH, L}}(M,  \lambda, J).
 \end{equation}
In addition, if $c$ is a positive constant, then there is a canonical ``scaling" isomorphism
\begin{equation}\label{scalingeq}
s_J: CH_*^{{EGH, L}}(M,  \lambda, J) \overset{\simeq}\longrightarrow CH_*^{{EGH, cL}}(M,  c \lambda, J^c),
\end{equation}
where $J^c$ is defined to agree with $J$ when restricted to the contact planes $\xi$.  This is because the chain complexes on both sides have the same generators and the self-diffeomorphism of $\R \times M$ sending $(s,y) \mapsto (cs, y) $ induces a bijection between the $J$-holomorphic curves and $J^c$-holomorphic curves.  

To define $CH^{EGH, L}_*(M,\lambda,J)$ one does not need the full assumption that $\lambda$ is nondegenerate; the below weaker notion in conjunction with the $L$-dynamically separated or $L$-dynamically convex assumption will suffice.

\begin{definition}\em
The contact form $\lambda$ is \textbf{$L$-nondegenerate} if all Reeb orbits of action less than $L$ are nondegenerate and there is no Reeb orbit of action exactly $L$. An \textbf{$L$-nondegenerate dynamically separated} contact form is one which is both $L$-nondegenerate and $L$-dynamically separated. 
\end{definition}

If $\lambda$ is $L$-hypertight, but possibly degenerate, and if $\lambda$ does not have any Reeb orbit of action equal to $L$, then one can still define the filtered cylindrical contact homology, nonequivariant, or $S^1$-equivariant contact homology by letting $\lambda'$  be a small $L$-nondegenerate and $L$-hypertight perturbation of $\lambda$, see \cite[\S 1.6]{HN2}.  This does not depend on the choice of  $\lambda'$ if the perturbation is sufficiently small. With this definition, if $\lambda$ is hypertight but possibly degenerate, then we still have the direct limit \eqref{directlimiteq}.  We will mimic a similar construction for prequantization bundles in this paper. 

We obtain the following theorem, which asserts that under the dynamically separated assumption, filtered cylindrical contact homology and the various maps on it do not depend on $J$.   The proof is completed in Section \ref{invariance}. 

\begin{theorem}\label{invariance-thm1}
Let $M$ be a closed oriented connected $3$-manifold.
\begin{description}
\item[(a)] If $\lambda$ is an $L$-nondegenerate dynamically separated contact form on $M$ then $CH_*^{EGH, L}(M,\lambda,J)$ is well-defined and does not depend on the choice of generic $\lambda$-compatible almost complex structure, so we denote it by $CH_*^{EGH, L}(M,\lambda)$.  
\item[(b)] If $L<L'$ and if $\lambda$ is an $L'$-nondegenerate dynamically separated contact form on $M$, then the maps $\iota_J^{L,L'}$ in \eqref{filteredinclusion} induce a well-defined map
\[
\iota_J^{L,L'}: CH_*^{EGH, L}(M,\lambda) \to CH_*^{{EGH, L'}}(M,\lambda).
\]
\item[(c)] If $c>0$, then the scaling isomorphisms $s_J$ in \eqref{scalingeq} induce a well-defined isomorphism
\[
s_J: CH_*^{EGH, L}(M,  \lambda) \overset{\simeq}\longrightarrow CH_*^{EGH, cL}(M,  c \lambda)
\]
\end{description}

\end{theorem}

In Section \ref{filtered-continuation} we show that the filtered cylindrical contact homology and various maps on it do not depend on the choice of ``nearby" $L$-nondegenerate dynamically separated contact forms.

\subsection{Methods and applications for prequantization bundles}  
A motivating example of dynamically separated condition comes from the following perturbation of the canonical contact form on a prequantization bundle.

\begin{definition}[Prequantization] 
\em
Let $(\Sigma^{2n-2}, \omega)$ be a closed symplectic manifold such that the cohomology class $-[\om]/(2\pi) \in H^2(\Sigma;\R)$ is the image of an integral class $e\in H^2(\Sigma; \Z)$.  The principle $S^1$ bundle $\pi:V^{2n-1} \to \Sigma$ with first Chern class $e$ is the prequantization space.  The prequantization space $V$ admits a contact form which is the  real-valued connection 1-form $\lambda$  on V whose curvature is $\om$.
\end{definition}

\begin{remark} \em
In the above definition, $S^1$ acts freely on $V$ with quotient $\Sigma$ and the primary obstruction to finding a section $\Sigma \to V$ is $e\in H^2(\Sigma; \Z)$.  The derivative of the $S^1$ action, denoted $R$, is the vector field on $V$ tangent to the fibers.  Moreover $\lambda$ is invariant under the $S^1$ action, $\lambda(R)=1$, and $d\lambda =\pi^*\om$.  Thus $R$ is the Reeb vector field associated to $(V, \lambda)$ and the Reeb orbits are comprised of the fibers of this bundle, by design of period $2\pi$, and their iterates. 
\end{remark}

One can perturb the contact form $\lambda$ on $V$ via a lift of a Morse-Smale\footnote{We make a slight abuse terminology here, saying that $H$ Morse-Smale instead of saying that the pair $(H, g=\omega(\cdot, J\cdot) )$ is Morse-Smale. } function $H$ which is $C^2$ close to  1 on the base $\Sigma$,
\begin{equation}
\label{perturbedform1}
\lep=(1+\vepsilon \pi^*H)\lambda.
\end{equation}
The cylindrical contact homology can then be expressed in terms of the Morse homology of the base.  Details of similar constructions have previously appeared in work of Bourgeois \cite{B02} and Vaugon \cite[\S 6]{annethesis}. We define the contact form


\begin{lemma}
The Reeb vector field of $\lep$ is given by
\begin{equation}
\label{perturbedreeb1}
R_{\vepsilon}=\frac{R}{1+\vepsilon \pi^*H} + \frac{\veps \widetilde{X}_H}{{(1+\vepsilon \pi^*H)}^{2}},
\end{equation}
where $X_{H}$ is the Hamiltonian vector field\footnote{We use the convention $\omega(X_H, \cdot) = dH.$} on $\Sigma$ and $\widetilde{X}_{ H}$ is its horizontal lift.
\end{lemma}



We have the following formula for the Conley-Zehnder indices of iterates of orbits which project to critical points $p$ of $H$.  We denote the $k$-fold iterate of an orbit which projects to $p \in \mbox{Crit}(H)$ by $\gamma_p^k$.

\begin{lemma}\label{lempre}
Fix $L>0$ and $H$ a Morse-Smale function on $\Sigma$ which is $C^2$ close to 1.  Then there exists $\vepsilon >0$ such that all periodic orbits $\gamma$ of $R_\vepsilon$ with action $\mathcal{A}(\gamma) <L$ are nondegenerate and project to critical points of $H$.  The Conley-Zehnder index such a Reeb orbit over $p \in \mbox{\em Crit}(H)$ is given by
\[
\czm^\Phi(\gamma_p^k) = \mu^\Phi_{RS}(\gamma^k) -n + \mbox{\em index}_pH,
\]
where $\mu^\Phi_{RS}(\gamma^k)$ is the Robbin-Salamon index of the $k$-fold iterate of the fiber $\gamma = \pi^{-1}(p)$.  
\end{lemma}

There is a well known relation between the Maslov index of the fiber $\gamma$ and the Chern number of the base $(\Sigma, \omega)$, for example in \cite{vknotes}.  If $(\Sigma,\omega)$ is the standard $(S^2, \omega_0)$ where $\int_{S^2}\omega_0 =4 \pi$ we have the following result.


\begin{proposition}\label{spheremaslov}
Let $(V,\lambda)$ be the prequantization bundle over the closed symplectic manifold $(S^2, k \omega_0)$ for $k \in \Z_{>0}$.  Then $(V,\xi) = (L(k,1),\xi_{std}) $ and the $k$-fold cover of every simple orbit $\gamma$ is contractible and $\mu_{RS}^\Phi(\gamma^k) = 4$.  
\end{proposition}

These results are proven in Section \ref{cz-section}, permitting us to conclude that the contact form $\lep$ associated to any prequantization bundle over $(S^2, k\omega_0)$ is dynamically separated up to large action. We obtain a natural filtration on both the action and the SFT-grading of Reeb orbits associated to $R_\vepsilon$.  We investigate this double filtration in Section \ref{filtration}, yielding the following Morse-Bott computational result.

\begin{proposition}
Under the assumptions of Lemma \ref{lempre}, for generic $\lambda_\vepsilon$-compatible $J_\vepsilon$ and with respect to each free homotopy class $\Gamma$, the filtered cylindrical contact homology $CH^{EGH,L_\vepsilon}_*(V,\lambda_\vepsilon, J_\vepsilon, \Gamma)$ consists of copies of  $H_*^{\mbox{\tiny Morse}}(\Sigma, H; \Q)$ with $\partial_\pm^{EGH} = \partial_H^{\mbox{\tiny Morse}}$ on each copy.  
\end{proposition}



The use of direct limits in conjunction with the above geometric perturbation allows us avoid the analytic difficulties of directly degenerating moduli spaces of pseudoholomorphic cylinders.    

\begin{theorem}\label{prequantch}
Let $(V,\ker \lambda)$ be a prequantization bundle over an integral closed symplectic surface $(\Sigma^2, \omega)$.  Then with respect to each free homotopy class $\Gamma$, $CH^{EGH}_*(V,\ker \lambda, \Gamma)$ consists of an infinite number of appropriately SFT-grading shifted copies of the singular homology of the base.   
\end{theorem}





\begin{remark}[Applicability to higher genus surfaces]\em
Prequantization bundles over closed Riemann surfaces $\Sigma_g$ with $g\geq1$ are not  dynamically separated as there does not exist a global trivialization of $\xi$.  However, there exist local constant trivializations which are sufficient to define and compute cylindrical contact homology as in Theorem \ref{prequantch}.   This is due to the
 the absence of contractible orbits, that the multiplicity of the orbit determines its free homotopy class, and the existence of trivializations which guarantee regularity of the relevant unbranched covers of low index cylinders.  This is explained in Section \ref{cz-section}.
 \end{remark}

The following remarks detail applications of the above Morse-Bott methods for prequantization bundles over closed oriented surfaces.   {These applications require more robust invariance results than obtained in this paper, such as those in \cite{HN2} in the hypertight case or the forthcoming joint work with Hutchings \cite{HN3} for the three dimensional dynamically convex case.  The abstract perturbation methods under development by Hofer, Wysocki, and Zehnder, together with Fish and Wehrheim are also expected to suffice.}

\begin{remark}[Growth Rates]\em
In conjunction with Vaugon's work \cite{annegrowth}, we expect the above methods to permit us to prove growth results  for the cylindrical contact homology of prequantization bundles over closed oriented surfaces.  The growth rates should depend on the Euler characteristic of the base and the Euler number of the fibration.
\end{remark}

\begin{remark}[Refinements of the Conley Conjecture]\em
Ginzburg, G\"urel, and Macarini explain in \cite[\S 6]{ggm2} how one could use cylindrical contact homology in conjunction with Morse-Bott methods to refine \cite[Theorem 2.1]{ggm2}.  This would give more precise lower bounds on the number of geometrically distinct contractible (non-hyperbolic) periodic Reeb orbits of prequantization bundles.  Another application is a refinement of the Conley Conjecture \cite[Theorem 2.1]{ggm1}, which under certain assumptions (cf. \S 4.2-4.3) guarantees that for every sufficiently large prime $k$, the Reeb flow has a simple closed orbit in the $k$-th iterate of the free homotopy class of the fiber.   We expect that the methods of this paper in conjunction with the stronger invariance results of \cite{HN2, HN3} permit these extensions for prequantization bundles $(V^3,\xi)$ over closed oriented surfaces $(\Sigma^2, \omega)$.  {In their work, Ginzburg, G\"urel, and Macarini previously analyzed $S^1$-equivariant symplectic homology to rigorously extract dynamical information of Reeb flows associated to prequantization bundles.}
\end{remark}

\begin{remark}[Hope for higher dimensions] \em
Recent work by Wendl \cite{wendlmult} establishes transversality for certain multiply covered closed curves in higher dimensions.  Given that there is no obvious obstruction to applying the same techniques to study punctured curves in symplectic
cobordisms we expect that Wendl's methods combined with those used to prove Theorem \ref{prequantch} can be generalized to apply to prequantization spaces over higher dimensional monotone symplectic manifolds.  
\end{remark}

\subsection{Examples}\label{examples-intro}

We conclude this section with some examples.

\begin{example}[3-sphere]\label{prequant}\em
The contact 3-sphere $(S^3, \xi_{std}=\ker \lambda_0)$ can be realized as a prequantization space via the Hopf fibration $\hopf $ over $(S^2, \omega_0)$,

\[
h(u,v)=(2u \bar{v}, |u|^2- |v|^2), \ (u,v) \in S^3 \subset \bbC^2.
\]

Let $H$ be a Morse-Smale function on $S^2$ and $\lep$ as in (\ref{perturbedform1}).  The only fibers that remain Reeb orbits associated to $\lep$ are iterates of fibers over the critical points $p$ of $H$.  For sufficiently small $\vepsilon$ the surviving $k$-fold covers of simple orbits in the fiber, denoted by $\gamma_p^k$, have action $L \lesssim 1/\vepsilon$, are non-degenerate, and satisfy
\begin{equation}
\label{czprequant1}
\czm(\gpk)=4k-1+\mbox{{index}}_p(H).
\end{equation}

 \begin{figure}
 \begin{center}
    \includegraphics[width=0.63\textwidth]{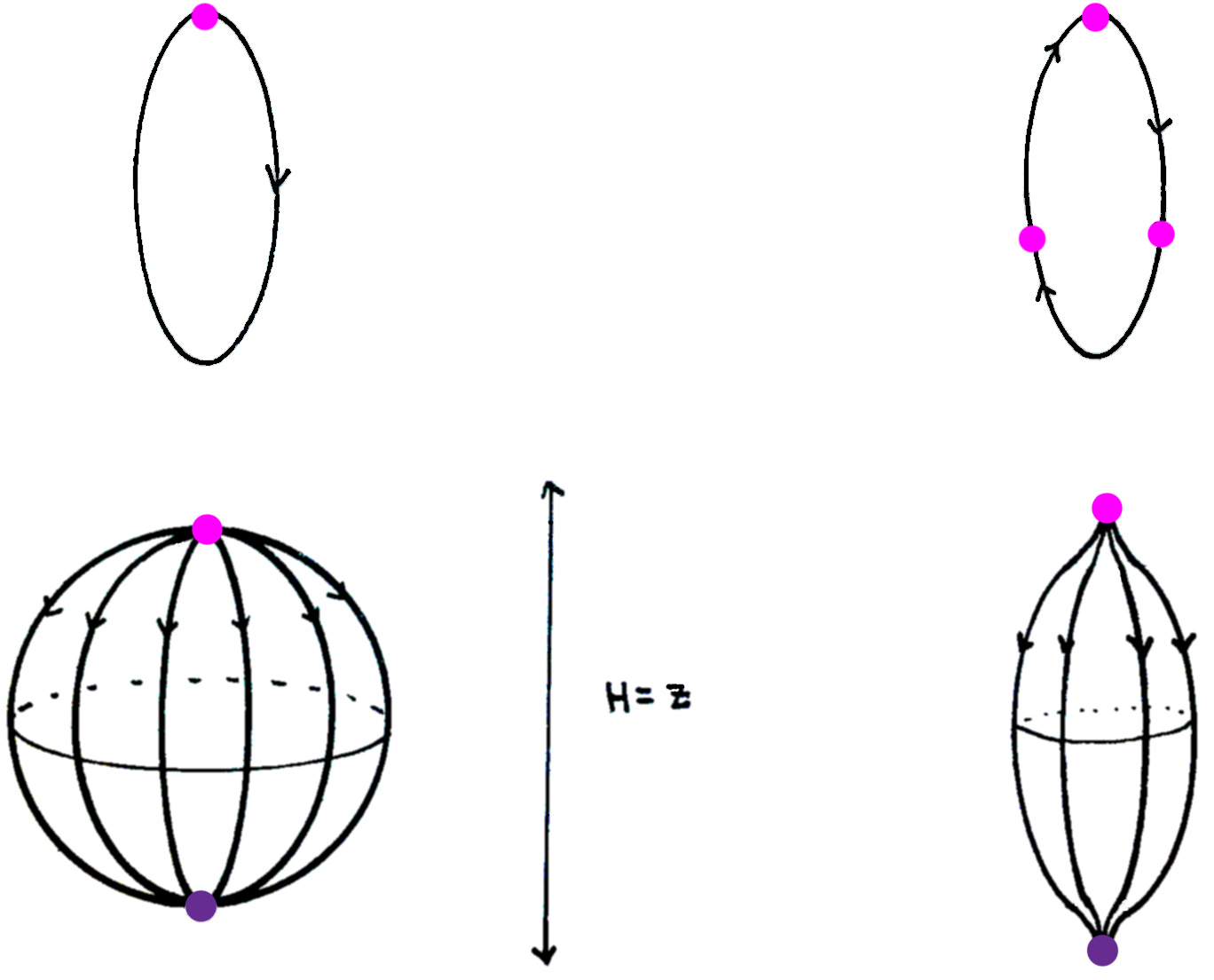}
\end{center}
  \caption{ $-\nabla H$ for $H=z$  with a fiber over $S^2$ and $S^2/\Z_3$ respectively.}
  \label{figlens}
\end{figure}

If we take $H=z$, the height function on $S^2$ as in Figure \ref{figlens} then we obtain a maximum at the north pole (index 2) and a minimum at the south pole (index 0).  Because the index increases by 4 under iteration, we have that $\czm(\gpk)$ in (\ref{czprequant1}) is always odd, so the differential vanishes, resulting in the following theorem.  
\end{example}
\begin{theorem}\label{HCsphere}
The cylindrical contact homology for the sphere $(S^{3},\xi_{std})$ is given by
\[
CH_*(S^{3}, \xi_{std};\Q) = \left\{  \begin{array}{cl}
   \Q & *  \geq 2, \mbox{ even }  \\
    0 & * \ \mbox{ else } \\
\end{array} \right.
\]
\end{theorem}

We similarly obtain the following result for the lens space $L(n+1,1)$, equipped with the standard contact structure induced from the standard one on $S^3$.

\begin{theorem}
The cylindrical contact homology for the lens space $(L(n+1,1), \xi_{std})$ is given by
\[
CH_*(\lens, \xi_{std};\Q) = \left\{  \begin{array}{cl}
    \Q^n & *  =0   \\
    \Q^{n+1} & * \geq 2, \mbox{ even } \\
       0 & * \ \mbox{ else } \\
\end{array} \right.
\]
\end{theorem}

We are able to adapt these methods to compute cylindrical contact homology of $(L(n+1,n),\xi_{std})$ as follows.  It is interesting to note that cylindrical contact homology groups alone cannot distinguish $(L(n+1,n),\xi_{std})$ from $(L(n+1,1),\xi_{std})$.  However, the classical first Chern class is capable of distinguishing them.    

\begin{example}[$(L(n+1,n),\xi_{std})$]\label{lensex}
\em
If  $\pi_1(M)$ is abelian then the $k_i(\baar, \ga)$ form an arithmetic progression because 
\[
\pi_0(\Omega M) = \pi_1(M) / \{ \mbox{conjugacy} \} \cong \pi_1(M).
\]
This applies to the lens space $(L(n+1,n), \xi_{std})$, as each free homotopy class $\baar$ may be represented as an element of $\{ 0, 1, ... n\}$, where 0 represents a contractible class.  As a result, an arbitrary cover of a closed orbit may not be of the same free homotopy class $\baar$.  This will only be the case when the $k_\ell(\baar, \ga)$-th cover is given by 
\[
k_\ell(\baar, \ga) = \ell (n+1)+\baar,\mbox{ for } \baar \neq 0 \mbox{ and } \ell \in \Z_{\geq 0}.
\]
The procedure described in the previous example holds, though some care must be taken in regards to the fact that the base is now a symplectic orbifold.

We note that the Lens spaces $(L(n+1, n), \xi_{std})$ are contactomorphic to the links of the $A_n$ singularities $(L_{A_n},\xi_{A_n})$, with
\[
L_{A_n}:=\{\mathbf{z} \in \bbC^3 \ | \ z_0^{n+1} + z_1^2+z_2^2 =0 \}\cap S^{5} 
\]
and the canonical contact structure given by
\[
\xi_{A_n}:= T( L_{A_n} )\cap J_0 T( L_{A_n}).
\]
As $(L_{A_n},\xi_{A_n})$ is an example of a Brieskorn manifold, it is well known that $c_1(\xi_{A_n})=0$ \cite[\S 2]{vk}, thus $c_1(\xi_{L(n+1,n)})=0.$  The quotient of $S^3$ with the following cyclic subgroup of $\mbox{SU}_2(\bbC)$ yields the Lens space $L(n+1,n)$.  This cyclic subgroup is $\Z_{n+1}$, which acts on $\bbC^2$ by $u \mapsto \varepsilon u, \  v \mapsto \varepsilon^{-1}v$, where $\varepsilon = e^{2\pi i/(n+1)}$, a primitive $(n+1)$-th order root of unity.   The complex volume form $du \wedge dv $ on $\bbC^2$ can be used to compute the Conley-Zehnder indices associated to Reeb orbits of $S^3$ without local trivializations.  Since $\Z_{n+1} \subset \mbox{SU}_2(\mathbb{C})$, this means that the complex volume form $du \wedge dv $ descends to the quotient, allowing one to compute the Conley-Zehnder indices associated to Reeb orbits of $L(n+1,n)$.  This procedure yields the following formulas for the Conley-Zehnder indices. 

Let $\ga_p$ be the underlying simple orbit over a critical point $p$ of $H$.  For every $\ell \in  \Z_{> 0}$, we obtain a contractible orbit $\ga_p^{\ell(n+1)}$ of index
\begin{equation}
\label{cznoncontractible}
\czm(\ga_p^{\ell(n+1)})=4\ell- 1+\mbox{{index}}_p(H)
\end{equation}

Otherwise for every $\ell \in  \Z_{\geq 0}$ we obtain a noncontractible Reeb orbit $\ga_p^{\ell(n+1)+\Gamma}$ in the free homotopy class $\Gamma \in \{1, 2, ..., n\}$ of index
\begin{equation}
\label{czcontractible}
\begin{array}{lcl}
\czm(\ga_p^{\ell(n+1)+\Gamma})&=&2+ 4 \left \lfloor \frac{\ell(n+1)+\Gamma}{n+1}  \right \rfloor - 1+\mbox{\emph{index}}_p(H) \\
&=& 2 + 4 \ell - 1 +\mbox{{index}}_p(H), \\
\end{array}
\end{equation}

When using the height function as in Figure \ref{figlens} the differential vanishes in light of (\ref{cznoncontractible}), yielding the following theorem.

\begin{theorem}\label{HClens}
The cylindrical contact homology for the lens space $(\lens, \xi_{std})$ is given by
\[
CH_*(\lens, \xi_{std};\Q) = \left\{  \begin{array}{cl}
    \Q^n & *  =0   \\
    \Q^{n+1} & * \geq 2, \mbox{ even } \\
       0 & * \ \mbox{ else } \\
\end{array} \right.
\]
\end{theorem}

\end{example}

The lens space $(L(n+1,n),\xi_{std})$ is contactomorphic to the link of the $A_n$ singularity \cite[Theorem 1.8]{reu}, and our computation agrees with \cite[Theorem 1.5]{reu}.  Thus an alternate interpretation of Theorem \ref{HClens} is that cylindrical contact homology of the link of the $A_n$ singularity is a free $\Q[u]$ module of rank equal to the number of conjugacy classes of the finite subgroup $A_n$ of SL$(2;\C)$.

 In future work, we will generalize Theorem \ref{prequantch} so that one can work with prequantization bundles over symplectic orbifolds.  In this setting, the contact homology differential should agree with the Morse orbifold differential.  This would allow us to compute cylindrical contact homology of many Seifert fiber spaces and many three dimensional links of weighted homogeneous polynomials.  When the defining polynomial is homogeneous the link can be realized as a prequantization bundle over a symplectic manifold. This generalization yields a Floer theoretic interpretation of the McKay correspondence in terms of the Reeb dynamics of the links of the simple singularities.   This agrees with  work by McLean and Ritter \cite{MR} which establishes a relationship between the cohomological McKay correspondence and symplectic homology.  Thus we expect that the cylindrical contact homology $CH_*^{EGH}(S^3/\Gamma, \xi_0)$  is a free $\Q[u]$ module of rank equal to the number of conjugacy classes of the finite subgroup $\Gamma$ of $\mbox{\em SL}(2;\C)$. 



\section{Pseudoholomorphic preliminaries} \label{background}
The chain map and chain homotopy will be defined via counts of elements of moduli spaces of cylinders in Section \ref{invariance}. However, we still need to consider moduli spaces of finite energy genus 0 curves with one positive and an arbitrary number of negative ends asymptotic to Reeb orbits.  Section \ref{jback} reviews the necessary background of finite energy genus 0 curves with an arbitrary number of negative punctures. Section \ref{cz-sec} reviews some facts about the Conley-Zehnder index of Reeb orbits associated to contact 3-manifolds.  Section \ref{reg-sec} shows that under the dynamically separated assumption, index -1 and 0 cylinders are regular, a key component in proving that the chain map and chain homotopy are well-defined.  

\subsection{The letter $J$ is for pseudoholomorphic}\label{jback}
Let $(\overline{W}, \overline{\lambda})$ be a compact, connected, exact symplectic manifold such that
\[
\partial \overline{W} = M_+ - M_-,
\]
and $\lambda_\pm =\lambda \arrowvert_{M_\pm}$ is a contact form on $M_\pm$.  Define $(W,\lambda)$ to be the completion of $(\overline{W}, d\overline{\lambda})$ by

\[
W = (-\infty, 0] \times M_- \sqcup \overline{W} \sqcup [0,\infty) \times M_+.
\]
Let $J$ be an almost complex structure which is $d\lambda$-compatible on $\overline{W}$ as well as $\lambda_\pm$-compatible on the symplectization ends of $W$.  The pair $(W, \lambda)$ is called an \textbf{exact symplectic cobordism}. 

 An almost complex structure $J$ on $W$ is said to be \textbf{cobordism compatible} if 
 \begin{itemize}
\item $J$ agrees on $[0, \infty) \times M_+$ with the restriction of a $\lambda_+$-compatible almost complex structure $J_+$ on $\R \times Y_+$; 
\item $J$ agrees on $(-\infty, 0] \times M_-$ with the restriction of a $\lambda_-$-compatible almost complex structure $J_-$ on $\R \times Y_-$;
\item $J$ is compatible with the symplectic form $d\overline{\lambda}$ on $\overline{X}$. 
\end{itemize}
Throughout we will assume that all cobordisms are exact.

When constructing the chain homotopy we need to consider a one parameter family of 1-forms $\{ \overline{\lambda}_\tau \}_{\tau \in[0,1]}$ on $\overline{W}$ such that $d\overline{\lambda}_\tau$ is symplectic and $\overline{\lambda}_\tau\arrowvert_{M_\pm}=\lambda_\pm$ for all $\tau \in [0,1]$.   For each $\tau \in [0,1]$, let $(W,\lambda^\tau)$ be the completion of $(\overline{W},\overline{\lambda}^\tau)$ and let $\mathbb{J} =\{J_\tau\}$ be a  1-parameter smooth family of almost complex structures which is cobordism compatible for each $\tau \in [0,1]$.

\textbf{Asymptotically cylindrical curves} are equivalent to \textbf{finite (Hofer) energy curves} and defined as follows.  Let $(\Sigma, j)$  be a closed Riemann surface and $\Gamma$ be a set of points which are the punctures of $\dot{\Sigma}:=\Sigma \setminus \Gamma$.   Asymptotically cylindrical maps are pseudoholomorphic maps 
\[
u: (\dot{\Sigma}, j) \to (W, J),
\]
subject to the asymptotic condition (\ref{ass1}).  The domain of all the curves of interest in this paper is a multiply punctured sphere $(\dot{\Sigma}, j):=(S^2\setminus \{x,y_1,...,y_{s}\}, j_0)$.

After partitioning the punctures into positive and negative subsets wherein $\g_+:= \{x\}$ and $\g_-:=\{y_1,..y_s\} $,  we consider asymptotically cylindrical $J$-holomorphic curves which are assumed to have the property that for each $z \in \g_\pm$, there exist holomorphic cylindrical coordinates identifying a punctured neighborhood of $z$ with a respective positive half-cylinder $Z_+=[0, \infty) \times S^1$ or negative half-cylinder $Z_-=(-\infty, 0] \times S^1$ 
and a trivial cylinder $u_{\gamma_z}: \R \times S^1 \to \R \times M$ such that
\begin{equation}\label{ass1}
u(s,t) = \exp_{u_{\gamma_z}(s,t)}h_z(s,t) \mbox{ for $|s|$ sufficiently large,}  
\end{equation}\label{ass2}
where $h_z(s,t)$ is a vector field along $u_{\gamma_z}$ satisfying $|h_z(s,\cdot)| \to 0$ uniformly as $s \to \pm \infty$.  Both the norm and the exponential map are assumed to be defined with respect to a translation-invariant choice of Riemannian metric on $\R \times M$.

The moduli space of asymptotically cylindrical curves is the space of equivalence classes of asymptotically cylindrical pseudoholomorphic maps; here an equivalence class is defined by the data $(\Sigma, j, \Gamma, u)$, where $\Gamma$ is an ordered set.  An equivalence class $(\Sigma, j, \g, u) \sim (\Sigma', j', \g', u')$ of asymptotically cylindrical pseudoholomorphic maps, $[(\Sigma, j, \g, u)]$, is determined whenever there exists a biholomorphism $\phi: (\Sigma, j) \to (\Sigma', j')$ taking $\g$ to $\g'$ with the ordering preserved, i.e. $\phi(\g_+) =\g'_+$ and $\phi(\g_-) =\g'_-$, such that $u = u' \circ \phi.$

We denote the moduli space of genus 0 \textbf{asymptotically cylindrical pseudoholomorphic} curves with 1 positive end and $s$ negative ends limiting on the Reeb orbits $\gp, \ga_1,...,\ga_s$ by
\[
\calm(\gp;\ga_1,...,\ga_s).
\]
We also are interested in genus 0 \textbf{finite energy planes\footnote{Note that $(S^2 \setminus \{x\}, j_0)$ is biholomorphic to $(\bbC, j_0)$, hence the terminology plane.}}, which are pseudoholomorphic curves
\[
u: (S^2 \setminus \{x\}, j_0) \to (W, \Jt)
\]
asymptotically cylindrical to a single nondegenerate Reeb orbit $\ga$ at the puncture $x$. When  $W=\R \times M$ and $\Jt$ is $\R$-invariant then $\R$ acts on these moduli spaces by \textbf{external translations}
\[
u=(a,f) \to (a + \rho, f),
\]
and we denote the quotient by 
\[
\mhat(\gamma;\gamma_1,...\gamma_s):=\calm(\gamma;\gamma_1,...\gamma_s)/\R.
 \]


{When $W= \R \times M$ and $J$ is $\lambda$-compatible, then the maximum principle implies that the puncture of a finite energy plane is always positive.  In a non-$\R$ invariant exact symplectic cobordism, Stokes' theorem can be used to obtain that the puncture is always positive, because the energy is positive.  For further details see \cite[\S 10]{wendl-sft}.}


\begin{definition}\em
An asymptotically cylindrical pseudoholomorphic curve 
\[
u: (\ds:=\Sigma \setminus ( \g_+ \sqcup \g_- ), j) \to (W, J)
\]
is said to be \textbf{multiply covered} whenever there exists a pseudoholomorphic curve
\[
v: (\ds':=\Sigma' \setminus ( \g'_+ \sqcup \g'_-), j') \to (W, J),
\]
 and a holomorphic branched covering $\varphi: (\Sigma, j) \to (\Sigma', j') $ with $\g'_+ = \varphi(\g_+)$ and  $\g'_- = \varphi(\g_-)$ such that
\[
u = v \circ \varphi, \ \ \ \mbox{deg}(\varphi)>1,
\]
allowing for $\varphi$ to not have any branch points.  Recall that $\mult(u) := \mbox{deg}(\varphi).$ 
\end{definition}

An asymptotically cylindrical pseudoholomorphic curve \ $u$ is called \textbf{simple} whenever it is not multiply covered.  In \cite[\S 3.2]{jo1} we gave a proof of the folk theorem that that every simple asymptotically cylindrical curve  is \textbf{somewhere injective}, meaning for some $z \in \ds$,
\[
du(z) \neq 0 \ \  \ u^{-1}(u(z))=\{z\}.
\]
A point $z \in \ds$ with this property is called an \textbf{injective point} of $u$.  

An \textbf{immersed} pseudoholomorphic curve (with one positive puncture) is an equivalence class of tuples $(\Sigma, j, \Gamma, u)$ such that $u$ is an immersion.

If the asymptotic orbits of a curve $u \in \calm(\ga; \ga_1,...,\ga_s)$ are all nondegenerate, then the \textbf{virtual dimension} of $\calm(\ga; \ga_1,...,\ga_s)$ is equal to the \textbf{index},  which is given by
\begin{equation}\label{w-index}
{\ind}(u) = -\chi(\ds)   +  \mu_{CZ}^\Phi(\gamma) - \displaystyle \sum_{i=1}^s\mu_{CZ}^\Phi(\gamma_i) + 2c_1^\Phi(u^*TW, J),
\end{equation}
as in \cite{Wtrans}, with $\chi(\ds) =(2-2g(\Sigma) - \#\g^+ - \#\g^-)$ and $\Phi$ a trivialization of $\xi$ along the asymptotic orbits of $u$.   In particular, $c_1^\Phi(u^*TW, J)$ is the relative first Chern number of $(u^*TW, J) \to \ds $ with respect to a suitable choice of $\Phi$ along the ends and boundary.  Moreover, the relative first Chern class vanishes when the trivialization $\Phi$ extends to a trivialization of $u^*\xi$.   
\begin{remark}\label{trivchoice}\em 
When $g=0$ we can always choose a trivialization $\Phi$ (fixed up to homotopy) such that $c_1^\Phi(u^*TW, J)=0$. 
More precisely, we choose a trivialization $\Phi$ so that $c_1^\Phi(v^*TW, J)=0$ for a somewhere injective curve genus 0 asymptotically cylindrical curve $v$ with one positive puncture and at least one negative puncture.  This implies for any (branched) cover  $u := \varphi \circ v$,  that $c_1^\Phi(u^*TW, J)=0$. Without loss of generality we can work with the following index formula

   \begin{equation}\label{indexformula}
 {\ind}(u) = -(1-s) + \czm^\Phi(\gamma) - \displaystyle \sum_{i=1}^s\czm^\Phi(\gamma_i).
 \end{equation}
 \end{remark}

 If $u$ is a non-constant curve then the action of $\mbox{Aut}(\ds,j)$ induces a natural inclusion of its Lie algebra $\mathfrak{aut}(\dot{\Sigma},j)$ into $\ker D\deebar(u).$  In \cite[\S7]{wendl-sft} Wendl provides a complete proof that the moduli spaces of somewhere injective curves are cut out transversely for generic choice of $J$ in a cobordism $(W,J)$.  When $J$ is $\R$-invariant and $W=\R \times M$, the ``standard'' argument must be modified; see \cite[\S 8]{wendl-sft}. 
 
Next we recall some transversality theorems in cobordisms.  In the statements of these theorems we suppress the notation specifying Reeb orbits and denote $\calm$ to be a moduli space of asymptotically cylindrical curves.  

\begin{theorem}{\em \cite[Theorem 8.1]{wendl-sft}} \label{si-thm1}
Let $\J$ be the set of all $\lambda$-compatible almost complex structures on $\R \times M$ where $\lambda$ is nondegenerate. Then there exists a comeager subset
$\J_{reg} \subset \J, $ such that for every $J \in \J_{reg}$, every curve $u \in \calm$ with a representative $u \colon (\dot{\Sigma},j) \to (\R \times M, \Jt)$ that has an injective point $z\in \dot{\Sigma}$ satisfying $ \pi_\xi \circ du(z)  \neq 0 $ is Fredholm regular.
\end{theorem}

The above result also holds for the set of somewhere injective curves in completed exact symplectic cobordisms $(W, J)$; see \cite[Theorem 7.2]{wendl-sft}.  
\begin{theorem}\label{si-thm2}
Let $\lambda_\pm$ be nondegenerate contact forms on a closed manifold $M$ and $J$ be a generic cobordism compatible almost complex structure.  Then every somewhere injective curve $u \in \calm$ is Fredholm regular. 
\end{theorem}

Moreover, we have that Fredholm regularity implies that a neighborhood of a curve admits the structure of a smooth orbifold.

\begin{theorem}[Theorem 0, \cite{Wtrans}]\label{folk0}
Assume that $u\colon (\dot{\Sigma},j) \to (W, \Jt)$ is a non-constant curve in $\calm(\ga; \ga_1,...,\ga_s)$ asymptotic to nondegenerate orbits.  If $u$ is regular, then a neighborhood of $u$ in $\calm(\ga, \ga_1,...,\ga_s)$ naturally admits the structure of a smooth orbifold of dimension 
\[
{\mbox{\em ind}}(u) =  -(1-s) + \czm^\Phi(\gamma) - \displaystyle \sum_{i=1}^s\czm^\Phi(\gamma_i),
\] 
whose isotropy group at $u$ is given by
\[
\mbox{\em Aut}:=\{ \varphi \in \mbox{\em Aut}(\dot{\Sigma}, j) \ | \ u = u \circ \varphi \}.
\] 
Moreover, there is a natural isomorphism
\[
T_u\calm(\ga; \ga_1,...,\ga_s) = \ker D\deebar(j,u)/\mathfrak{aut}(\dot{\Sigma},j).
\]
\end{theorem}

\begin{remark}\label{par-reg}\em
The above results can be extended to include moduli spaces dependent on finitely many parameters, necessary in establishing the chain homotopy.  Let $P$ be a smooth finite-dimensional manifold and let $\{J_\tau\}_{\tau \in P}$ be a smooth family of complex structures.  A \textbf{parametric moduli space} is defined by
\[
\widehat{\mathcal{M}}(\{ J_\tau \}_{\tau \in P}) = \{ (\tau,u) \ \arrowvert \ \tau \in P, \ u\in \M(J_\tau) \}.
\] 
An analogous notion of \textbf{parametric regularity} holds for pairs $(\tau, u) \in \calm(\{ J_\tau \}_{\tau \in P})$, which is an open condition such that the space $
\widehat{\mathcal{M}}{\mbox{\tiny reg}}(\{ J_\tau \}_{\tau \in P})$
of parametrically regular elements will be an orbifold of dimension
\[
\mbox{dim }\widehat{\mathcal{M}}{\mbox{\tiny reg}}(\{ J_\tau \}_{\tau \in P}) = \mbox{vdim}\calm + \mbox{dim}(P).
\]
\end{remark}

In particular, we obtain the following parametric regularity result.  The proof follows by modifying the full details given in the closed case in \cite[\S 4.5]{Wnotes} to the set up for the punctured case in \cite[\S 7]{wendl-sft}.  Full details in the Hamiltonian Floer setting are given in \cite[\S 11.3]{ADfloer}.

\begin{theorem}\label{thmparreg}
Let $\lambda_\pm$ be nondegenerate contact forms on a closed, connected manifold $M$ and suppose the smooth family of cobordism compatible almost complex structures $\{ J_\tau \}_{\tau \in P}$ {is generic} and varies on an open subset $\mathcal{U}$ in the complement of the cylindrical ends of $(W,J)$ for $\tau$ lying in some precompact open subset $\mathcal{V} \subset {P}$.  Then all elements $(\tau,u) \in \widehat{\mathcal{M}}(\{ J_\tau \}_{\tau \in P})$ for which $\tau \in \mathcal{V}$ and $u$ has an injective point mapping to $\mathcal{U}$ are parametrically regular.
 \end{theorem}
 
 \begin{remark}\em
In this paper we will take $P=[0,1]$ and $\mathcal{V}=(0,1)$ so that we can consider generic homotopies of almost complex structures.  Note that regularity in the sense of Theorem \ref{si-thm2} always implies parametric regularity, while the converse is false.  However, when automatic transversality holds, one can guarantee regularity for all $J_\tau$ with no need for genericity.  
 \end{remark}


\subsection{The Conley-Zehnder index in dimension 3}\label{cz-sec}

In this section we review properties of the Conley-Zehnder index of a Reeb orbit $\gamma$, with respect to an appropriate (local) trivialization $\Phi$, in an arbitrary 3-dimensional nondegenerate contact manifold $(M, \lambda)$.  First, pick a parametrization $\gamma: \R / T\Z \to M$.  Let $\{ \varphi_t\}_{t\in \R}$ denote the one-parameter group of diffeomorphisms defined by the flow of the Reeb vector field $R$.  The linearized flow
\[
d\varphi_t : T_{\gamma(0)}M \to T_{\gamma(t)}M
\]
induces a symplectic linear map
\[
\phi_t : \xi_{\gamma(0)} \to \xi_{\gamma(t)},
\]
which can be realized as a $2 \times 2$ symplectic linear map via the trivialization $\Phi$.  We have that $\phi_0=1$ and $\phi_T$ is the linearized return map with respect to our trivialization.  

We define and compute the Conley-Zehnder index $\czm^\Phi(\gamma):=\czm^\Phi\left(\{ \phi_t \}_{t\in[0,T]}\right) \in \Z$ via the family of symplectic matrices $\{ \phi_t \}_{t\in[0,T]}$ as follows.


 \begin{list}{\labelitemi}{\leftmargin=2em }
\item[\textbf{Elliptic case:}]

In the elliptic case there is a special trivialization that one can pick so that the linearized flow $\{\phi_t\}$ can be realized as a path of rotations.  If we take $\Phi$ to this trivialization so that each $\phi_t$ is rotation by the angle $2 \pi \vartheta_t \in \R$ then $\vartheta_t$ is a continuous function of $t \in [0, T]$ satisfying $\vartheta_0=0$ and $\vartheta:=\vartheta_T \in \R \setminus \Z $.  The number $\vartheta \in \R \setminus \Z $ the \textbf{rotation angle} of $\gamma$ with respect to the trivialization and
\[
\mu_{CZ}^\Phi(\ga^k) = 2 \lfloor k \vartheta \rfloor + 1. 
\]
More generally, there is a definition of rotation number of a path of invertible $2 \times 2 $ matrices (starting at the identity) which does not require any of the matrices to be a rotation, resulting in the same formula in terms of $\vartheta$.  In the latter situation we continue the abusive practice of referring to the quantity $\vartheta$ as the rotation angle of $\gamma$.

 \item[ \textbf{Hyperbolic case:}] 
 Let $v \in \R^2$ be an eigenvector of $\phi_T$. Then for any trivialization used, the family of vectors $\{ \phi_t(v) \}_{t \in [0,T]}$, rotates through angle $\pi r$ for some integer $r$.  The integer $r$ is dependent on the choice of trivialization $\Phi$, but is always even in the positive hyperbolic case and odd in the negative hyperbolic case.  We obtain
 \[
 \mu_{CZ}(\ga^k) = k r.
 \]
 \end{list}

The Conely-Zehnder index depends only on the Reeb orbit $\gamma$ and homotopy class of $\Phi$ in the set of homotopy classes of symplectic trivializations of the 2-plane bundle $\gamma^*\xi$ over $S^1 = \R / T\Z$.  Our sign convention is that if 
\[
\Phi_1, \ \Phi_2: \gamma^*\xi \to S^1 \times \R^2
\]
are two trivializations then
\begin{equation}\label{degtrivs}
\Phi_1 - \Phi_2 = \mbox{deg} \left(\Phi_2 \circ \Phi_1^{-1} : S^1 \to \mbox{Sp}(2,\R) \cong S^1\right).
\end{equation}
Given two trivializations $\Phi_1$ and $\Phi_2$ we have that
\begin{equation}\label{cztrivs}
\czm^{\Phi_1}(\gamma^k) - \czm^{\Phi_2}(\gamma^k) =2 k (\Phi_1 - \Phi_2 ).
\end{equation}
We denote the set of homotopy classes of symplectic trivializations of the 2-plane bundle $\gamma^*\xi$ over $S^1$ by $\mathcal{T}(\gamma)$.

 The following proposition shows that in dimension 3, the Conley-Zehnder index grows almost linearly and will be used in Section \ref{quandaries}.  It follows immediately by considering the above Conley-Zehnder index formulas; see \cite[Prop. 4.4]{jo1} for further details.

 \begin{proposition}\label{almostlinear}
 Let $(M, \A)$ be a nondegenerate contact 3-manifold.  Let $\ga$ be any closed Reeb orbit of $R$ and $\ga^k$ its $k$-fold cover. Then
 \begin{equation}\label{czlove}
 k\czm^\Phi(\ga) - k+1 \leq \czm^\Phi(\ga^k) \leq k \czm^\Phi(\ga)+k-1.
 \end{equation}
\end{proposition}

The almost linear behavior of the Conley-Zehnder index is used to prove the following estimate on the index of multiply covered cylinders in symplectizations; see \cite[Lem. 2.5]{HN1} and \cite[Prop 4.5]{jo1} for full details.  

\begin{lemma}\label{index-cyl-pos}
Let $(M^3, \lambda)$ be a closed nondegenerate contact manifold, $J$ be a generic $\lambda$-compatible almost complex structure, and  $u$ be a nontrivial $J$-holomorphic cylinder in $\R \times M$.  If $\overline{u}$ denotes the somewhere injective pseudoholomorphic cylinder underlying $u$ then
\begin{enumerate}[{\em (i)}]
\item $1 \leq \mbox{\em ind}(\overline{u}) \leq \mbox{\em ind}(u).$
\item If $\mbox{\em ind}(u) =1$, and if $u$ has an end at a bad Reeb orbit, then the corresponding end of $\overline{u}$ is also at a bad Reeb orbit.
\end{enumerate}

\end{lemma}


\subsection{Regularity for cylinders}\label{reg-sec}
In this section we flesh out an observation of Hutchings \cite{obg-blog}, which enables us to obtain transversality for certain unbranched multiple covered cylinders in cobordisms of closed contact 3-manifolds where the usual automatic transversality approach, e.g. \cite[Theorem 1]{Wtrans}, is not applicable.  Before stating the results, we review the necessary set up, including background needed from embedded contact homology (ECH).

Let $u$ be an immersed pseudoholomorphic curve in $W$. Let $\pi:\widetilde{u}\to u$ be a degree $k$ unbranched\footnote{Much of this discussion also holds for branched covers of curves.} cover of $u$.  Let $N_u$ be the normal bundle to $u$.  As explained in \cite[\S 2.3]{Hu2}, there is a deformation operator
\begin{equation}\label{def-op-u}
D_u:L^2_1(\dot{\Sigma}, N_u) \to L^2(\dot{\Sigma}, T^{(0,1)}\dot{\Sigma}\otimes_{\C} N_u)
\end{equation}
and the moduli space of pseudoholomorphic curves is cut out transversely when $D_u$ is surjective.  When $D_u$ is surjective, the tangent space of the moduli space  can be identified with ker$(D_u)$ and the index of $D_u$ is the Fredholm index $\ind(u)$.


There is an induced operator associated to the cover $\widetilde{u}$ of $u$
\begin{equation}\label{def-op-cover}
D_{\widetilde{u}}:L^2_1(\widetilde{\dot{\Sigma}}, \pi^*N_u)\to L^2(\widetilde{\dot{\Sigma}}, T^{(0,1)}\widetilde{\dot{\Sigma}}\otimes_\C \pi^*N_u).
\end{equation}
The definition of these operators requires the choice of a local complex trivialization of $N_u$.  Let $z=s + it$ be a local coordinate on $u$ and use $id\bar{z}$ to locally trivialize $T^{(0,1)}\dot{\Sigma}$.  Then choose a local trivialization of $N_u$ over this coordinate neighborhood.  With respect to these coordinates and trivializations, the operator locally is of the form
\[
D_u = \partial_s + J\partial_t + \beta
\]
where $\beta$ is some $(0,1)$-form on $u$, determined by the derivative of the almost complex structure in directions normal to $u$. Using the same local trivialization for $\pi^*N_u$, we define
\[
D_{\widetilde{u}} =  \partial_s + J\partial_t + \pi^*\beta.
\]
Intuitively speaking, $D_{\widetilde{u}}$ sees deformations of $\widetilde{u}$ in directions normal to $u$. 

\begin{definition} \em
The  cover $\widetilde{u}$ is \textbf{agreeable} if $\mbox{ker}(D_{\widetilde{u}})=0.$
\end{definition}

\begin{remark}\label{nobranchtransverse}\em
If there are no branch points and $\ind(\widetilde{u})=0,$ then $\widetilde{u}$ is agreeable if and only if it is transverse.  The regularity result we will prove is in regards to unbranched covers.  However, should one need to consider the possibility of branch points,  all branched covers in the moduli space of branched covers of $u$ containing $\widetilde{u}$ must be agreeable.  This is necessary to define an \emph{obstruction bundle} over the moduli space of such branched covers in order to do gluing as in \cite{HT2}.
\end{remark}

When there are no branch points,
\[
 \ind ( D_{\widetilde{u}}) = \ind ( \widetilde{u}),
\]
otherwise
\begin{equation}\label{branchyD}
\ind(D_{\widetilde{u}}) = \ind(\widetilde{u}) - 2b.
\end{equation}
{To see how \eqref{branchyD} arises, recall that the Fredholm index of $\widetilde{u}$ is given by
\[
\ind ( \widetilde{u}) = - \chi ( \widetilde{u}) + 2c_1^\Phi(\widetilde{u}^*TW) + \mu^\Phi(\widetilde{u}).
\]
If $\widetilde{u}$ is a $k$-fold cover of $u$ we obtain by Riemann-Hurwitz, Theorem \ref{RH},
\[
\ind ( \widetilde{u}) = k\chi( u ) + b + 2kc_1^\Phi(u^*TW) + \mu^\Phi(\widetilde{u}),
\]
where $b$ is the weighted count of branch points.  However, $\ind ( \widetilde{u}) $ is not the Fredholm index of the operator $D_{\widetilde{u}}$ because the operator $D_{\widetilde{u}}$ does not consider deformations of $\widetilde{u}$ that move the branch points, so the dimension of its domain is $2b$ fewer.}

Let $h_+(\widetilde{u})$ denote the number of ends of $\widetilde{u}$ that are at positive hyperbolic orbits; this includes even covers of negative hyperbolic orbits.  A basic form of automatic transversality for asymptotically cylindrical curves is as follows.

\begin{proposition}\label{aut-agree}
Suppose that $\widetilde{u}$ is an immersed pseudoholomorphic curve and
\begin{equation}\label{eqn:atcondition}
{\mbox{\em ind}(\widetilde{u})  - 2b + 2- 2g(\widetilde{u}) - h_+(\widetilde{u}) > 0.}
\end{equation}
 Then $\widetilde{u}$ is agreeable.
\end{proposition}

\begin{proof}
Suppose $\psi\in \mbox{ker}(D_{\widetilde{u}})$ is not identically zero. From the Carleman similarity principle, every zero of $\psi$ is isolated and has positive multiplicity.  Thus the signed count of zeroes of $\psi$ is nonnegative. On the other hand, we can bound the number of zeroes of $\psi$ as in \cite[\S 4.1]{HN1} to obtain
\[
0 \geq {2 \cdot \#\psi^{-1}(0) \geq \ind(\widetilde{u}) -2b + 2- 2g(\widetilde{u})  -h_+(\widetilde{u}) .}
\]
If the right hand side is negative we obtain a contradiction. As a consequence, if the right hand side is negative then $\widetilde{u}$ is agreeable.
\end{proof}

In \cite[\S 4.2]{HN1}, we obtained the following transversality result for cylinders in symplectizations via the above form of automatic transversality.

\begin{lemma}
\label{lem:at}
Let $M$ be a closed three-manifold with a nondegenerate contact form $\lambda$. Let $J$ be a generic $\lambda$-compatible almost complex structure on $\R\times M$. Then:
\begin{enumerate}[{\em (i)}]
\item
For any Reeb orbits $\gamma_+$ and $\gamma_-$, the moduli space $\widehat{\mathcal{M}}^J_1(\gamma_+,\gamma_-)/\R$ is a $0$-manifold cut out transversely.
\item
If $\gamma_+$ and $\gamma_-$ are good Reeb orbits, then the moduli space $\widehat{\mathcal{M}}^J_2(\gamma_+,\gamma_-)/\R$ is a $1$-manifold cut out transversely.
\item If $\gamma_+$ and $\gamma_-$ are good, then the function
\[
d:\widehat{\mathcal{M}}^J_2(\gamma_+,\gamma_-)/\R \longrightarrow \Z^{>0},
\]
which associates to each cylinder its covering multiplicity, is locally constant.
\end{enumerate}
\end{lemma}

We also want to show in certain situations that $\widetilde{u}$ is agreeable, even when Proposition \ref{aut-agree} is not applicable.  The formulation and proof of these conditions, necessitates some embedded contact homology (ECH) apparatus, which we now review. 

\subsubsection{The ingredients comprising the ECH index}

The definition of the ECH index depends on three components: the relative first Chern class $c_\Phi$, which detects the contact topology; the relative intersection pairing $Q_\Phi$, which detects the algebraic topology; and the Conley-Zehnder terms, which detect the contact geometry.  
Let $\alpha=\{(\alpha_i,m_i)\}$ and $\beta=\{(\beta_j,n_j)\}$ be Reeb orbit sets  in the same homology class, $\sum_i m_i [\alpha_i]=\sum_j n_j [\beta_j]=\Gamma\in H_1(M).$   Let $H_2(M,\alpha,\beta)$ denote the set of relative homology classes of 2-chains $Z$ in $M$ such that 
\[
\partial Z =\sum_i m_i \alpha_i - \sum_j n_j \beta_j.
\]

\begin{definition}[relative first Chern class] 

\em
 Fix trivializations $\Phi_i^+ \in \mathcal{T}(\alpha_i)$ for each $i$ and 
$\Phi_j^- \in \mathcal{T}(\beta_j)$ and denote this set of trivalization choices by $\Phi \in \mathcal{T}(\alpha,\beta)$.  Let  $Z \in H_2(M,\alpha,\beta)$.  We define the \textbf{relative first Chern class}
\[
c_\Phi(Z) : =  c_1\left( \xi |_Z, \Phi \right) \in \Z
\]
in terms of the following signed count of zeros of a particular section.  Given a class $Z\in H_2(M,\alpha,\beta)$ we represent $Z$ by a smooth map $f: S \to M$ where $S$ is a compact orieted surface with boundary.  Choose a section $\psi$ of $f^*\xi$ over $S$ such that $\psi$ is transverse to the zero section and $\psi$ is nonvanishing over each boundary component of $S$ with winding number zero with respect to the trivialization $\Phi$.  We define 
\[
c_\Phi(Z) : = \# \psi^{-1}(0),
\]
where `\#' denotes the signed count.
\end{definition}

  In addition to being well-defined, the relative first Chern class satisfies
\[
c_\Phi(Z) - c_\Phi(Z') = \langle c_1(\xi), Z - Z' \rangle.
\]
Given another collection of trivialization choices $\Phi' = \left( \{ {\Phi'}_i^+ \}, \{ {\Phi'}_j^-\} \right)$ over the orbit sets and the convention \eqref{degtrivs}, we have
\begin{equation}\label{cherntriv}
c_\Phi(Z) - c_{\Phi'}(Z) = \sum_i m_i\left(\Phi_i^+-{\Phi}_i^{+'}\right) - \sum_j n_j \left(\Phi_j^--{\Phi}_j^{-'}\right).
\end{equation}
{The above formula  \eqref{cherntriv} also holds for computations of the relative first Chern class of the normal bundle of a curve.}

Before defining the relative intersection pairing we define the writhe and linking number.  Given a somewhere injective curve $u \in \M^J(\gamma_+,\gamma_-)$, we consider the slice $u \cap ( \{s\} \times M)$.  If $s \gg 0$, then the slice $u \cap ( \{s\} \times M)$ is an embedded curve which is a braid $\zeta_+$ around the Reeb orbit $\gamma_+$ with $m(\gamma_+)$ strands.  If $\gamma_+$ is an embedded Reeb orbit with tubular neighborhood $N$ then we can identify $N$ with a disk bundle in the normal bundle to $N$, and also with $\xi|_{\gamma_+}$.

Thus $\zeta_+$ can be realized as a braid in $N$, defined as a link in $N$ such that that the tubular neighborhood projection restricts to a submersion $\zeta_+ \to \gamma$.    Since the braid $\zeta_+$ is embedded for all $s \gg 0$, its isotopy class does not depend on $s\gg 0$. The trivialization $\Phi$ is used to identify the braid $\zeta_+$ with a link in $S^1 \times D^2$.  We identify $S^1 \times D^2$ with an annulus cross an interval, projecting $\zeta_+$ to the annulus, and require that the normal derivative along $\gamma$ agree with the trivialization $\Phi$.  

\begin{definition}[writhe]

\em
We define the \textbf{writhe} of this link, which we denote by $w_\Phi(\zeta_+) \in \Z$, by counting the crossings of the projection to $\R^2 \times \{ 0 \} $ with (nonstandard) signs.  Namely, we use the sign convention in which counterclockwise rotations in the $D^2$ direction as one travels counterclockwise around $S^1$ contribute positively.  Analogously the slice $u \cap ( \{s\} \times M)$ for $s \gg 0$ produces a braid $\zeta_-$ and we denote this braid's writhe by $w_\Phi(\zeta_-) \in \Z$.  
\end{definition}

The writhe depends only on the isotopy class of the braid and the homotopy class of the trivialization $\Phi$.  If $\zeta$ is an $m$-stranded braid and  $\Phi' \in \mathcal{T}(\gamma)$ is another trivialization then
\[
w_\Phi(\zeta) - w_{\Phi'}(\zeta) = m(m-1)(\Phi' - \Phi)
\]
because shifting the trivialization by one adds a full clockwise twist to the braid.

If $\zeta_1$ and $\zeta_2$ are two disjoint braids around an embedded Reeb orbit $\gamma$ we can define their \textbf{linking number} $\ell_\Phi(\zeta_1,\zeta_2) \in \Z$
to be the linking number of their oriented images in $\R^3$.  The latter is by definition one half of the signed count of crossings of the strand associated to $\zeta_1$ with the strand associated to $\zeta_2$ in the projection to $\R^2 \times \{0\}$.  If the braid $\zeta_k$ has $m_k$ strands then a change in trivialization results in the following formula
\[
\ell_\Phi(\zeta_1,\zeta_2) - \ell_{\Phi'}(\zeta_1,\zeta_2) = m_1m_2 (\Phi'-\Phi).
\]
The writhe of the union of two braids can be expressed in terms of the writhe of the individual components and the linking number:
\[
w_\Phi(\zeta_1 \cup \zeta_2) = w_\Phi(\zeta_1) + w_\Phi(\zeta_2) + 2\ell_\Phi(\zeta_1,\zeta_2).
\]
If $\zeta$ is a braid around an embedded Reeb orbit $\gamma$ which is disjoint from $\gamma$ we define the \textbf{winding number} to be the linking number of $\zeta$ with $\gamma$:
\[
\eta_\Phi(\zeta) :=\ell_\Phi(\zeta,\gamma) \in \Z.
\]
In order to speak more ``globally" of writhe and winding numbers associated to a curve, we need the following notion of an admissible representative for a class $Z \in H_2(M,\alpha,\beta)$, as in \cite[Def. 2.11]{Hrevisit}.  Given $Z\in H_2(M,\alpha,\beta)$ we define an \textbf{admissible representative} of $Z$ to be a smooth map $f: S \to [-1,1] \times M$, where $S$ is an oriented compact surface such that
\begin{enumerate}
\item The restriction of $f$ to the boundary $\partial S$ consists of positively oriented  covers of $\{ 1\} \times \alpha_i$ whose total multiplicity is $m_i$ and negatively covers of $\{ -1\} \times \beta_j$ whose total multiplicity is $n_j$.
\item The projection $\pi: [-1,1] \times M \to M$ yields $[\pi(f(S))]=Z$.
\item The restriction of $f$ to $\mbox{int}(S)$ is an embedding and $f$ is transverse to $\{-1,1\} \times M$.
\end{enumerate}
The utility of the notion of an admissible representative $S$ for $Z$ can be seen in the following. For $\varepsilon >0$ sufficiently small, $S \cap ( \{ 1 - \varepsilon \} \times Y)$ consists of braids $\zeta_i^+$ with $m_i$ strands in disjoint tubular neighborhoods of the Reeb orbits $\alpha_i$, which are well defined up to isotopy.  Similarly, $S \cap ( \{ -1+ \varepsilon \} \times Y)$ consists of braids $\zeta_j^-$ with $n_j$ strands in disjoint tubular neighborhoods of the Reeb orbits $\alpha_i$, which are well defined up to isotopy.  

Thus an admissible representative of $Z\in H_2(M;\alpha,\beta)$ permits us to define the \textbf{total writhe} of a curve interpolating between the orbit sets $\alpha$ and $\beta$ by 
\[
w_\Phi(S) = \sum_i w_{\Phi_i^+}(\zeta_i^+) -  \sum_j w_{\Phi_j^-}(\zeta_j^-).
\]
Here $\zeta_i^+$ are the braids with $m_i$ strands in a neighborhood of each of the $\alpha_i$ obtained by taking the intersection of $S$ with $\{ s \} \times M$ for $s$ close to 1.  Similarly, the $\zeta_j^-$ are the braids with $n_j$ strands in a neighborhood of each of the $\beta_j$ obtained by taking the intersection of $S$ with $\{ s \} \times M$ for $s$ close to $-1$.  Bounds on the writhe in terms of the Conley-Zehnder index are given in \cite[\S 3.1]{HN1}, which relates to asymptotic behavior of pseudoholomorphic curves, extensively explored by Hutchings, cf. \cite[\S 5.1]{Hu2}.  

Taking a similar viewpoint with regard to the linking number results in the following formula.  If $S'$ is an admissible representative of $Z' \in H_2(M,\alpha',\beta')$ such that the interior of $S'$ does not intersect the interior of $S$ near the boundary, with braids $\zeta_i^{+ '}$ and $\zeta_j^{- '}$ we can define the \textbf{linking number of $S$ and $S'$} to be
\[
\ell_\Phi(S,S') := \sum_i \ell_\Phi(\zeta_i^+,\zeta_i^{+ '}) - \sum_j \ell_\Phi(\zeta_j^-,\zeta_j^{- '}).
\]
Above the orbit sets $\alpha$ and $\alpha'$ are both indexed by $i$, so sometimes $m_i$ or $m_{i}'$ is 0, similarly both $\beta$ and $\beta'$ are indexed by $j$ and sometimes $n_j$ or $n_{j}'$ is 0.  The trivialization $\Phi$ is a trivialization of $\xi$ over all Reeb orbits in the sets $\alpha, \alpha', \beta,$ and $\beta'$.

The relative intersection pairing can be defined using an admissible representative, which is more general than the notion of a $\Phi$-representative \cite[Def. 2.3]{Hindex}, as the latter uses the trivialization to control the behavior at the boundary.  Consequently, \emph{we see an additional linking number term appear in the expression of the relative intersection pairing when we use an admissible representative. }  
\begin{definition}[relative intersection pairing using an admissible representative]
 \em
Let  $S$ and $S'$ be two surfaces which are admissible representatives of $Z \in H_2(M,\alpha,\beta)$ and $Z'\in H_2(M,\alpha',\beta')$  whose interiors $\dot{S}$ and $\dot{S}'$ are transverse and do not intersect near the boundary.  We define 
the \textbf{relative intersection pairing} by the following signed count
\begin{equation}
Q_\Phi(Z,Z'):= \# \left( \dot{S} \cap \dot{S}'\right) - \ell_\Phi(S,S').
\end{equation}
Moreover, $Q_\Phi(Z,Z')$ is an integer which depends only on $\alpha, \beta, Z, Z'$ and $\Phi$.  If $Z=Z'$ then we write $Q_\Phi(Z):=Q_\Phi(Z,Z)$.  
\end{definition}


For another collection of trivialization choices $\Phi'$, 
\[
Q_\Phi(Z,Z') - Q_{\Phi'}(Z,Z') = \sum_i m_i m_i' (\Phi_i^+ - \Phi_i^{+'}) - \sum_i n_j n_j' (\Phi_i^- - \Phi_i^{-'}).
\]
We recall how \cite[\S 3.5]{Hrevisit} permits us to compute the relative intersection pairing using embedded surfaces in $M$.  An admissible representative $S$ of $Z \in H_2(M,\alpha,\beta)$ is said to be \textbf{nice} whenever the projection of $S$ to $M$ is an immersion and the projection of the interior $\dot{S}$ to $M$ is an embedding which does not intersect the $\alpha_i$'s or $\beta_j$'s. Lemma 3.9 from  \cite{Hrevisit} establishes that if none of the $\alpha_i$ equa the $\beta_j$ then every class $Z \in H_2(M,\alpha,\beta)$ admits a nice admissible representative.

If $S$ is a nice admissible representative of $Z$ with associated braids $\zeta_i^+$ and $\zeta_j^-$ then we can define the winding number
\[
\eta_\Phi(S) : = \sum_i \eta_{\Phi_i^+}\left(\zeta_i^+\right) -\sum_j \eta_{\Phi_j^-}\left(\zeta_j^-\right)
\]

\begin{lemma}[Lemma 3.9 \cite{Hrevisit}]
Suppose that $S$ is a nice admissible representative of $Z$.  Then
\[
Q_\Phi(Z) = -w_\Phi(S) - \eta_\Phi(S)
\]
\end{lemma}

We are now ready to give the definition of the ECH index.


\begin{definition}[ECH index]
\em
We define the \textbf{ECH index} to be
\[
I(\alpha,\beta,Z) = c_1^\Phi(Z) + Q_\Phi(Z) +  \sum_i \sum_{k=1}^{m_i}\czm^\Phi(\alpha_i^k) - \sum_j \sum_{k=1}^{n_j}\czm^\Phi(\beta_j^k).
\]
\end{definition}

Given a trivialization $\Phi$ of $\xi$ over the $\gamma_i$'s contained in an orbit set $\gamma = \{ (\gamma_i, m_i )\}$ we make the shorthand definition
\[
\mu_\Phi(\gamma) := \sum_i \sum_{k=1}^{m_i}\czm^\Phi(\gamma_i^k).
\]
In this shorthand notation the ECH index is expressed as
\begin{equation}
I(\alpha,\beta,Z) = c_1^\Phi(Z) + Q_\Phi(Z) +  \mu_\Phi(\alpha) - \mu_\Phi(\beta).
\end{equation}
Another set of trivialization choices $\Phi'$ for $\gamma$ yields
\begin{equation}
\mu_\Phi(\gamma) - \mu_{\Phi'}(\gamma) = \sum_i (m_i^2 + m_i)(\Phi_i' - \Phi_i).
\end{equation}
Moreover, the ECH index does not depend on the choice of trivialization.  The budding ECH enthusiast can find further details in \cite[\S 3]{Hu2}.

\begin{remark} \em
If $u$ is a cylinder, then the orbit sets $\alpha$ and $\beta$ each consist of single Reeb orbit.  We denote these orbits by $\gamma_+$ and $\gamma_-$, respectively.  We further take $\overline{\gamma_+}$ and $\overline{\gamma_-}$ to be the respective underlying embedded Reeb orbits for $\gamma_+$ and $\gamma_-$, e.g. 
\[
\overline{\gamma_\pm}^{m(\gamma_\pm)} = \gamma_\pm,
\]
where $m(\gamma)$ is the multiplicity of the orbit $\gamma$.
Then for $Z \in H_2(M,\gamma_+,\gamma_-)$ we have 
\[
{I(u) = c_1^\Phi(u^*\xi) + Q_\Phi(Z) + \sum_{\ell=1}^{m(\gamma_+)} \czm\left (\overline{\gamma_+}^\ell \right)- \sum_{\ell=1}^{m(\gamma_-)} \czm\left (\overline{\gamma_-}^\ell \right)}
\]

\end{remark}

\begin{definition}[\em{$I_u(\widetilde{u})$, the ECH index of $\widetilde{u}$ ``relative to $u$''}]\em
We can similarly define the ECH index of an immersed curve $u$ and its cover $\widetilde{u}$ within the (underlying) normal bundle $N_u$.  We regard $N_u$ a type of completed symplectic cobordism between the disjoint union over the ends of $u$ of the normal bundle of the corresponding (possibly multiply covered) Reeb orbit.  When $u$ is regarded as a zero section it defines an embedded pseudoholomorphic curve in $N_u$ whose ends are all at distinct simple Reeb orbits, even if this is not true for the original curve $u$.   As a result, there is a well defined notion of the ECH index of $\widetilde{u}$ in the normal bundle $N_u$. This is defined by copying the above formulas in the normal bundle $N_u$.  We can think of this as the ECH index ``relative to $u$'', and we denote it by $I_u(\widetilde{u}).$ 
\end{definition}

\begin{remark}\em
If $u$ is somewhere injective in $(W,J)$ and all its ends are at distinct simple Reeb orbits, then the ECH index of $\widetilde{u}$ in $N_u$ agrees with the ECH index of $\widetilde{u}$ in $W$,
\[
I_u(\widetilde{u}) =I(\widetilde{u}). 
\]
If $u$ does not have these properties, then it is possible that $I_u(\widetilde{u})\neq I(\widetilde{u}).$
\end{remark}

\subsubsection{Recollections on the relative adjunction formula}

In this section we review the relative adjunction formulas of interest, which are later used to show certain multiply covered cylinders are agreeable.   This is taken from \cite[\S 3]{Hindex} and is stated for pseudoholomorphic curves interpolating between orbit sets $\alpha$ and $\beta$ in symplectizations.   As explained in \cite[\S 4.4]{Hrevisit} the proof carries over in a straightforward manner to exact symplectic cobordisms.   

\begin{lemma}
\label{lem:adjunction}
Let $u\in \mathcal{M}(\alpha,\beta)$ be somewhere injective, $S$ be a representative of $Z\in H_2(M,\alpha,\beta)$, and $\Phi\in \mathcal{T}(\alpha,\beta)$.  Let $N_S$ be the normal bundle to $S$.  
\begin{enumerate}[{\em (i)}]
\item If $u$ is further assumed to be embedded everywhere then
\begin{equation}\label{relcalc}
c_1^{\Phi}(Z) = \chi(S) + c_1^\Phi(N_S).
\end{equation}
\item For general embedded representatives $S$, e.g. ones not necessarily coming from pseuoholomorphic curves, \eqref{relcalc} holds mod 2 and
\begin{equation}\label{chernwQ}
c_1^\Phi(N_S) = w_\Phi(S) + Q_\Phi(Z,Z).
\end{equation}
\item If $u$ is embedded except at possibly finitely many singularities then
\begin{equation}\label{relsingeq}
c_1^\Phi(Z) = \chi(u) + Q_\Phi(Z) + w_\Phi(u) - 2\delta(u), 
\end{equation}
where $\delta(u)$ is a sum of positive integer contributions from each singularity.
\end{enumerate}
\end{lemma}

\begin{proof}[Sketch of Proof]
\begin{enumerate}[(i)]
\item[]
\item We have the following decomposition of complex vector bundles: 
\begin{equation}\label{VBiso}
({\mathbb C}\oplus\xi)|_S = TW|_S = TS\oplus N_S.
\end{equation}
Let $\psi_\xi$ and $\psi_N$ be $\Phi$-trivial sections of $\xi|_S$ and $N|_S$ and let $\psi_S$ be a nonvanishing section of $TS|_{\partial S}$ tangent to $S$.  Over $\partial S$ we have a homotopy through nonvanishing sections of the determinant line bundles
\[
1 \wedge \psi_\xi \approx \psi_S \wedge \psi_N.
\]
In general, if $L_i$ is a complex line bundle on $S$ and $s_i$ is a nonvanishing section of $L{_i}|_{\partial S}$ up to homotopy for $i=1,2$ then
\[
c_1(\mbox{det}(L_1 \oplus L_2), s_1 \wedge s_2) = c_1(L_1,s_1) + c_1(L_2,s_2).
\]
In light of this identity with respect to the above sections, we obtain
\[
c_1^\Phi(Z) = \chi(S) + c_1^\Phi(N_S),
\]
as desired.
\item The isomorphism in \eqref{VBiso} still holds at the level of real vector bundles and still respects the complex structure on $\partial S$ after straightening $S$ to be normal to $\{-1,1\} \times M$.  As a result, the relative first Chern classes differ by an even integer because changing the complex structure on a rank two complex vector bundle over a closed surface changes the first Chern class by an even integer.

To prove \eqref{chernwQ} we recall the following argument.  Let $\epsilon > 0$ be small and let 
\[
S_0 = S \cap (-1 +\epsilon, 1-\epsilon) \times M).  
\]
Let $S'$ be a surface in which $S \setminus S_0$ is replaced by a surface $S_1$, consisting of cobordisms with $\Phi$-trivial braids so that $S'$ is a $\Phi$-representative of $Z$.  Let $\psi$ be a section of the normal bundle $N_{S'}$ that is $\Phi$-trivial over $\partial S_1$.  Let $\psi_0$ and $\psi_1$. denote the restrictions of $\psi$ to $S_0$ and $S_1$ respectively.  We can compute $Q_\Phi(Z,Z)$ by counting the intersections of $S'$ with a pushoff of $S'$ via $\Psi$ so that
\[
\begin{array}{lcl}
Q_\Phi(Z,Z) &=& \# \Psi^{-1}(0) \\
&=& \# \psi^{-1}_0(0) + \# \psi^{-1}_1(0) \\
&=&c_1^\Phi(N_S) + \# \psi^{-1}_1(0), \\
\end{array}
\]
where `$\#$' indicates the number of points with signs after appropriately perturbing to obtain transversality.  

To see why
\[
\# \psi^{-1}_1(0) = -w_\Phi(S),
\]
we note that in our cobordism of braids that we can take $\psi_1$ to be the projection of a nonzero vertical tangent vector in the annulus cross (-1,1) that we have identified a tubular neighborhood of $\gamma$ with.  This section will have zeros at the branch points of the projection to $([-1,-1+\epsilon] \cup [1-\epsilon, 1])\times M$ where the writhes of the braids change.
\item Near each singular point we can perturb the surface to become an immersion which is symplectic with respect to $\omega + ds \wedge dt$ on $\R \times M$ and which only has transverse double point singularities.  The local contribution to $\delta$ is then the number of double points.  

To prove \eqref{relsingeq} we carry out the above perturbation near each singularity without affecting any of the other terms.  As a result, the normal bundle $N_u \to u$ is well defined and a modification of the proof shows that  shows that \eqref{relcalc} still holds while a correction of $2\delta(u)$ is needed for \eqref{chernwQ} to hold. 
\end{enumerate}

\end{proof}

\begin{remark}\em
If $u$ is a closed pseudoholomorphic curve, then there is no writhe term or trivialization choice, and \eqref{relsingeq} reduces to the usual adjunction formula
\[
\langle c_1(TW),[u]\rangle = \chi(u) + [u] \cdot [u] - 2\delta(u).
\]
\end{remark}

The following remark follows from Lemma \ref{lem:adjunction}(ii) and will be used later on to compute the ECH index for certain unbranched multiply covered cylinders relative to the embededed curve.

\begin{remark}\label{helpfuladjrem}\em
Suppose that $u$ is a somewhere injective cylinder in $(W,J)$ with ends at simple positive hyperbolic orbits $\gamma_+$ and $\gamma_-$ satisfying $\ind(u) = 0$ and that $\Phi$ is a trivialization for which $c_1^{\Phi}(u)=0$.  Let $Z$ be the zero section in the normal bundle $N$ determined by $u$.  Using the same trivialization $\Phi$ we have for any representative $S$ of $Z$ that $c_1^{\Phi}(N_S)=0$.  Moreover, $w_{\Phi}(Z)=0$ because the ends of $u$ are at simple distinct Reeb orbits and since $u$ is embedded, $\delta(u)=0$.  Thus the relative adjunction formula implies that $Q_{\Phi}(Z)=0$.  By \cite[(3.11)]{Hrevisit} we can deduce that for any $k$-fold cover of $u$, the associated zero section $kZ$ in $N$ satisfies
\[
Q_{\Phi_0}(kZ)=k^2Q_{\Phi_0}(Z)=0.
\]
\end{remark}


\subsubsection{The ECH partition conditions and index inequality}
Our regularity result also relies on the ECH partition conditions.  These conditions are a topological type of data associated to the pseudoholomorphic curves (and currents) which can be obtained indirectly from certain ECH index relations.  In particular, the covering multiplicities of the Reeb orbits at the ends of the nontrivial components of the pseudoholomorphic curves (and currents) are uniquely determined by the trivial cylinder component information.  While not needed in this paper, we note that the genus can be determined by the current's relative homology class.

\begin{definition}\em \cite[\S 3.9]{Hu2} Let $\gamma$ be an embedded Reeb orbit and $m$ a positive integer.  We define two partitions of $m$, the \textbf{positive partition} $P^+_\gamma(m)$ and 
the \textbf{negative partition} $P^-_\gamma(m)$\footnote{Previously the papers \cite{Hindex, Hrevisit} used the terminology incoming and outgoing partitions.}  as follows.
\begin{itemize}
\item If $\gamma$ is positive hyperbolic, then
\[
P_\gamma^+(m): = P_\gamma^-(m): = (1,...,1).
\]
\item If $\gamma$ is negative hyperbolic, then
\[
P_\gamma^+(m): = P_\gamma^-(m): = \left\{ \begin{array}{ll}
(2,...,2) & m \mbox{ even,} \\
(2,...,2,1) & m \mbox{ odd. } \\
\end{array}
\right .
\]
\item If $\gamma$ is elliptic then the partitions are defined in terms of the quantity $\vartheta \in \R \setminus \Z $ for which $\czm^\Phi(\gamma^k) = 2\lfloor k \vartheta \rfloor + 1$.  We write 
\[
P_\gamma^\pm(m): = P_\vartheta^\pm(m),
\]
with the right hand side defined as follows. 

 Let $\Lambda^+_\vartheta(m)$ denote the lowest convex polygonal path in the plane that starts at $(0,0)$, ends at $(m,\lceil m \vartheta \rceil)$, stays above the line $y = \vartheta x$, and has corners at lattice points.  Then the integers $P^+_\vartheta(m)$ are the horizontal displacements of the segments of the path $\Lambda^+_\vartheta(m)$ between the lattice points.

Likewise, let  $\Lambda^-_\vartheta(m)$ denote the highest concave polygonal path in the plane that starts at $(0,0)$, ends at $(m,\lfloor m \vartheta \rfloor)$, stays below above the line $y = \vartheta x$, and has corners at lattice points.  Then the integers $P^-_\vartheta(m)$ are the horizontal displacements of the segments of the path $\Lambda^-_\vartheta(m)$ between the lattice points,

Both $P_\vartheta^\pm(m)$ depend only on the class of $\vartheta$ in $\R \setminus \Z$.  Moreover, $P_\vartheta^+(m) = P_{-\vartheta}^-(m)$.
\end{itemize}

\end{definition}

\begin{example}\em
If the rotation angle for elliptic orbit $\gamma$ satisfies $\vartheta \in (0,1/m)$ then
\[
\begin{array}{lcl}
P_\vartheta^+(m) &=& (1,...,1) \\
P_\vartheta^-(m) &=& (m). \\
\end{array}
\]
The partitions are quite complex for other $\vartheta$ values, see \cite[Fig. 1]{Hu2}.
\end{example}

\begin{definition} \em
We say that $\widetilde{u}$ satisfies \textbf{the ECH partition conditions ``relative to u''} if it satisfies the usual ECH partition conditions in the normal bundle $N_u.$  
\end{definition}
If all ends of $u$ are at distinct simple Reeb orbits then $\widetilde{u}$ satisfies the ECH partition conditions if and only if $\widetilde{u}$ satisfies the ECH partition conditions relative to $u$. \\

We end this section by mentioning the ECH index inequality \cite[Theorem 4.15]{Hrevisit} in symplectic cobordisms.  As before we take $\alpha=\{(\alpha_i,m_i)\}$ and $\beta=\{(\beta_j,n_j)\}$ to be Reeb orbit sets in the same homology class.  Let $C \in \mathcal{M}(\alpha,\beta)$.  For each $i$ let $a_i^+$ denote the number of positive ends of $C$ at $\alpha_i$ and let $\{ q_{i,k}^+\}_{k=1}^{a_i^+}$ denote their multiplicities.  Thus $\sum_{k=1}^{a_i^+} q_{i,k}^+=m_i$.   Likewise, for each $j$ let $b_i^-$ denote the number of negative ends of $C$ at $\beta_j$ and let $\{ q_{j,k}^-\}_{k=1}^{b_j^-}$
denote their multiplicities; we have $\sum_{k=1}^{b_j^-} q_{j,k}^-=n_j$.


\begin{theorem}[ECH index inequality]
Suppose $C \in \mathcal{M}(\alpha,\beta)$ is somewhere injective.  Then
\[
\mbox{\em ind}(C) \leq I(C) - 2 \delta(C).
\]
Equality holds only if $\{  q_{i,k}^+ \} = P_{\alpha_i}^+(m_i)$ for each $i$ and $\{  q_{j,k}^- \} = P_{\beta_j}^-(n_j)$ for each $j$.
\end{theorem}

\subsubsection{Agreeability via ECH and regularity for cylinders}

We are now ready to prove the below result in regards to agreeable multiply covered curves, which we will use to prove that certain cylinders in cobordisms are regular.

\begin{proposition}\label{awesomesauce}
Assume that $\mbox{\em ker}(D_u)=0$. Suppose that either $\mbox{\em ind}(\widetilde{u})>I_u(\widetilde{u}),$ or $\mbox{\em ind}(\widetilde{u})=I_u(\widetilde{u})$ and $\widetilde{u}$ does not satisfy the ECH partition conditions relative to $u$. Furthermore, if $\widetilde{u}\to u$ factors through a branched cover $\widehat{u}\to u$ whose degree is between 1 and $k$, then assume that the above condition also holds with $\widetilde{u}$ replaced by $\widehat{u}$. Then $\widetilde{u}$ is agreeable.
\end{proposition}

\begin{proof}

Regarding the normal bundle $N_u$ as a  four manifold, there is a unique almost complex structure on $N_u$ whose restriction to the fibers agrees with the almost complex structure $J$ on $W$, such that a local section $\psi$ is in the kernel of the operator $D_u$ if and only if $\psi$ is a pseudoholomorphic map from a neighborhood in $u$ to $N_u.$

Suppose $\widetilde{\psi}$ is a nonzero element of $\mbox{ker}(D_{\widetilde{u}})$. Let $\psi$ denote the image of $\widetilde{\psi}$ under the projection $\pi^*N_u\to N_u.$ Then $\psi$ is a holomorphic curve in $N_u.$ By the assumption that $\mbox{ker}(D_u)=0$ and the assumption about intermediate branched covers in Proposition \ref{awesomesauce}, we can assume without loss of generality that $\psi$ is somewhere injective.

A version of the ECH index inequality tells us that $\ind(\psi) \le I_u(\psi)$, with equality only if $\psi$ satisfies the ECH partition conditions relative to $u$. This is proven in the same manner as the usual ECH index inequality, except that in this case one does not need Siefring's nonlinear analysis \cite{Sief}. Rather, one can appeal to the linear analysis of \cite{Hindex}.

As a consequence, we can replace $\psi$ everywhere by $\widetilde{u}$ without changing anything in the first paragraph of our proof. This yields a contradiction to the assumptions of said Proposition \ref{awesomesauce}, which means that $\psi$ could not exist, so $\widetilde{u}$ is agreeable. 

\end{proof}

We can now use the above result to obtain regularity for the unbranched covers of a somewhere injective cylinder with Fredholm index zero having one positive end and one negative end, each at positive hyperbolic orbits, in a cobordism.  Unbranched covers of cylinders with Fredholm index zero which do not limit on positive hyperbolic orbits are guaranteed to be regular by automatic transversality as stated in Proposition \ref{aut-agree}.

\begin{proposition}\label{cylcoverssep}
Let $(W, \lambda)$ be an exact symplectic cobordism between two dynamically separated contact forms and $J$ a generic compatible almost complex structure..   Suppose that $u \in \calm(\ga_+; \ga_-)$ is a somewhere injective nonconstant cylinder with Fredholm index zero which has one positive end and one negative end, each at positive hyperbolic orbits.  Then any unbranched cover of $\widetilde{u} \in \calm(\ga_+^k; \ga_-^k)$ is agreeable and hence regular.

\end{proposition}

\begin{proof}
From the definition of dynamically separated and because $[\ga_+]=[\ga_-]$ we have
\[
{\mbox{ ind}}(\widetilde{u})={\mbox{ ind}}({u}) =0 .
\]
The discussion in Remark \ref{helpfuladjrem} permits us to conclude that
\[
I_u(\widetilde{u}) = {\mbox{ ind}}(\widetilde{u}) =0 .
\]
Moreover, the ECH partition conditions fail, because they stipulate that $\widetilde{u}$ would need to have $k$ positive ends and $k$ negative ends.  As a result, Proposition \ref{awesomesauce} permits us to conclude that $\widetilde{u}$ is agreeable.  Finally, since $\widetilde{u}$ does not have any branch points Remark \ref{nobranchtransverse} yields that agreeability holds if and only if regularity does.
 \end{proof}


\begin{proposition}\label{cylcoverssep-par}
Let $(M,\xi)$ be a closed contact 3-manifold which admits two distinct nondegenerate dynamically separated contactomorphic contact forms $\lambda_\pm$ and $J_\pm$ be generic $\lambda_\pm$-compatible almost complex structures.  Suppose $J_0$ and $J_1$ are two generic choices of compatible almost complex structures on $W$ that both match $J_\pm$ on the cylindrical ends and let $\{ J_\tau \}_{\tau \in [0,1]}$ be a generic smooth path of cobordism compatible almost complex structures connecting $J_0$ to $J_1$. Then any unbranched cover $\widetilde{u}$ of a somewhere injective parametrically regular cylinder $u \in \M^{J_\tau}(\gamma_+,\gamma_-)$ satisfying $\mbox{\em ind}(u)=-1$, which exists for isolated values of $\tau \in (0,1)$, satisfies 
$\mbox{\em ind}(\widetilde{u})=-1$ and is also parametrically regular.
\end{proposition}

\begin{proof}
From the definition of dynamically separated and because $[\ga_+]=[\ga_-]$ we have
\[
{\mbox{ ind}}(\widetilde{u})={\mbox{ ind}}({u}) = -1.
\]
As a result, one of the orbits must be positive hyperbolic and an analogous argument as in Remark \ref{helpfuladjrem} permits us to conclude that
\[
I_u(\widetilde{u}) = -{\mbox{ deg}}(\widetilde{u}).
\]
Moreover, the ECH partition conditions fail, given that one of the orbits is hyperbolic for index reasons.  As a result, Proposition \ref{awesomesauce} permits us to conclude that $\widetilde{u}$ is agreeable. 

Finally, we show that parametric agreeability in this situation implies parametric regularity.  Recall the deformation operators $D_u$ and $D_{\widetilde{u}}$ from \eqref{def-op-u} and \eqref{def-op-cover} respectively.  Parametric regularity is the condition that $\dfrac{d}{d\tau}\left( \overline{\partial}_{J_\tau}u \right)$ span $\mbox{Coker}(D_u)$.  Denoting $D_u^*$ as the formal adjoint of $D_u$ we have that $\mbox{Coker}(D_u) = \mbox{Ker}(D_u^*)$, hence parametric regularity is equivalent to the following conditions
\begin{enumerate}[(i)]
\item $\mbox{dim}(\mbox{Ker}(D_u^*))=1$;  
\item If $\eta$ generates $\mbox{Ker}(D_u^*)$ then $\left \langle \eta, \dfrac{d}{d\tau}\left( \overline{\partial}_{J_\tau}u \right) \right \rangle =0.$
\end{enumerate}

We established that $\widetilde{u}$ is agreeable, e.g. $\mbox{dim}(\mbox{Ker}(D_{\widetilde{u}}^*))=1$, so condition (i) is satisfied.   To see why condition (ii) holds,  note that the generator $\eta$ of the underlying somewhere injective cylinder $u$ pulls back to $\widetilde{\eta}$ which generates $\mbox{Ker}(D_{\widetilde{u}}^*).$  Thus 
\[
\left \langle \widetilde{\eta}, \dfrac{d}{d\tau}\left( \overline{\partial}_{J_\tau}\widetilde{u} \right) \right \rangle = - \mbox{deg}(\widetilde{u}) \left \langle \eta, \dfrac{d}{d\tau}\left( \overline{\partial}_{J_\tau}u \right) \right \rangle =0,
\]
hence we may conclude that $\widetilde{u}$ is parametrically regular.
 \end{proof}


 \begin{remark}\label{remcyltrans}\em
Proposition \ref{aut-agree}, Proposition \ref{cylcoverssep}, and Proposition \ref{cylcoverssep-par} are used to show that the chain map and chain homotopy equations are well-defined, as they guarantee (parametric) regularity for the necessary index 1, 0, -1 cylinders.  Outside of a symplectization, wherein automatic transversality holds (cf. Lemma \ref{lem:at}), these (parametric) regularity results are highly dependent on the dynamically separated condition.  
\end{remark}


\section{Bounds on buildings}\label{invariance}
In Section \ref{quandaries} we obtain lower bounds on the Fredholm index which enables us to rule out noncylindrical levels from appearing in the compactification of  moduli spaces of cylinders in cobordisms between nondegenerate dynamically separated contact forms.   When combined with the regularity results of Section \ref{background}, we will be able to define the chain map and chain homotopy by directly counting of elements of moduli spaces of cylinders. 
The construction of the chain map is given in Section \ref{chainmap-sec} and the construction of the chain homotopy is given in Section \ref{chainhtpy-sec}.  
For ease of exposition, we first prove the following unfiltered invariance result, assuming the existence of nondegenerate dynamically separated contact forms.    Slight variations on these arguments yield analogous results on filtered cylindrical contact homology, which are explained in Section \ref{filtered-continuation}.

We define a \textbf{dynamically separated pair} $(\lambda, J)$ on a closed contact manifold $(M^3,\xi)$ to consist of a dynamically separated contact form $\lambda$ such that ker $\lambda=\xi$ and a generic $\lambda$-compatible almost complex structure $J$.  

\begin{theorem}\label{chainthm}
Let $(\lambda_1, J_1)$ and $(\lambda_2, J_2)$ be two nondegenerate dynamically separated pairs on a closed contact manifold $(M^3, \xi)$.  Then there exists a natural isomorphism
\begin{equation}\label{chainmap}
\Phi^{21}:CH_*(M, \lambda_1, J_1;\Q) \to CH_*(M, \lambda_2, J_2;\Q).
\end{equation}
If $(\A_3, J_3)$ is another nondegenerate dynamically separated pair, then
\[
\Phi^{31}=\Phi^{32} \circ \Phi^{21}, \ \ \ \Phi^{11}=\Phi^{22}=\Phi^{33}=\mbox{id}.
\]
\end{theorem}




\subsection{Numerology of the multiply covered}\label{quandaries}


This section gives lower bounds on the index of  multiply covered curves via the Riemann-Hurwitz theorem and Conley-Zehnder index formulas in dimension three,  extending methods previously used in \cite{HN1, jo1}.  To avoid cumbersome statements in this section, all propositions and lemmata are stated under the assumption that nondegenerate dynamically separated contact forms exist.  They hold for $L$-nondegenerate dynamically separated contact forms provided that the Reeb orbits comprising the asymptotics of the moduli spaces are all of action less than $L$.

First we recall the Riemann-Hurwitz Theorem.
\begin{theorem}[Hartshorne, Corollary IV.2.4]\label{RH}
Let $\varphi:\widetilde{\dot{\Sigma}} \to \ds$ be a compact $k$-fold cover of the Riemann surface $\ds$.  Then
\[
\chi(\widetilde{\dot{\Sigma}}) =  k\chi(\ds) - \sum_{p \in\widetilde{\dot{\Sigma}}} (e(p)-1),
\]
where $e(p)-1$ is the ramification index of $\varphi$ at $p$.
\end{theorem}

We will use $b$ to keep track of the number of branch points counted with multiplicity:
\begin{equation}\label{b}
b:=\sum_{p \in \widetilde{\dot{\Sigma}}} (e(p)-1).
\end{equation}
At unbranched points $p$ we have $e(p)-1=0$, thus for any $q \in \ds$, 
\[
\sum_{p \in \varphi^{-1}(q)}e(p)=k.
\]
The multiplicity of the Reeb orbits of the cover of an asymptotically cylindrical curve are determined by the monodromy with the local behavior of a curve near its punctures \cite{MiWh, Sief}.

\begin{remark}\em
In this section, we denote $\gamma^\ell_{+}$ to be the $\ell$-fold cover of a simple orbit $\gp$ and $\gamma^d_{-}$ the $d$-fold cover of a simple orbit $\gm$. Depending on the multiplicities of the orbits and existence of a covering map, the curve $u \in \calm (\glp;\gdm)$ may or may not be multiply covered.
\end{remark}

We obtain the following result.
\begin{proposition}\label{Hurwitztentacles}
Let $(M^3, \ker \A_\pm)$ be a closed contact manifold such that $\lambda_\pm$ are nondegenerate dynamically convex contact forms.  Let $(W,J)$ be a generic exact symplectic cobordism and $u \in \calm(\gp; \ga_0, ... \ga_s)$ be a somewhere injective curve.  Then any genus zero $k$-fold cover $\widetilde{\calc}$ of $\calc$  with 1 positive puncture must have  $1+ks+b$ negative punctures and satisfies
\begin{equation}\label{RHtentacleseqn}
{\mbox{\em ind}}(\widetilde{\calc}) \geq  k \cdot {\mbox{\em ind}}(\calc) -2k + 2b +2 .
\end{equation}
\end{proposition}

\begin{proof}
Recall that the index of the underlying curve $\calc$ is given by
\begin{equation}\label{underlying}
\begin{array}{lcl}
{\ind}(\calc) &=& s + \czm(\gp) - \displaystyle \sum_{i=0}^s\czm(\ga_i), \\
\end{array}
\end{equation}
and that Lemma \ref{almostlinear}  yields
\begin{equation}\label{iterateineq}
k \czm(\ga) - k +1 \leq \czm(\ga^k) \leq k \czm(\ga) + k -1.
\end{equation}

From the Riemann-Hurwitz Theorem if $\widetilde{u}$ has 1 positive puncture then it must have $1+ks+b$ negative punctures.   

Let $\delta_0,...,\delta_{ks+b}$ denote the Reeb orbits at which $\calc$ has negative ends; these are covers of $\ga_0,...,\ga_s$. Moreover,
\begin{equation}\label{waytogo}
\sum_{i=0}^{ks+b}\czm(\delta_i) \leq k \sum_{i=0}^s \czm(\ga_i) + (k(s+1)-(ks+b+1))
 \end{equation}
Then (\ref{iterateineq}) and (\ref{waytogo}) yield
\[
\begin{array}{lcl}
{\ind}(\widetilde{\calc}) &=& ks+b+ \czm(\ga_+^k) - \displaystyle \sum_{i=0}^{ks+b}\czm(\delta_i) \\
& \geq & ks  + b + (\displaystyle k\czm(\gp) -k +1) - k \sum_{i= 0}^s \czm(\ga_i) - k +b +1  \\
& = & k\left(  s + \czm(\gp) -\displaystyle \sum_{i=0}^s\czm(\ga_i)\right) -2k + 2b + 2 \\
&=& k \cdot {\ind}(\calc) -2k + 2b +2. \\
\end{array}
\]
\end{proof}

In some cases $ k \cdot {\ind}(\calc) -2k + 2b +2  \leq 0$, but this is not problematic because we will cap off $ks+b$ ends, each of which have index $\geq 2$; precise arguments appear in a subsequent series of lemmata. 

 However, we need to improve improve the preceding result when the underlying curve is a cylinder as well as when it is a pair of pants of index -1 in a one parameter family of exact symplectic cobordisms.  The following results for covers of cylinders in symplectizations are proven in \cite[Prop. 4.11, 4.12]{jo1}.
\begin{proposition}[covers of cylinders in a symplectization]\label{trivcyl}
Let $(M^3,\A)$ be a nondegenerate closed contact manifold and $\Jt$ a generic $\A$-compatible almost complex structure on $\R \times M$.  Any genus zero branched $k$-fold cover $\widetilde{\calc}$ of a nontrivial cylinder $\calc$  with 1 positive puncture must be an element of $\calm(\ga_+^k;\ga_-^{k_1},...\ga_-^{k_n})$ where $k:=k_1+...+k_n$.  Moreover,
\begin{itemize}
\item[\em (i)]if ${\mbox{\em ind}}(\calc) \geq 2$ then ${\mbox{\em ind}}(\widetilde{\calc}) \geq 2n$;
\item[\em (ii)] if ${\mbox{\em ind}}(\calc) = 1$ with $\gp$ hyperbolic then ${\mbox{\em ind}}(\widetilde{\calc}) \geq 2n-1$;
\item[\em (iii)] if ${\mbox{\em ind}}(\calc) = 1$ with $\gm$ hyperbolic then ${\mbox{\em ind}}(\widetilde{\calc}) \geq n$.
\end{itemize}

Any genus zero  $k$-fold cover $\widetilde{\calc}_0$ of of a trival cylinder $\calc_0$  with 1 positive puncture must either be an element of $\calm(\ga^k;\ga^k)$ or $\calm(\ga^k;\ga^{k_1},...\ga^{k_n})$ where $k:=k_1+...+k_n$.  In the former case when we have 
\[
{\mbox{\em ind}}(\widetilde{\calc}_0)=0 .
\]
In the latter case when $\widetilde{\calc}_0 \in \calm(\ga^k;\ga^{k_1},...\ga^{k_n})$ we have
\begin{equation}\label{RHtrivial}
{\mbox{\em ind}}(\widetilde{\calc}_0) \geq 0.
\end{equation}
\end{proposition}

The next result we need is concerns covers of cylinders in exact symplectic cobordisms between nondegenerate dynamically separated pairs.  Since we are only concerned with those covers which could appear in a building, the positive end and one of the negative ends must be in the same free homotopy class while the remaining negative ends  must be  contractible.  We denote the free homotopy class of a Reeb orbit $\gamma$ by $[\gamma]$.

\begin{proposition}\label{coveredcyl}
Let $(W,J)$ be a generic exact symplectic cobordism between nondegenerate dynamically separated contact forms on a closed contact 3-manifold.  Let $\calc \in \calm(\ga_+; \ga_-)$ be a somewhere injective.  Then any genus zero  $k$-fold cover with $n-1$ contractible ends $\widetilde{\calc} \in \calm(\ga_+^k;\ga^{k_1}_-,...\ga^{k_n}_-)$ satisfies $k:=k_1+...+k_n$.  If $[\ga_+^k]=[\ga_-^{k_1}]$ and $0=[\ga_-^{k_2}]=...=[\ga_-^{k_n}]$ then 
\begin{equation}\label{RHtrivial}
{\mbox{\em ind}}(\widetilde{\calc}) \geq  n+1 - \sum_{i=2}^n \czm(\gamma_-^{k_i}).
\end{equation}
\end{proposition}

\begin{proof}
Since $[\ga_+^k] = [\gamma_-^{k_1}]$ and $[\gamma_+]=[\gamma_-]$ the definition of dynamically separated forces 
\[
\czm(\ga_+^k) - (\gamma_-^{k_1}) \geq -2+ 4 =2,
\]
and the result follows from the Fredholm index formula.
\end{proof}

We improve the preceding result when the underlying somewhere injective curve is a pair of pants of index -1 in a cobordism.  We first treat the case when there are no branch points.

\begin{proposition}\label{unbranchedpants}
Let $(W, \mathbb{J}=\{J^\tau\}_{\tau \in [0,1]})_{\tau \in [0,1]}$ be a one parameter family  of generic exact symplectic cobordisms between nondegenerate dynamically separated pairs.   Let $\calc \in \widehat{\mathcal{M}}^{J_\tau}(\gp; \ga_0, \ga_1)$ be somewhere injective with $\mbox{\em ind}(u)=-1$ for some $\tau \in (0,1).$  Then any genus zero unbranched $k$-fold cover $\widetilde{\calc}$ of $\calc$  with 1 positive puncture must be an element of $ \widehat{\mathcal{M}}^{{J_\tau}}(\ga_+^k;\ga_0^{k},\underbrace{\ga_1,...,\ga_1}_\text{$k$ copies})$ or $ \widehat{\mathcal{M}}^{{J_\tau}}(\ga_+^k;\underbrace{\ga_0,...,\ga_0}_\text{$k$ copies},\ga_1^k)$.

\noindent In the former case,
\[
\mbox{\em ind}(\widetilde{u}) \geq -2k+1 + k \czm(\ga_0) - \czm(\ga_0^k).
\]
\end{proposition}

\begin{proof}
{As a preliminary step we first explain why being an unbranched cover forces the configuration of asymptotic limits at negative punctures as claimed.  
From Riemann-Hurwitz, we have that number of punctures of $\dot{\Sigma}$ is $2+k$. To understand the multiplicities of the negative asymptotic limits we note that these are determined by the monodromy of the local behavior of a curve near its punctures which are in turn governed by the monodromy of the covering.  }

{We can model the $k$-fold cover $\varphi : \widetilde{\dot{\Sigma}} \to \dot{\Sigma}$ in terms of the cover of a closed unit disk $ p: D^2 \to D^2$ by $z \mapsto z^k$, with interior marked points corresponding to the negative asymptotics.  The boundary of the disk will correspond to the positive asymptotic $\gamma_+$ in the target and $\gamma_+^k$ in the preimage.  If we fix the origin to become the puncture corresponding to $\gamma_0$ or $\gamma_1$ in the target then we obtain $\gamma_0^k$ or $\gamma_1^k$ in the preimage respectively.    
Any other point in the interior of $D^2 \setminus \mathbf{0}$ will have $k$ preimages, which corresponds to the $k$-copies of $\gamma_1$ or $\gamma_0$, corresponding to the respective asymptotic assignment of $\gamma_0$ or $\gamma_1$ to the origin.}


We have
\[
\mbox{ ind}(u)  = 1 + \czm(\ga_+) - \czm(\ga_0)-\czm(\ga_1) = -1,
\]
thus
\[
\czm(\ga_1)=2+\czm(\ga_+)-\czm(\ga_0).
\]
Then in the former case,
\[
\begin{array}{lcl}
\mbox{ ind}(\widetilde{u}) &=& k + \czm(\ga_+^k) -\czm(\ga_0^k) - k\czm(\ga_1) \\
 &=& - k + \czm(\ga_+^k)  - \czm(\ga_0^k) + k(\czm(\ga_0)-\czm(\ga_+)) \\
 &\geq& -k + k\czm(\gp) - k + 1-  \czm(\ga_0^k) + k(\czm(\ga_0)-\czm(\ga_+)) \\
 &=& -2k +1 + k \czm(\ga_0) - \czm(\ga_0^k).
 \end{array}
\]
\end{proof}

We refine Proposition \ref{unbranchedpants} when $k=2$ and $\czm(\gamma_0)=1$.
 
 \begin{proposition}\label{unpant1}
 Let $\calc \in  \widehat{\mathcal{M}}^{J_\tau}(\gp; \ga_0, \ga_1)$ be somewhere injective with $\mbox{ \em ind}(u)=-1$ and $\czm(\gamma_0)=1$ for some $\tau \in (0,1).$ Then either $\gamma_+$ hyperbolic or $\gamma_1$ hyperbolic.  For an unbranched cover $\widetilde{\calc} \in  \widehat{\mathcal{M}}^{{J_\tau}}(\ga_+^2;\ga_0^{2},{\ga_1,\ga_1})$, if $\gamma_+$ is hyperbolic then
\[
\mbox{\em ind}(\widetilde{u}) \geq -3
\]
otherwise if $\gamma_1$ is hyperbolic {then}
\[
{\mbox{\em ind}(\widetilde{u}) \geq -2}.
\]
\end{proposition}

\begin{proof}
Since $\ind(u) = -1 $ and $\czm(\ga_0)=1$ then 
\[
\czm(\gp)-\czm(\ga_1)=-1,
\]
thus one of $\gamma_+$ and $\gamma_1$ is hyperbolic.  Moreover, 
\[
\czm(\ga_0^2) \leq 2 \czm(\ga_0) + 2 -1,
\]
thus $-\czm(\ga_0^2) \geq -3.$

\noindent Case 1: If $\gp$ is hyperbolic then  $\czm(\ga_+^2) = 2\czm(\gp)$ and
\[
\begin{array}{lcl}
\ind(\widetilde{u}) &=& 2 + \czm(\ga_+^2) - \czm(\ga_0^2) - 2\czm(\ga_1) \\
&\geq & -1 + 2\czm(\gp) - 2 \czm(\ga_1) \\
&= & -3. \\
\end{array}
\]

\noindent Case 2: If $\ga_1$ is hyperbolic then $\czm(\ga_1^2) = 2\czm(\ga_1)$. In combination with Proposition \ref{almostlinear},
\[
{\begin{array}{lcl}
\ind(\widetilde{u}) &=& 2 + \czm(\ga_+^2) - \czm(\ga_0^2) - 2\czm(\ga_1) \\
&\geq & 2 + 2\czm(\ga_+) + 2 -1 - \czm(\ga_0^2) - 2\czm(\ga_1) \\
& = & 2 -2  + 2 -1 - \czm(\ga_0^2) \\
& \geq & 1 - 3. \\
\end{array}}
\]
as desired.
\end{proof}

Next we consider branched covers of pants in cobordisms.

\begin{proposition}\label{branchedpants}
  Let $\calc \in  \widehat{\mathcal{M}}^{J_\tau}(\gp; \ga_0, \ga_-)$ be somewhere injective with $\mbox{ \em ind}(u)=-1$.  Then any genus zero branched $k$-fold cover $\widetilde{\calc}$ of $\calc$  with 1 positive puncture and $b$ branch points must be an element of  $ \widehat{\mathcal{M}}^{J_\tau}(\ga_+^k;\ga_0^{k_0},,...,\gamma_0^{k_b}, \underbrace{\ga_1,...,\ga_1}_\text{$k$ copies})$ or  $ \widehat{\mathcal{M}}^{J_\tau}(\ga_+^k;\underbrace{\ga_0,...,\ga_0}_\text{$k$ copies},\ga_1^{k_0},...,\gamma_1^{k_b})$ with  $k_0+...+k_b=k$.
  
\noindent In the former case,
\[
\mbox{\em ind}(\widetilde{u}) \geq -2k + b+1 + k\czm(\ga_0) - \sum_{i=0}^b\czm(\ga_0^{k_i}).
\]

\end{proposition} 

\begin{proof}
Since $\ind(u)=-1$, 
\[
-\czm(\ga_1) = -2 + \czm(\ga_0) - \czm(\gp), 
\]
and
\[
\begin{array}{lcl}
\ind(\widetilde{u}) &=& k+ b + \czm(\ga_+^k) - d\czm(\ga_1) - \sum_{i=0}^b \czm(\ga_0^{k_i}) \\
&=& -k+b + \czm(\ga_+^k)  +k\czm(\ga_0) - k\czm(\gp) - \sum_{i=0}^b \czm(\ga_0^{k_i}) \\
&\geq& -k + b + k\czm(\gp) - k +1 + k\czm(\ga_0) - k\czm(\gp) - \sum_{i=0}^b \czm(\ga_0^{k_i})  \\
&=& -2k + b+1 + k\czm(\ga_0) - \sum_{i=0}^b \czm(\ga_0^{k_i}). \\
\end{array}
\]

\end{proof}

\begin{remark}\label{rembp} \em
If $\widetilde{u}$ is a branched 2-fold cover of $\calc \in  \widehat{\mathcal{M}}^{{J}}(\gp; \ga_0, \ga_1)$ then $\widetilde{u} \in  \widehat{\mathcal{M}}^{{J}}(\gamma_+^2; \gamma_0,\gamma_0, \gamma_1,\gamma_1)$.
\end{remark}

The proof of Theorem \ref{chainthm} relies on the following series of inductive lemmata utilizing the above numerics. These results will allow us to exclude complicated compactifications as in Figure \ref{egads}.  Before proceeding, we recall the definition of a pseudoholomorphic building from \cite{BEHWZ}, which we adapt to our setting in which all curves and their limits are non-nodal and unmarked.

 \begin{figure}[h!]
  \centering
    \includegraphics[scale=2]{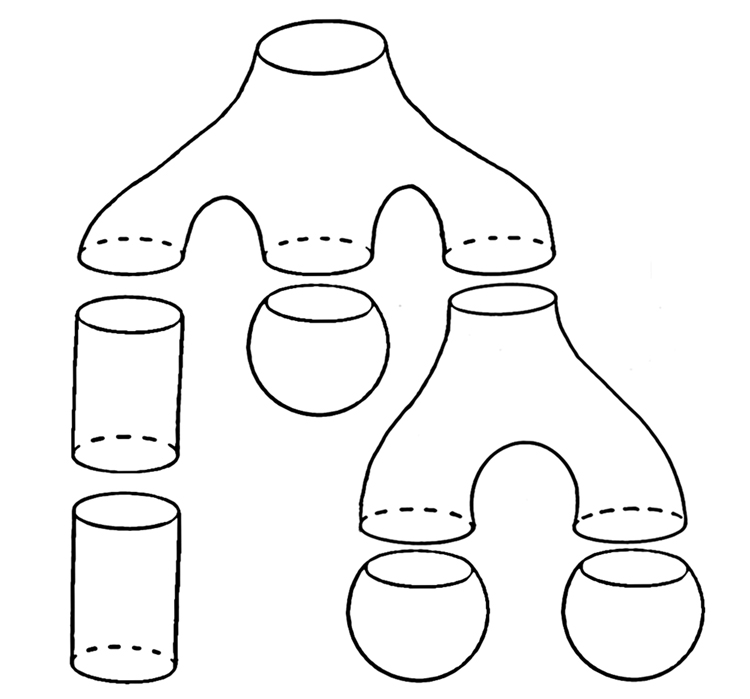}
\caption{A catastrophe of compactness best avoided.}
\label{egads}
\end{figure}

\begin{definition}\label{building} \em
Any asymptotically cylindrical curve $u_i=[\Sigma_i, j_i, \g_i:=\g^+_i \sqcup \g^+_i, u_i]$, with $\Sigma_i$ possibly disconnected, is said to be a \textbf{height-1 non-nodal building}, or height-1 building for short.  
Assuming there are bijections $\Psi_i: \to \g_i^- \to \g_{i+1}^+$ between the negative punctures of one curve and the positive punctures of the curve next in the sequence, a \textbf{height-$k$ non-nodal building} consists of a sequence $(u_1,...u_m)$ of $m$ height-1, non-nodal buildings and $(\Psi_1,...\Psi_{m-1})$, provided the punctures identified via $\Psi_i$ have the same asymptotic limit. 
\end{definition}

\begin{remark}\label{nontrivbuilding}\em
Throughout the following lemmata we assume that each level $u_i$ of the building $(u_1,..,u_m)$  contains at least one nontrivial component, i.e. a component which is neither a trivial cylinder nor a constant map.   
\end{remark}

These proofs are done via induction on the number of levels of $\calb:=(u_1,...,u_m)$, where the $u_i$ are levels of $\calb$ in decreasing order, e.g. $u_1$ is the top level.  For any $\calb$ with one positive end asymptotic to the Reeb orbit $\gp$ and no negative ends, 
\begin{equation}\label{helpfuless1}
\ind(\calb) = \czm(\gp) -1.
\end{equation}

\begin{remark}\label{nontrivbuilding}\em
Throughout the following lemmata in Sections \ref{obs-chain}, \ref{obs-htpy} we assume that each level $u_i$ of the building $(u_1,..,u_n)$  contains at least one nontrivial component, i.e. a component which is neither a trivial cylinder nor a constant map.   
\end{remark}

\subsubsection{Excluding obstructions to the chain map}\label{obs-chain}

Unless otherwise specified, we assume we are working in a generic exact symplectic cobordism of contactomorphic nondegenerate dynamically separated pairs.  This means that $W$ is an exact symplectic cobordism between dynamically separated contact forms and $J$ is a generic cobordism compatible almost complex structure. We first recall the following result from \cite{jo1}.  

\begin{lemma}[Lemma 4.15 \cite{jo1}]\label{lemma0}
Let $\calb:=(u_1,...u_n)$ be a genus 0 building with one positive contractible end, and no negative ends in a symplectization.  Then ${\mbox{\em ind}}(\calb) \geq 2,$
with equality if and only if $\calb$ consists only of a pseudoholomorphic plane.
\end{lemma}

Next, we prove the following result.

\begin{proposition}\label{0compact}
Let $(W,J)$ be a generic exact dynamically separated symplectic cobordism.  If $\czm(\gp) = \czm(\gm)$ then the space $\calm(\gp, \gm)$ is a compact 0-dimensional manifold.
\end{proposition}

\begin{proof}
 By Lemmata \ref{index-cyl-pos} and \ref{lem:at} all nontrivial cylinders $u \in \widehat{\M}^{J_\circ}(x,y)$ in a symplectization $(\R \times M, J_\circ)$ equipped with a generic $\lambda_\circ$-compatible almost complex structure $J_\circ$ satisfy ind$(u) \geq 1$ and are cut out transversely.  Thus for moduli spaces of nontrivial cylinders,
\[
  \left\{
  \begin{array}{ll}
\widehat{\mathcal{M}}^{J_+}(x_i, x_{i+1}) = \emptyset & \mbox{if } \czm(x_i) - \czm(x_{i+1}) \leq 0; \\
\widehat{\mathcal{M}}^{J_-}(y_i, y_{i+1}) = \emptyset & \mbox{if } \czm(y_i) - \czm(y_{i+1}) \leq 0. \\
\end{array} \right.
\]
Thus 
\[
\czm(x_i) - \czm(x_{i+1}) > 0 \mbox{ and } \czm(y_i) - \czm(y_{i+1}) > 0.
\]
The dynamically separated assumption ensures that there are no unbranched covers of cylinders $u$ with ind$(u) <0$.  Proposition \ref{cylcoverssep} and automatic transversality ensure that all index 0 cylinders are regular.  As a result, when $\czm(\gp)=\czm(\gm)$, the moduli space $\calm(\gp,\gm)$ is cut out transversely and compact.  

 To prove that $\widehat{\mathcal{M}}^J(\gp,\gm)$ is in fact a manifold near $u$, we need to further show that the order $k$ group of deck transformations of $u$ over a somewhere injective cylinder $\overline{u}$ acts trivially on $\mbox{Ker}(D_u)$. To do so, it suffices to show that every element of $\mbox{Ker}(D_u)$ is pulled back from an element of $\mbox{Ker}(D_{\overline{u}})$, thus it will be enough to show that $\mbox{ind}(u)=\mbox{ind}(\overline{u})$.  Under the dynamically separated assumption this equality holds.
 \end{proof}

\begin{lemma}\label{lemma1}
Let $\calb:=(u_1,...u_n)$ be a genus 0 building with one positive contractible end, and no negative ends in an generic exact dynamically separated symplectic cobordism.  Then 
${\mbox{\em ind}}(\calb) \geq 2,$ with equality if $\calb$ consists of one holomorphic plane or an index 0 cylinder $u\in \widehat{\M}^J(\gp,\gm)$ and a plane.
\end{lemma}

\begin{proof}
This proof will be done via induction on the number of levels of $\calb:=(u_1,...,u_n)$, where the $u_i$ are levels of $\calb$ in decreasing order, e.g. $u_1$ is the top level.  For any $\calb$ with one positive end asymptotic to the Reeb orbit $\gp$ and no negative ends, 
\begin{equation}\label{helpfuless}
\hind(\calb) = \czm(\gp) -1.
\end{equation}

If $\calb$ consists of only one level we are done since $\czm(\gp)\geq3$ by the dynamically separated hypothesis.

Suppose $n>1$ and that Lemma is true for buildings of height $n-1$.  We need to show that $\hind(\calb) > 2$. The building $(u_2,...u_n)$ is the disjoint union of $\ell$ genus 0 buildings, each having one positive end at each of the negative ends of $u_1$ and no negative ends.  By the inductive hypothesis we have
\[
\hind(\calb) \geq \hind(u_1) + 2\ell.
\]
Thus we must show that
\begin{equation}\label{lemma1ineq}
\hind(u_1) + 2\ell \geq 2.  
\end{equation}

If $u_1$ is somewhere injective then $\hind(u_1) \geq 0$.  If $u_1$ is the unbranched cover of a cylinder then $\hind(u_1) \geq 0$ and $u_1$ is regular by Proposition \ref{cylcoverssep}.  By Proposition \ref{0compact} there cannot be additional index 0 cylindrical levels.  

If $u_1$ is the $k$-fold cover with of a somewhere injective curve $\calc \in \calm(\gp;\ga_0,...\ga_s)$ with $b$ branch points\footnote{Note $b$ could be 0 since the result holds if $u_1$ doesn't have any branch points.} counted with multiplicity then Proposition \ref{Hurwitztentacles} yields
\[
\hind(u_1) \geq 2-2k + 2b,
\]
with $\ell=1+ ks +b$.   Thus $\hind(u_1) + 2\ell \geq 4+4b+2k(s-1) > 2$ for $s \geq 1$. If $s=0$ it remains to consider $u_1 \in \mhat(\ga_+^k; \ga_-^{k_1},...,\ga_-^{k_n})$ with $k_1 + ... + k_n =k$. Moreover,
\[
\hind(\calb) = \hind(u_1) + \sum_{i=1}^n |\ga_-^{k_i}|,
\]
and, in combination with Proposition \ref{coveredcyl}, we obtain
\[
 \hind(\calb) \geq n+1 - \sum_{i=2}^n \czm(\gamma_-^{k_i}) + \sum_{i=1}^n |\ga_-^{k_i}| \geq 4.
\]

\end{proof}


Building upon this theme we continue with the following lemma.
\begin{lemma}\label{lemma2}
Let $\calb:=(u_1,...u_n)$ be a  genus 0 building with one positive end and one negative end in a cobordism. Then ${\mbox{\em ind}}(\calb) \geq 0,$
with equality if and only if $\calb$ consists of only one cylinder.
\end{lemma}
\begin{proof}
As before the proof will be done via induction on the number of levels of the building $\calb:=(u_1,...,u_n)$, where the $u_i$ are levels of $\calb$ in decreasing order, e.g. $u_1$ is the top level.  {If $\calb$ consists of only one level then the proof is complete by Proposition \ref{0compact}. }


Suppose there is more than one level.   Call the top level $u_1$ and assume that the lemma is true for buildings of height $n-1$.  We need to show that $\hind(\calb) > 1$. The building $(u_2,...u_n)$ is the disjoint union of $\ell$ genus 0 buildings, each consisting of one positive end at each of the negative ends of $u_1$ and all but one, say $\calb_1$ having no negative ends.  This exceptional building, $\calb_1$, has one positive and one negative end.  By the inductive hypothesis we have
\[
\hind(\calb) \geq \hind(u_1) + \hind(\calb_1) + 2(\ell-1).
\]
Thus we must show that
\begin{equation}\label{lemma2ineq}
\hind(u_1) + 2(\ell-1) \geq 1.  
\end{equation}

If $u_1$ is somewhere injective then $\hind(u_1) \geq 0$.  If $u_1$ is the unbranched cover of a cylinder then $\hind(u_1) \geq 0$ and $u_1$ is regular by Proposition \ref{cylcoverssep}.  By Proposition \ref{0compact} there cannot be additional index 0 cylindrical levels.  

If $u_1$ is the cover of a somewhere injective curve $\calc \in \calm(\gp;\ga_0,...\ga_s)$,  by Proposition \ref{Hurwitztentacles}
\[
\hind(u_1) \geq 2-2k + 2b,
\]
with $\ell = 1+ks+b$. Thus for $s \geq 1$, $\hind(u_1) + 2(\ell-1)  \geq 2 + 2k(s-1) + 4b \geq 2$.  If $s=0$ then $u_1 \in \mhat(\ga_+^k; \ga_-^{k_1},...,\ga_-^{k_n})$ with $k_1 + ... + k_n =k$. Without loss of generality, 
\[
\hind(\calb) = \hind(u_1) + \sum_{i=2}^n |\ga_-^{k_i}|.
\]
In combination with Proposition \ref{coveredcyl} we obtain
\[
 \hind(\calb) \geq n+1 - \sum_{i=2}^n \czm(\gamma_-^{k_i}) + \sum_{i=2}^n |\ga_-^{k_i}| \geq 2.
\]
\end{proof}

We recall the following lemma in regards to buildings in symplectizations.

\begin{lemma}[Lemma 4.19 \cite{jo1}]\label{lemma3}
Let $\calb$ be a genus 0 building with one positive end and one negative end associated to a nondgenerate dynamically separated contact form in a symplectization $(\R \times M, J_\circ)$, with ${\mbox{\em ind}}(\calb) =2.$ Then $\calb$ is one of the following types,
\begin{enumerate}
\item[\em {(i)}] An unbroken cylinder of index 2;
\item[\em {(ii)}] A once broken cylinder given by a pair of cylinders, each of index 1, $(u,v) \in  \widehat{\mathcal{M}}^{J_\circ}(x, y) \times \widehat{\mathcal{M}}^{J_\circ}(y, z)$, where $\czm(x)-\czm(y)=1$.
\end{enumerate}

\end{lemma}

Stacking the above lemmata together we obtain the following result for buildings in exact symplectic cobordisms of nondegenerate dynamically separated pairs.
\begin{lemma}\label{bldgchain}
Let $\calb$ be a genus 0 building with one positive end and one negative end in a generic exact symplectic cobordism of nondegenerate dynamically separated pairs.  If ${\mbox{\em ind}}(\calb) =1,$ then $\calb$ is one of the following types,
\begin{enumerate}
\item[\em {(i)}] An unbroken cylinder of index 1 in the cobordism;
\item[\em {(ii)}] A once broken cylinder of one of the following forms:
\[
 (u, v) \in  \widehat{\mathcal{M}}^{J_+}(x, y) \times \widehat{\mathcal{M}}^J(y, z)\]
with $\czm(x)-\czm(y)=1, \ \czm(y)=\czm(z) $ or 
\[  
 (u,v) \in \widehat{\mathcal{M}}^J(x, y) \times \widehat{\mathcal{M}}^{J_-}(y, z)
\]
 with $\czm(y)=\czm(x), \ \czm(y)-\czm(z)=1$.
\end{enumerate}
\end{lemma}

\begin{proof}
 Since the index of a building is additive with respect to its components the results of Lemmata \ref{lemma0} - \ref{lemma3} and Proposition \ref{0compact} imply that the only possible configurations for a building $\calb$ of index 1 are those described in (i)-(ii).
 \end{proof}


\subsubsection{Excluding obstructions to the chain homotopy}\label{obs-htpy}

In this section we consider parametric moduli spaces of pseudoholomorphic curves.  As before, we take $(W, \mathbb{J}=\{J_\tau\}_{\tau \in [0,1]})$ to  be a generic one parameter family (e.g. homotopy) of exact symplectic cobordisms between nondegenerate dynamically separated pairs.   As in Section \ref{jback}, the parametric moduli space is defined by
\[
\widehat{\M}^\mathbb{J}(\gp,\gm) = \{ (\tau,u) \ \arrowvert  \ \tau \in [0,1], \ u\in \widehat{\M}^{J_\tau}(\gp, \gm) \},
\]
and is of dimension
\[
\begin{array}{lcl}
\dim \widehat{\M}^\mathbb{J}(\gp,\gm) &=&  \czm(\gp) - \czm(\gm) +1 \\
&=& \mbox{ind}(u) + 1\\
\end{array}
\]
Our first result is the analogous statement to Proposition \ref{0compact} for 0-dimensional parametric moduli spaces.

\begin{proposition}\label{0compacthtpy}
 If $\mathbb{J}=\{J_\tau\}_{\tau \in [0,1]}$ is a generic one parameter family of almost complex structures associated to an exact dynamically separated symplectic cobordism $W$ then the 0-dimensional parametric moduli space $\widehat{\M}^\mathbb{J}(\gp, \gm)$ is a compact 0-dimensional manifold.
\end{proposition}

\begin{proof}
 By Proposition \ref{0compact} we have for any generic exact symplectic cobordism $(W,J)$ between two dynamically separated contact forms that any cylinder $u \in \widehat{\M}^J(\gp,\gm)$ satisfies ind$(u) \geq 0$ and is regular.  Thus for moduli spaces of cylinders
\[
  \left\{
  \begin{array}{ll}
\widehat{\mathcal{M}}^{J_{\tau_0}}(x_i, x_{i+1}) = \emptyset & \mbox{if } \czm(x_i) - \czm(x_{i+1}) < 0; \\
\widehat{\mathcal{M}}^{J_{\tau_1}}(y_i, y_{i+1}) = \emptyset & \mbox{if } \czm(y_i) - \czm(y_{i+1}) < 0. \\
\end{array} \right.
\]
hence 
\[
\czm(x_i) - \czm(x_{i+1}) \geq 0 \mbox{ and } \czm(y_i) - \czm(y_{i+1}) \geq 0.
\]
The dynamically separated assumption allows us to conclude that there are no unbranched covers of cylinders $u$ with ind$(u) <-1$.  Proposition \ref{cylcoverssep-par} and the related results in Section \ref{reg-sec} allow us to conclude that we know that there  and that all index -1, 0, and 1 cylinders are regular.  Since the moduli spaces associated to the homotopy at $\tau=0$ and $\tau=1$ are regular there cannot be any cylinders with index $-1$, e.g. when $\czm(\gp)=\czm(\gm)-1$ for $\tau=0$ and $\tau=1$.  However, in a generic 1-parameter family, such cylinders do occur for isolated parameter values of $\tau$. 

As a result, the 0-dimensional parametric moduli space $\widehat{\M}^{\mathbb{J}}(\gp,\gm)$ is cut out transversely and compact. The manifold structure follows from the same arguments given in the proof of Proposition \ref{0compact}.
 \end{proof}

\begin{remark}\em
The counts of cylinders with index $-1$, e.g. when $\czm(\gp)=\czm(\gm)-1$ which occur for isolated parameter values of $\tau \in (0,1)$  give rise to the chain homotopy equivalence.  
\end{remark}

\begin{lemma}\label{lemma1htpy}
Let $\calb:=(u_1,...u_n)$ be a genus 0 building with 1 positive contractible end, and no negative ends in $(W, \mathbb{J})$.  Then 
\[
{\mbox{\em ind}}(\calb) = \sum_{i=1}^n  {\mbox{\em ind}}(u_i) \ {+ \ n } \geq 2,
\]
with equality if and only if there {are} isolated parameter {values} $\tau \in (0,1)$ for which there exists $u_i \in \widehat{\mathcal{M}}^{J_\tau}(\gp, \gm)$ with ${\mbox{\em ind}}(u_i)=-1$.
\end{lemma}

\begin{proof}
This proof will be done via induction on the number of levels of $\calb:=(u_1,...,u_n)$, where the $u_i$ are levels of $\calb$ in decreasing order, e.g. $u_1$ is the top level.  For any $\calb$ with one positive end asymptotic to the Reeb orbit $\gp$ and no negative ends, 
\begin{equation}\label{helpfuless}
\hind(\calb) = \czm(\gp) -1 + 1.
\end{equation}

If $\calb$ consists of only one level then we are done since $\czm(\gp)\geq3$ by the  dynamically separated hypothesis.

Suppose $n>1$ and that Lemma is true for buildings of height $n-1$.  We need to show that $\hind(\calb) \geq 2$ with equality if and only if there are isolated parameter values $\tau \in (0,1)$ for which there exists $u_i \in \widehat{\mathcal{M}}^{J_\tau}(\gp, \gm)$ with ${\mbox{ ind}}(u_i)=-1$. The building $(u_2,...u_n)$ is the disjoint union of $\ell$ genus 0 buildings, each having one positive end at each of the negative ends of $u_1$ and no negative ends.  By the inductive hypothesis we have
\[
\hind(\calb) \geq \hind(u_1) +1 + 2\ell.
\]
Thus we must show that
\begin{equation}\label{lemma1ineq}
\hind(u_1) +1 + 2\ell \geq 2.  
\end{equation}
If $u_1$ is somewhere injective and $\hind(u_1) \geq 0$ then we are done.  If $u_1$ is the unbranched cover of a cylinder of $\hind =0$ then $\hind(u_1) = 0$ and $u_1$ is regular by Proposition \ref{cylcoverssep}.  By Proposition \ref{0compact} we know  that for any generic exact dynamically separated symplectic cobordism $(W,J)$, if $\czm(\gp) = \czm(\gm)$ then the space $\widehat{\mathcal{M}}^{J_\tau}(\gp, \gm)$ is compact.  By Proposition \ref{0compacthtpy} there are only isolated {values} for $\tau \in (0,1)$ in which  $\hind(u_1) = -1$, and moreover, such a $u_1$ is regular. There cannot be additional cylindrical levels as a result of Section \ref{obs-chain}. 


If $u_1$ is the $k$-fold cover with of a somewhere injective curve $\calc \in \widehat{\mathcal{M}}^{J_\tau}(\gp;\ga_0,...\ga_s)$ with $b$ branch points\footnote{Note $b$ could be 0 since the result holds if $u_1$ doesn't have any branch points.} counted with multiplicity then Proposition \ref{Hurwitztentacles} yields
\[
\hind(u_1) \geq 2-3k + 2b,
\]
with $\ell=1+ ks +b$.   Thus 
\[
\hind(u_1) + 2\ell \geq 4+4b+k(2s-3) > 2
\]
 so if $s\geq2$ we are ok. 

If $s=0$ then $u_1 \in \widehat{\mathcal{M}}^{J_\tau}(\ga_+^k; \ga_-^{k_1},...,\ga_-^{k_n})$ with $k_1 + ... + k_n =k$. Moreover
\[
\hind(\calb) -1 = \hind(u_1) + \sum_{i=1}^n |\ga_-^{k_i}|,
\]
in combination with Proposition \ref{coveredcyl} yields
\[
 \hind(\calb) -1 \geq n+1 - \sum_{i=2}^n \czm(\gamma_-^{k_i}) + \sum_{i=1}^n |\ga_-^{k_i}| \geq 4.
\]

If $s=1$ we are ok by our pant propositions as follows.  Note that $\czm(\ga_0) \geq 1$ and since $\ga_1$ is contractible, $\czm(\ga_1) \geq 3$.   If $u_1$ is an unbranched $k$-fold cover Proposition \ref{unbranchedpants} yields
\[
\begin{array}{lcl}
\hind(\calb) -1 &=& \hind(u_1) + \czm(\ga_0^{k}) - 1 + k \czm(\ga_1) - k \\
 &\geq& -3k+ k \czm(\ga_0)  +  k \czm(\ga_1)  \\
 & \geq &  k. \\
\end{array}
\]
 If $u_1$ is branched $k$-fold cover Proposition \ref{branchedpants} yields
\[
\begin{array}{lcl}
\hind(\calb) -1 &=& \hind(u_1) + \sum_{i=0}^b\czm(\ga_0^{k_i}) - (b+1) + k \czm(\ga_1) - k \\
 &\geq& -3k + k \czm(\ga_0)  + k \czm(\ga_1) \\
 & \geq &  k. \\
\end{array}
\]
\end{proof}

Building again upon this theme we continue with the following lemma.
\begin{lemma}\label{lemma2htpy}
Let $\calb:=(u_1,...u_n)$ be a  genus 0 building with one positive end and one negative end in a cobordism. Then ${\mbox{\em ind}}(\calb) \geq 1,$ with equality if and only if one of the following holds
\begin{enumerate}
\item[\em {(i)}]  $\calb$ consists of a pair $(\tau, u)$ with $\tau \in [0,1]$ and $u $ is a cylinder with $\mbox{\em ind}(u)=0${\em ;}
\item[\em {(ii)}]  There {are} isolated parameter {values} $\tau \in (0,1)$ for which there exists $u_i \in \widehat{\mathcal{M}}^{J_\tau}(\gp, \gm)$ with ${\mbox{\em ind}}(u_i)=-1$.  
\end{enumerate}
\end{lemma}
\begin{proof}
As before the proof will be done via induction on the number of levels of the building $\calb:=(u_1,...,u_n)$, where the $u_i$ are levels of $\calb$ in decreasing order, e.g. $u_1$ is the top level. 

 
 If $u_1$ is somewhere injective and $\hind(u_1) \geq 0$ then we are done.  If $u_1$ is the unbranched cover of a cylinder of $\hind =0$ then $\hind(u_1) = 0$ and $u_1$ is regular by Proposition \ref{cylcoverssep}.  By Proposition \ref{0compact} we know  that for any generic exact dynamically separated symplectic cobordism $(W,J)$, if $\czm(\gp) = \czm(\gm)$ then the space $\widehat{\mathcal{M}}^{J_\tau}(\gp, \gm)$ is compact.  By Proposition \ref{0compacthtpy} there are only isolated {values} for $\tau \in (0,1)$ in which  $\hind(u_1) = -1$, and moreover, such a  $u_1$ is regular. There cannot be additional cylindrical levels as a result of Section \ref{obs-chain}.

Suppose there is more than one level.   Call the top level $u_1$ and assume that the lemma is true for buildings of height $n-1$.  We need to show that $\hind(\calb) > 1$. The building $(u_2,...u_n)$ is the disjoint union of $\ell$ genus 0 buildings, each consisting of one positive end at each of the negative ends of $u_1$ and all but one, say $\calb_1$ having no negative ends.  This exceptional building, $\calb_1$, has one positive and one negative end.  By the inductive hypothesis we have
\begin{equation}\label{bound1}
\hind(\calb) -1 \geq \hind(u_1) + \hind(\calb_1) + 2(\ell-1).
\end{equation}
Thus it suffices to show that
\begin{equation}\label{lemma2ineq}
\hind(u_1) + 2(\ell-1) \geq 1.  
\end{equation}


If $u_1$ is the cover of a somewhere injective curve $\calc \in \widehat{\mathcal{M}}^{J_\tau}(\gp;\ga_0,...\ga_s)$,  by Proposition \ref{Hurwitztentacles}
\[
\hind(u_1) \geq 2-3k + 2b,
\]
with $\ell = 1+ks+b$.    Thus 
\[
\hind(u_1) + 2\ell -2  \geq 2+4b+k(2s-3) > 2
\]
 so if $s\geq2$ we are ok.

If $s=0$ then $u_1 \in \widehat{\mathcal{M}}^{J_\tau}(\ga_+^k; \ga_-^{k_1},...,\ga_-^{k_n})$ with $k_1 + ... + k_n =k$. Without loss of generality,
\[
\hind(\calb) - 1 = \hind(u_1) + \sum_{i=2}^n |\ga_-^{k_i}|.
\]
In combination with Proposition \ref{coveredcyl} we obtain
\[
 \hind(\calb) -1 \geq n+1 - \sum_{i=2}^n \czm(\gamma_-^{k_i}) + \sum_{i=1}^n |\ga_-^{k_i}| \geq 4.
\]

If $s=1$ we are ok by our pant propositions as follows.     \\

\noindent \textbf{Case (1):} Let $u_1$ be an unbranched $k$-fold cover.  Without loss of generality, denote $u_1 \in \widehat{\mathcal{M}}^{J_\tau}(\ga_+^k;\ga_0^k,\ga_1,...,\ga_1)$. \\

\noindent \textbf{(1a):} All $k$ of the $\ga_1$-ends are capped off.  Then $\czm(\ga_1) \geq 3$ and
\[
\czm(\gp)-\czm(\ga_0) = \czm(\ga_1) -2 \geq 1.
\]
Since $[\gp]=[\ga_0]$ and $[\ga_+^k]=[\ga_0^k]$  the dynamically separated assumption implies
\[
\czm(\ga_+^k)-\czm(\ga_0^k) \geq 1.
\]
Since
\[
\ind(u_1) = k + \czm(\ga_+^k)-\czm(\ga_0^k) - k\czm(\ga_1) 
\]
then
\[
\ind(\calb) -1 = \czm(\ga_+^k)-\czm(\ga_0^k) \geq 1.
\]

\noindent \textbf{(1b):} If alternately, the $\ga_0^k$ end is capped off then Proposition \ref{unbranchedpants} yields
\[
\begin{array}{lcl}
\hind(\calb) - 1 &\geq & \hind(u_1) + \czm(\ga_0^{k}) -1+ 2k-2 \\
 &\geq&  -2 +k \czm(\ga_0). 
\end{array}
\]
It remains to check $k=2$ and $\czm(\ga_0)=1$, for which we appeal to Proposition \ref{unpant1}, which yields that $\hind(\calb) > 1$ { because \eqref{lemma2ineq} is immediately seen to be satisfied}.

\medskip

\noindent \textbf{Case (2):} Let $u_1$ be a branched $k$-fold cover.  Without loss of generality, denote $u_1 \in \calm(\ga_+^k;\ga_0^{k_0},...,\ga_0^{k_b},\ga_1,...,\ga_1)$.  \\

\noindent \textbf{(2a):} All $k$ of the $\ga_1$-ends are capped off. Then for some $i$,
\[
\ind(\calb) -1 = \czm(\ga_+^k) - \czm(\ga_0^{k_i}).
\]
By the identical argument in Case (1a) we obtain $\ind(\calb) -1 \geq 1$.

\noindent \textbf{(2b):} If alternately, all the $\ga_0^{d_i}$-ends are capped off then Proposition \ref{branchedpants} yields
\[
\begin{array}{lcl}
\hind(\calb) -1 &\geq & \hind(u_1) + \left[ \sum_{i=0}^b\czm(\ga_0^{k_i}) \right] - (b+1) + 2k -2 \\
 &\geq& -2 + k \czm(\ga_0).   \\
\end{array}
\]
It remains to check $k=2$.  By Remark \ref{rembp} we see that if $k=2$ then $\ga_0$ and $\ga_1$ are both contractible.  Thus $\czm(\ga_0) \geq 3$, and the result follows.


\end{proof}

Stacking the above lemmata together, we obtain the following result.


\begin{lemma}\label{bldghomotopy}
Let $\calb$ be a genus 0 building with one positive end $\gp$ and one negative end $\gm$ in a a cobordism $W$ between nondegenerate dynamically separated contact forms and a generic smooth family $\{J_\tau\}_{\tau \in [0,1]}$ of cobordism compatible complex structures.  If
\[
\czm(\gp)-\czm(\gm)=0,
\]
then $\calb$ is one of the following types,
\begin{enumerate}
\item[\em {(i)}] Pairs $(\tau, u)$ with $\tau \in (0,1)$ and $u \in {\mathcal{M}}^{J_{\tau_0}}(\gp,\gm)$;
\item[\em {(ii)}] Pairs $(0, u)$ with $u \in {\mathcal{M}}^{J_0}(\gp,\gm)$;
\item[\em {(iii)}] Pairs $(1, u)$ with $u \in {\mathcal{M}}^{J_1}(\gp,\gm)$;
\item[\em {(iv)}] Pairs $(\tau, (u_+,u_0))$ with $(u_+,u_0)$ a broken cylinder with upper level $u_+ \in {\mathcal{M}}^{J_+}(\gp,\ga_0)$ and main level $u_0 \in {\mathcal{M}}^{J_\tau}(\ga_0,\gm)$ for some $\tau \in (0,1)$;
\item[\em {(v)}]  Pairs $(\tau, (u_0,u_-))$ with $ (u_0,u_-)$ a broken cylinder with lower level $u_- \in {\mathcal{M}}^{J_-}(\ga_0,\gm)$ and main level $u_0 \in {\mathcal{M}}^{J_\tau}(\gp,\ga_0)$ for some $\tau \in (0,1)$.
\end{enumerate}
\end{lemma}

\begin{proof}
 Since the index of a building is additive with respect to its components the results of Proposition \ref{0compacthtpy}, Lemmata \ref{lemma1htpy} and \ref{lemma2htpy} imply that the only possible configurations for a building $\calb$ of index 1 are those described in (i)-(v).
 \end{proof}


\subsection{The chain map}\label{chainmap-sec}

Recall that there are two equivalent differentials $\pa_+^{EGH}$ and $\pa_-^{EGH}$ defined by \eqref{pa1} and \eqref{pa2} respectively.  Throughout this discussion we will fix the differential under consideration to be $\partial_+^{EGH}$.

Let $(W,J)$ be a generic completed symplectic cobordism between $(M_+,\lambda_+,J_+)$ and $(M_-,\lambda_-,J_-)$ where $J_\pm$ are $\lambda_\pm$ compatible almost complex structures.  We define a morphism of complexes
\[
\Phi^{+-}_J:  C_*^{EGH}(M, \lambda_+, J_+) \to C_*^{EGH}(M,\lambda_-, J_-),
\]
 by\footnote{If we use the differential $\partial_-$ then $\langle \Phi^{+-}_J \ga_+, \ga_- \rangle = \displaystyle \sum_{\substack{\ga_- \in \calp_{\mbox{\tiny good}}(\lambda_-)\\  \czm(\ga_+)=\czm(\ga_-) }} \sum_{\calc \in \calm(\ga_+, \gm)} \epsilon(u)\frac{\mult(\gm)}{\mult(\calc)}.$} 

\begin{equation}\label{chaineq}
 \langle \Phi^{+-} \gp, \gm \rangle = \sum_{\substack{\gm \in \calp_{\mbox{\tiny good}}(\lambda_-), \\\czm(\gp)=\czm(\gm) }} \sum_{ \calc \in \calm(\gp, \gm)} \epsilon(u)\frac{\mult(\gp)}{\mult(\calc)} 
\end{equation}


After extracting subsequences, if necessary, the sequence of trajectories in  $\calm(\gp, \gm)$ have limits that are concatenations of at most one broken trajectory from $\mathcal{M}^{J_+}$, exactly one trajectory in $\calm(\gp, \gm)$, and at most one broken trajectory from $\mathcal{M}^{J_-}$.  For ease of notation we denote the Reeb vector field associated to $\lambda_+$ by $R_+$ and to $\lambda_-$ by $R_-$

\begin{proposition}\label{propchain}
Let $\lambda_+$ and $\lambda_-$ be dynamically separated contact forms and $(u_n)$ be a sequence of elements in $\calm(\gp, \gm)$ such that $ 0 \leq \czm(\gp) - \czm(\gm) \leq 1$. There exist
\begin{enumerate}
\item[\em {(i)}] A subsequence of $(u_n)$;
\item[\em {(ii)}] Good Reeb orbits $\gamma_+=x_0, x_1, ... , x_k$ of $R_{+}$;
 \item[\em {(iii)}] Good Reeb orbits $y_0, y_1,...,y_\ell = \gamma_-$ of $R_-$;
 \item[\em {(iv)}] Real sequences $(s^i_n)$ for $ 0 \leq i \leq k-1$ that tend to $+\infty$ and $(\varsigma_n^j)$ for $0\leq j \leq \ell -1$ that tend of $-\infty$.
 \item[\em {(v)}] Cylinders $u^i \in \mathcal{M}^{J_+}(x_i, x_{i+1})$ for $0 \leq i \leq k-1$ and cylinders $v^j \in \mathcal{M}^{J_-}(y_j, y_{j+1})$ for $0 \leq j \leq \ell-1$
 \item[\em {(vi)}] An element $w \in \calm(x_k, y_0)$ such that for $0 \leq i \leq k-1$ and $0 \leq j \leq \ell-1$,
 \[
 \lim_{n \to + \infty} u_n \cdot s_n^i = u^i, \ \ \ \lim_{n \to +\infty} u_n \cdot \varsigma_n^j = v^j
 \]
 and such that
 \[
 \lim_{n \to + \infty} u_n =w.
 \]
 \item[\em {(vii)}] 
Moreover, $ \czm(\gp) - \czm(\gm) \geq k+ \ell $ with $k, \ell \geq 0$.
\end{enumerate}
\end{proposition}

\begin{proof}
That we obtain good Reeb orbits follows from the dynamically separated condition.  By Lemma \ref{bldgchain} there is no bad breaking, and regularity for our cylinders follows from Propositions \ref{cylcoverssep} and \ref{0compact}.  Standard SFT compactness from \cite{BEHWZ} produces (iv)-(vi); see also \cite[Proposition 10.19, 10.23]{wendl-sft}.  

Finally we prove (vi), that $ \czm(\gp) - \czm(\gm) \geq k+ \ell $. We know that
\[
  \left\{
  \begin{array}{ll}
\mathcal{M}^{J_+}(x_i, x_{i+1}) = \emptyset & \mbox{if } \czm(x_i) - \czm(x_{i+1}) \leq 0 \\
\mathcal{M}^{J_-}(y_i, y_{i+1}) = \emptyset & \mbox{if } \czm(y_i) - \czm(y_{i+1}) \leq 0. \\
\end{array} \right.
\]
Thus for $0 \leq i \leq k-1$ and $0 \leq j \leq \ell -1$ we have
\[
\czm(x_i) - \czm(x_{i+1}) > 0 \mbox{ and } \mbox{if } \czm(y_i) - \czm(y_{i+1}) > 0.
\]
Since $\calm(x_k,y_0)$ is cut out transversely by Proposition \ref{cylcoverssep}, $\calm(x_k,y_0) \neq \emptyset$ implies that $\czm(x_k) \geq \czm(y_0)$.  Summing these inequalities yields the desired result.  
\end{proof}

An immediate consequence of the Propositions \ref{0compact} and \ref{propchain} is the following corollary, which means that the chain map (\ref{chaineq}) is well-defined.

\begin{corollary}\label{0manchain}
If $\czm(\gp) = \czm(\gm)$ then the moduli space $\calm(\gp, \gm)$ is compact 0-manifold, hence $\Phi^{+-}_J$ is well defined.  
\end{corollary}

Next we prove the following result.

\begin{proposition}\label{idmap}
If $(\lambda_+, J_+) = (\lambda_-,J_-)$ and $(W, J)$ is the trivial cobordism from $(M,\lambda,J)$ to $(M,c\lambda,J_c)$, for any  constant $c>1$, then $\Phi^{+-}_J$ is the identity.
\end{proposition}
 
 \begin{proof}
Writing $c=e^a$ for $a > 0$, the exact symplectic cobordism is
\[
(\overline{W}, d\lambda) = ([0,a] \times M, d(e^\tau\lambda)).
\]
One can choose a compatible almost complex structure which matches $J$ and $J_c$ on $\xi$ while taking $\partial_\tau$ to $g(r)R_\lambda$ for a suitable function $g$ with $g(\tau)=1$ near $\tau=0$ and $g(\tau)=\frac{1}{c}$ near $\tau = a$.  The resulting almost complex manifold is biholomorphically diffeomorphic to the usual symplectization.  Thus our count of index 0 cylinders is equivalent to the count of such cylinders in the usual symplectization.  There if $\czm(\gp) = \czm(\gm)$ then $\calm(\gp, \gm) = \emptyset$ unless $\gp = \gm$.  All the trivial cylinders are Fredholm regular, so counting these shows that $\Phi^{+-}_J$ is the identity.  
\end{proof}

Finally, we verify that $\Phi^{+-}_J$ is a chain map.

\begin{theorem}
Let $\lambda_+$ and $\lambda_-$ be nondegenerate dynamically separated contact forms on $M^3$ and $J$ generic. Then
\[
\Phi^{+-}_J \circ \partial_{+, (\lambda_+, J_+)} = \partial_{+, (\lambda_-, J_-)} \circ \Phi^{+-}_J.
\]
\end{theorem}

\begin{proof}
Proposition \ref{propchain} in combination with a corresponding gluing theorem shows that for $\czm(\gp)-\czm(\gm) =1$ that $\overline{M}^J(\gp,\gm)$ is a compact 1-manifold whose boundary consists of two types of broken cylinders, depending on whether the index 1 curve occurs in an upper or lower level:
\[
\begin{aligned}
\partial \overline{\mathcal{M}}^J(\gp,\gm) & =& \ \ \ & \bigsqcup_{\substack{\ga_0 \in \calp_{\mbox{\tiny good}}(\lambda_+), \\\czm(\gp)-\czm(\ga_0)=1 }} \widehat{\mathcal{M}}^{J_+}(\gp, \ga_0) \times \calm(\ga_0,\gm) \\ 
&& \cup & \bigsqcup_{\substack{\ga_0 \in \calp_{\mbox{\tiny good}}(\lambda_-), \\\czm(\gp)-\czm(\ga_0)=1 }}  \calm(\gp, \ga_0) \times \widehat{\mathcal{M}}^{J_-}(\ga_0, \gm). \\
\end{aligned}
\] 

That $\overline{\mathcal{M}}^J(\gp,\gm)$ is a 1-manifold follows from Proposition \ref{cylcoverssep} in conjunction with an argument identical to the proof of Corollary \ref{0manchain}, all as a result of the dynamically separated assumption.  Gluing arguments follow the same reasoning as in \cite[\S 4.3]{HN1}.   Other readily accessible discussions of gluing are given in the Hamiltonian Floer setting \cite[\S11.2]{ADfloer} and in symplectic field theory in \cite[\S 11]{wendl-sft}, which includes a detailed discussion of orientations. 

Counting broken cylinders of the first type produces the coefficient of
$\langle \Phi^{+-}_J \circ \partial_{+,(\lambda_+,J_+)} \rangle$ and counts of the second type produces the desired coefficient of
$\langle \partial_{+,(\lambda_-,J_-)} \circ \Phi^{+-}_J \rangle$.  
\end{proof}

It now follows that $\Phi^{+-}_J$ descends to a homomorphism at the level of homology.



\subsection{The chain homotopy}\label{chainhtpy-sec}  

Finally, we show that $\Phi^{+-}_J$ induces an isomorphism at the level of homology.  Given two exact completed symplectic cobordisms $(W, J_0)$ and $(W,J_1)$ between dynamically separated contact forms $(\lambda_+, J_+)$ and $(\lambda_-,J_-)$, we want to prove that $\Phi_0:=\Phi^{+-}_{J_0}$ and $\Phi_1:=\Phi^{+-}_{J_1}$ are chain homotopic. 


Adapting Proposition \ref{propchain}  to allow for a converging sequence of almost complex structures in conjunction with Lemma \ref{bldghomotopy} yields the following proposition. 

\begin{proposition}\label{prophtpy}
Let $\lambda_+$ and $\lambda_-$ be dynamically separated contact forms and $(\tau_n,u_n)$ be a sequence of elements in $\widehat{\M}^{\mathbb{J}}(\gp, \gm)$ such that $ -1 \leq \czm(\gp) - \czm(\gm) \leq 0$. There exist
\begin{enumerate}
\item[\em {(i)}] A subsequence of $(\tau_n, u_n)$;
\item[\em {(ii)}] Good Reeb orbits $\gamma_+=x_0, x_1, ... , x_k$ of $R_{+}$;
 \item[\em {(iii)}] Good Reeb orbits $y_0, y_1,...,y_\ell = \gamma_-$ of $R_-$;
 \item[\em {(iv)}] Real sequences $(s^i_n)$ for $ 0 \leq i \leq k-1$ that tend to $+\infty$ and $(\varsigma_n^j)$ for $0\leq j \leq \ell -1$ that tend of $-\infty$.
 \item[\em {(v)}] Cylinders $u^i \in \mathcal{M}^{J_1}(x_i, x_{i+1})$ for $0 \leq i \leq k-1$ and cylinders $v^j \in \mathcal{M}^{J_0}(y_j, y_{j+1})$ for $0 \leq j \leq \ell-1$
 \item[\em {(vi)}] An element $(\tau_\star, w) \in \calm(x_k, y_0)$ such that for $0 \leq i \leq k-1$ and $0 \leq j \leq \ell-1$,
 \[
 \lim_{n \to + \infty} u_n \cdot s_n^i = u^i, \ \ \ \lim_{n \to +\infty} u_n \cdot \varsigma_n^j = v^j
 \]
 and such that
 \[
 \lim_{n \to + \infty} (\tau_n, u_n) = (\tau_\star, w).
 \]
 \item[\em {(vii)}] 
Moreover, $ \czm(\gp) - \czm(\gm)  + 1 \geq k+ \ell $ with $k, \ell \geq 0$.
\end{enumerate}
\end{proposition}

We can now define a homomorphism of odd degree by 
\[
S: C_*^{EGH}(M, \lambda_+, J_+) \to C_{*+1}^{EGH}(M, \lambda_-,J_-)
\]
by
\begin{equation}\label{htpyeq}
\langle S \gp, \gm \rangle =  \sum_{\substack{\gm \in \calp_{\mbox{\tiny good}}(\lambda_-), \\\czm(\gp)=\czm(\gm)-1 }} \sum_{ \calc \in \calm(\gp, \gm)} \epsilon(u) \# \widehat{\M}^\mathbb{J}(\gp,\gm)
\end{equation}

\begin{corollary}\label{0manhtpy}
If $\czm(\gp) = \czm(\gm) -1 $ then the moduli space $\widehat{\mathcal{M}}^\mathbb{J}(\gp, \gm)$ is compact 0-manifold, hence $S$ is well-defined.  
\end{corollary}

\begin{proof}
This follows from Propositions \ref{0compacthtpy} and \ref{prophtpy}.
\end{proof}

Next, we claim that $S$ is a chain homotopy, e.g. that
\[
\Phi_1 -\Phi_0 = S \circ \partial_{+, (\lambda_+, J_+)} + \partial_{+, (\lambda_-,J_-)}\circ S.
\]

\begin{proposition}\label{prophtpy1}
At the level of homology, the morphism $\Phi^{+-}$ induces a morphism that is independent of the choice of completed symplectic cobordism $(W,J)$ between $(\lambda_+, J_+)$ and $(\lambda_-,J_-)$.
\end{proposition}


\begin{proof}
This will follow from the boundary of the compactified 1-dimensional moduli space $\overline{\mathcal{M}}^\mathbb{J}(\gp,\gm)$, where $\czm(\gp)-\czm(\gm)=0$.  By Lemma \ref{bldghomotopy} along with appropriate gluing arguments, the boundary consists of four types of objects:
\begin{enumerate}
\item Pairs $(0, u)$ with $u \in {\mathcal{M}}^{J_0}(\gp,\gm)$ which are counted by $\Phi_0$.
\item Pairs $(1, u)$ with $u \in {\mathcal{M}}^{J_1}(\gp,\gm)$ which are counted by $\Phi_1$.
\item Pairs $(\tau, (u_+,u_0))$ with $(u_+,u_0)$ a broken cylinder with upper level $u_+ \in {\mathcal{M}}^{J_+}(\gp,\ga_0)$ and main level $u_0 \in {\mathcal{M}}^{J_\tau}(\ga_0,\gm)$ for some $\tau \in (0,1)$; these are counted by $S \circ \partial_+$.
\item Pairs $(\tau, (u_0,u_-))$ with $ (u_0,u_-)$ a broken cylinder with lower level $u_- \in {\mathcal{M}}^{J_-}(\ga_0,\gm)$ and main level $u_0 \in {\mathcal{M}}^{J_\tau}(\gp,\ga_0)$ for some $\tau \in (0,1)$; these are counted by $\partial_-\circ S$.
\end{enumerate}
The sum $\Phi_1 -\Phi_0 - S \circ \partial_{+, (\lambda_+, J_+)} -\partial_{+, (\lambda_-,J_-)}\circ S$ is therefore an oriented count of the boundary points of a compact 1-manifold, so it vanishes.

\end{proof}

The final step needed in the  proof of Theorem \ref{chainthm}  is the following proposition.

\begin{proposition}\label{prophtpy2}
Let $(\lambda_1, J_1), \ (\lambda_2, J_2), \ (\lambda_3, J_3)$ be three nondegenerate dynamically separated pairs on $(M, \xi)$ and let $(W_{21}, J_{21})$ and $(W_{32}, J_{32})$ be two completed symplectic cobordisms between $(\lambda_2, J_2), \ (\lambda_1, J_1)$ and $ (\lambda_3, J_3), \ (\lambda_2, J_2)$ respectively.  Then there exists a completed symplectic cobordism $(W_{31}, J_{31})$ between $ (\lambda_3, J_3), \ (\lambda_1, J_1)$ such that
\[
\Phi^{31} \mbox{ and } \Phi^{32} \circ \Phi^{21}: C_*^{EGH}(M, \lambda_3,J_3) \to C_*^{EGH}(M, \lambda_1,J_1)
\]
induce the same homomorphism at the level of homology.
\end{proposition}

\begin{proof}
The proof of this proposition relies on a neck stretching construction.  Explicit details of such constructions can be found in \cite[Appendix 1]{mclean}, \cite[\S 9.4.4]{wendl-sft}.  
After rescaling, suppose without loss of generality that $\lambda_i = e^{f_i}\lambda$ with $f_3>f_2 > f_1$,  Then the cobordism
\[
\overline{W}_{31}:=\{ (r,x) \ | \ f_1(x) \leq r \leq f_3(x) \}
\]
contains a contact-type hypersurface
\[
M_{2}:=\{ (f_1(x),x) \ | \ x \in M \} \subset 	\overline{W}_{31}.
\]
We choose a sequence of compatible almost complex structures $\{J^N_{31}\}_{N\in \N}$ on $W_{31}$ that are fixed outside a neighborhood of $M_2$ but degenerate in this neighborhood as $N \to \infty$.  This is equivalent to replacing a small tubular neighborhood of $M_2$ with increasingly large collars $[-N, N] \times M$ in which $J_{31}^N$ is $\lambda_2$-compatible.  The resulting chain maps
\[
\Phi_{J_{31}^N}^{31}:C_*^{EGH}(C_*^{EGH}(M,\lambda_3, J_3) \to C_*^{EGH}(M, \lambda_1,J_1)
\]
are chain homotopic for all $N$.  As $N \to \infty$, the index 0 cylinders counted by these maps coverge to buildings with two levels.  The top level is an index 0 cylinder which lives in the completed cobordism from $(M, \lambda_2, J_2)$ to $(M, \lambda_3, J_3)$, while the bottom level is an index 0 cylinder which lives in the completed cobordism from $(M, \lambda_1, J_1)$ to $(M, \lambda_2, J_2)$.  That there are no other levels follows from the calculations in \S \ref{obs-chain}-\ref{obs-htpy}.

The composition $\Phi^{32} \circ \Phi^{21}$ counts these broken cylinders and we have
\[
\begin{array}{lcl}
\langle \Phi^{31} \ga_{3}, \ga_1\rangle & =&  \displaystyle \sum_{\substack{\ga_2 \in \mathscr{P}_{\mbox{\tiny good}}(\lambda_2) \\ \czm(\ga_3) = \czm(\ga_2)}} \sum_{\substack{ u \in \widehat{\mathcal{M}}^{J_{32}}(\ga_3,\ga_2) \\ v \in \widehat{\mathcal{M}}^{J_{21}}(\ga_2,\ga_1) }} \left( \epsilon(u)\epsilon(v) \frac{\mult({\ga_3})}{\mbox{lcm}(\mult(u), \mult(v))} \frac{\mult(\ga_2)}{\mbox{gcd}(\mult(u)\mult(v) )}\right), \\
&&\\
 &=& \hspace{.9cm} \displaystyle \sum \epsilon(u)\epsilon(v)  \frac{ \mult(\ga_3)\mult(\ga_2)}{\mult(u)\mult(v)}\\
 &&\\
 &=& \displaystyle \sum_{\substack{\ga_2 \in \mathscr{P}_{\mbox{\tiny good}}(\lambda_2) \\ \czm(\ga_3) = \czm(\ga_2)}} \langle \Phi^{32} \ga_3, \ga_2 \rangle  \langle \Phi^{21} \ga_2, \ga_1 \rangle. \\
\end{array}
\]
Thus, at the level of homology,
\[
\Phi_{32} \circ \Phi_{21} = \Phi_{31}.
\]
As a result, we can conclude that each of the maps $\Phi^{+-}: CH_*(M,\lambda_+, J_+) \to CH_*(M,\lambda_-, J_-) $ is an isomorphism because composing $\Phi^{+-}$ with $\Phi^{-+}$ must give the identity by Proposition \ref{idmap}.  

\end{proof}


\subsection{Invariance of the filtered homology under continuation}\label{filtered-continuation}

Invariance of the filtered cylindrical contact homology groups under continuation is more subtle than in the unfiltered case because the filtered groups are invariant only along paths $(\lambda_\tau, J_\tau) $ for which $L$ is not a period of a Reeb orbit associated to $R_{\lambda_\tau}$.  Because the corresponding path spaces may not be connected, the resulting continuation isomorphism may depend on the homotopy class of the path.  However, in Theorem \ref{invariance-thm1} we fixed the contact form and allowed $J$ to vary, thus the proof follows by repeating the arguments in the construction of the chain map and chain homotopy for unfiltered cylindrical contact homology.  

Next we consider the degree of independence ont he choice of $L$-nondegenerate dynamically separated contact form.  Let $\{ \lambda_\tau, J_\tau \}$ be a smooth homotopy between $L$-nondegenerate dynamically separated contact forms and compatible almost complex structures satisfying
\[
\left\{
\begin{array}{lcl}
( \lambda_\tau, J_\tau ) = (\lambda_0, J_0)  &\mbox{for} & s \leq 0;\\
( \lambda_\tau, J_\tau ) = (\lambda_1, J_1)  & \mbox{for} & s \geq 1.\\
\end{array}
\right.
\]

 The continuation map 
\[
\Phi^L_{\{ \lambda_\tau, J_\tau \}}: CH_*^{EGH,L}(M,\lambda_1, J_1) \to CH_*^{EGH,L}(M,\lambda_0, J_0) 
\]
preserves the subcomplexes on the chain level if $\mathcal{A}(\gamma_1) < \mathcal{A}(\gamma_0)$. If this condition is not satisfied, we still obtain isomorphisms
\[
\Phi^L_{\{ \lambda_\tau, J_\tau \}}(\tau_1,\tau_0): CH_*^{EGH,L}(M,\lambda_1, J_1) \to CH_*^{EGH,L}(M,\lambda_0, J_0) 
\]
for $|\tau_1-\tau_0|$ sufficiently small. To see this, it suffices to replace $(\lambda_\tau, J_\tau)$ by the homotopy
\[
\tau \mapsto \left(\lambda_{\beta(\tau)}, J_{\beta(\tau)}\right) \ \mbox{ where } \ \beta(\tau):=\tau_0 + \rho(\tau)(\tau_1-\tau_0)
\]
and $\rho: \R \to [0,1]$ is a smooth cutoff function satisfying 
\[
\left\{
\begin{array}{lcl}
\rho(\tau) =0 &\mbox{for} & \tau \leq 0;\\
\rho(\tau) = 1 & \mbox{for} & \tau \geq 1.\\
\end{array}
\right.
\]

For a general pair of real numbers $\tau_0, \ \tau_1$, the isomorphism $\Phi^L_{\{ \lambda_\tau, J_\tau \}}(\tau_1,\tau_0)$ can then be defined as a composition of the isomorphisms  $\Phi^L_{\{ \lambda_\tau, J_\tau \}}(\tau_{i+1},\tau_{i})$ for a suitable partition of the interval $[\tau_0,\tau_1]$.  The resulting isomorphism is independent of the choice of partition.

By repeating the arguments in the construction of the chain map and chain homotopy for the unfiltered cylindrical contact homology, we deduce that the continuation isomorphisms on filtered Floer homology have the following properties.  

\begin{theorem}\label{invariance-thm2}
Let $M$ be a closed oriented connected 3-manifold  and $(\lambda_{\tau_i}, J_{\tau_i})$ are $L$-nondegenerate dynamically separated pairs for $i=0,1,2$. 
\begin{description}
\item[Naturality:] If $\{ \lambda_\tau, J_\tau \}$ is a generic smooth path through dynamically separated pairs then $\Phi^L_{\{ \lambda_\tau, J_\tau \}}(\tau_0,\tau_0) = \id$ and
\[
 \Phi^L_{\{ \lambda_\tau, J_\tau \}}(\tau_2,\tau_0) = \Phi^L_{\{ \lambda_\tau, J_\tau \}}(\tau_2,\tau_1) \circ \Phi^L_{\{ \lambda_\tau, J_\tau \}}(\tau_1,\tau_0).
\]

\item[Homtopy:] The isomorphism $\Phi^L_{\{ \lambda_\tau, J_\tau \}}(\tau_1,\tau_0) $ depends only on the homotopy class with fixed endpoints of the path $\{ \lambda_\tau, J_\tau \}$.
\item[Filtration:] If $L < L'$ and $(\lambda,J)$ is an $L$-nondegenerate dynamically separated pair then the continuation maps commute with the homomorphisms in the long exact sequence
\[
\begin{array}{llc c l}
... & \to &CH_*^{EGH, L}(M,\lambda, J) &\to & CH_*^{EGH, L'}(M,\lambda, J) \\ 
& \to& CH_*^{EGH, [L,L']}(M,\lambda, J) &\to& CH_{*-1}^{EGH, L}(M,\lambda, J) \to ... \\
\end{array}
\]
for generic smooth paths through $L'$-nondegenerate dynamically separated pairs.  Here $CH_*^{EGH, [L,L']}(M,\lambda, J)$ denotes the homology of the quotient complex.
\item[Monotonicity:] The continuation homomorphism preserves the subcomplexes $CH_*^{EGH, L}$ and induces a homomorphism
\[
CH_*^{EGH, L}(M,\lambda_{\tau_0}, J_{\tau_0}) \to CH_*^{EGH, L}(M,\lambda_{\tau_1}, J_{\tau_1}) 
\]
for $\tau_0 < \tau_1$.  If, in addition $\lambda_\tau$ is dynamically separated for every $\tau \in [\tau_0,\tau_1]$ then this is an isomorphism and agrees with $\Phi^L_{\{ \lambda_\tau, J_\tau \}}(\tau_1,\tau_0)$.

\end{description}

\end{theorem}

\begin{remark} \em
Let $M$ be a closed oriented connected 3-manifold  and $\lambda$ be a nondegenerate dynamically separated contact form.  Then 
\[
 CH_*^{EGH}(M, \mbox{ker }  \lambda) = \lim_{L \to \infty} CH_*^{{EGH, L}}(M,  \lambda).
\]
Moreover, in this case the filtered continuation isomorphisms agree with the usual ones.  Hence the filtration property asserts that there is a well-defined homomorphism
\[
\iota^L(\lambda): CH_*^{{EGH, L}}(M,  \lambda) \to  CH_*^{{EGH}}(M,  \lambda), 
\]
which is induced by the inclusion of chain complexes.  The filtration property also shows that every path of nondegenerate dynamically separated contact forms determines a commutative diagram
\[
\begin{tikzcd}
CH_*^{{EGH, L}}(M,  \lambda_0) \arrow[r, "\Phi^L_{\{ \lambda_\tau, J_\tau \}}"] \arrow[d, "\iota^L "]
& CH_*^{{EGH, L}}(M,  \lambda_1) \arrow[d, "\iota^L"] \\
CH_*^{{EGH}}(M,  \lambda_0) \arrow[r, "\Phi^{10}" ]
& CH_*^{{EGH}}(M,  \lambda_1)
\end{tikzcd}
\]
where $\Phi^L_{\{ \lambda_\tau, J_\tau \}}$ is the continuation isomorphism of filtered cylindrical contact homology and $\Phi^{10}$ is the canonical cylindrical continuation isomorphism.

\end{remark}


\section{Grinding through gradings}\label{cz-section}
This section provides the details on the Reeb dynamics of prequantization bundles and the computation of the Conley-Zehnder index of the associated Reeb orbits. Recall that $(V^3, \lambda)$ is a prequantization bundle over an integral closed symplectic surface $(\Sigma, \omega)$ so that $[\omega]$ is primitive and $c_1(\Sigma)=c[\omega]$.  Let $\lambda_\vepsilon=(1+\vepsilon \pi^*H)\lambda$ be perturbed by a Morse-Smale function $H$ on $\Sigma$ satisfying $|H|_{C^2}<1$ and $\ga_p$ be the simple Reeb orbit of $R_\veps$ which projects to $p\in \mbox{Crit}(H)$.

\begin{proposition}\label{czcontractible}
Fix a Morse function $H$ such that $|H|_{C^2}<1$ and a constant $T>0.$ There exists $\vepsilon >0$ such that all Reeb orbits with $\mathcal{A}(\gamma) < T$ are nondegenerate and project to critical points of $H$.   Moreover, when $\mathcal{A}(\gpk)<T$,
\begin{equation}
\label{czprequant}
\czm(\gpk)=\mu_{RS}(\gamma^k) -1+\mbox{\emph{index}}_p(H),
\end{equation}
where $\gamma^k$ is the $k$-th iterate of the simple degenerate Reeb orbit corresponding to the circle fiber of $V \to \Sigma$.
\end{proposition}

In Section \ref{crossingform}, we review the necessary material about the Robbin-Salamon index and compute it for linearized flows relevant to the proof of Proposition \ref{czcontractible}.  In Section \ref{prequantdyn} we review the Reeb dynamics of prequantization bundles and finish the proof of Proposition \ref{czcontractible}.
\subsection{The beloved crossing form of Robbin and Salamon}\label{crossingform}
The Conley-Zehnder index $\mu_{CZ}$, is a Maslov index for arcs of symplectic matrices which  assigns an integer $\mu_{CZ}(\Phi)$ to every path of symplectic matrices $\Phi : [0,T] \to \mbox{Sp}(n)$, with $\Phi(0) = \mathds{1} $.   In order to ensure that the Conley-Zehnder index assigns the same integer to homotopic arcs, one must also stipulate that 1 is not an eigenvalue of the endpoint of this path of matrices, i.e. $\det(\mathds{1} - \Phi(T))\neq 0$.  We define the following set of continuous paths of symplectic matrices that start at the identity and end on a symplectic matrix that does not have 1 as an eigenvalue. 
\[
\Sigma^*(n) = \{ \Phi :[0,T] \to \mbox{Sp}(2n)  \ | \ \Phi \mbox{ is continuous},  \ \Phi(0)=\mathds{1}, \mbox{ and }  \mbox{det}(\mathds{1} - \Phi(T)) \neq 0  \}.
\]

The Conley-Zehnder index is a functor satisfying the following properties, and is uniquely determined by the homotopy, loop, and signature properties.

\begin{theorem}{\em \cite[Theorem 2.3, Remark 5.4]{RS1}}, {\em \cite[Theorem 2, Proposition 8 \& 9]{GuCZ}}\label{CZprop} \\
There exists a unique functor $\mu_{CZ}$ called the {\bf{Conley-Zehnder index}} that assigns the same integer to all homotopic paths $\Psi$ in $\Sigma^*(n)$,
\[
\mu_{CZ}: \Sigma^*(n) \to \Z.
\]
 such that the following hold.
\begin{enumerate}[\em (1)]
\item {\bf{Homotopy}}: The Conley-Zehnder index is constant on the connected components of $\Sigma^*(n)$.
\item {\bf{Naturalization}}: For any paths $\Phi, \Psi: [0,1] \to Sp(2n)$, $\mu_{CZ}(\Phi\Psi\Phi^{-1}) = \mu_{CZ}(\Psi)$.
\item {\bf{Zero}}: If $\Psi(t) \in \Sigma^*(n)$ has no eigenvalues on the unit circle for $t >0$, then $\mu_{CZ}(\Psi) = 0$.
\item {\bf{Product}}: If $n = n' + n''$, identify $Sp(2n') \oplus Sp(2n'')$ with a subgroup of $Sp(2n)$ in the obvious way. For $\Psi' \in \Sigma^*(n')$, $\Psi'' \in \Sigma^*(n'')$, then $\mu_{CZ}(\Psi' \oplus \Psi'') = \mu_{CZ}(\Psi') + \mu_{CZ}(\Psi'')$.
\item {\bf{Loop}}: If $\Phi$ is a loop at $\mathds{1}$, then $\mu_{CZ}(\Phi\Psi) = \mu_{CZ}(\Psi) + 2\mu(\Phi)$ where $\mu$ is the Maslov Index.
\item {\bf{Signature}}: If $S \in M(2n)$ is a symmetric matrix with $||S|| < 2\pi$ and $\Psi(t) = \exp(J_0St)$, then $\mu_{CZ}(\Psi) = \frac{1}{2}\mbox{\em sign}(S)$.
\end{enumerate}
\end{theorem}

As before we will take $\ga$ to be a (nondegenerate) closed Reeb orbit of period $T$.  We fix a symplectic trivialization $\tau$\footnote{Since $\Phi$ is used to denote a matrix in this section, we will use $\tau$ for the choice of trivialization rather than $\Phi$ in this section.} of $\xi$ along $\ga$, as so that the linearized flow 
\[
  d\varphi_t: \xi_{\gamma(0)} \to \xi_{\gamma(t) }
  \]
for $t\in[0,T]$  is given by a path $\Psi(t)$ of symplectic matrices.  Note that $\Psi(0) = \id$ and, when $\ga$ is nondegenerate, $\det(\Psi(T)-\id)\neq 0$.  This permits us to compute the Conley-Zehnder index of $d\varphi_t, \ t\in[0,T],$  
\[
\mu_{CZ}^\tau(\gamma):=\mu_{CZ}\left( \left\{ d\varphi_t \right\}\arrowvert_{t\in[0,T]}\right).
\]

As explained in \S \ref{gradingsec}, this index is dependent on the choice of trivialization $\Phi$ of $\xi$ along $\gamma$ which was used in linearizing the Reeb flow. However, if $c_1(\xi;\Q)=0$ we can use the existence of an (almost) complex volume form on the symplectization to obtain a global means of linearizing the flow of the Reeb vector field. The choice of a complex volume form is parametrized by $H^1(\R \times M;\Z)$, so an absolute integral grading is only determined up to the choice of volume form.

We may alternately realize the Conley-Zehnder index in terms of crossing forms, and that both definitions agree is proven in \cite{RS1}.  Using crossing forms to compute the Conley-Zehnder also allows one to compute the index of arbitrary paths of symplectic matrices,
\[
\Psi(t) \in \Sigma(n):=\{\Psi:[0,T] \to \Sp(n) \ : \ \Psi \mbox{ is continuous, } T>0\mbox{ and } \Psi(0)=\id\}.
\]
In particular, associated to every periodic solution we obtain a half integer $\rs$ which agrees with $\czm$ in the nondegenerate case, i.e. when $\Psi(t) \in \Sigma^*(n).$

This is accomplished by realizing $\Psi(t)$ as a smooth path of Lagrangian subspaces.  To do this, we review the construction of $\rs$ via the index of the Lagrangian path 
\[
\mbox{Graph}(\Psi(t)):=\{ (x,\Psi(t) x)\ | \ x\in \R^n \}
\]
 in $(\R^{2n}\times \R^{2n}, ((-\om_0) \oplus \om_0) )$ relative to the diagonal 
 \[
 \Delta := \{ (X,X)\ | \ X\in \R^{2n} \}. 
 \]
 Here $\om_0$ is the standard symplectic form on $\R^{2n}$.  Assuming $\Psi(a)=\id$ and $\det(\id - \Psi(b))\neq 0$ then the index of this Lagrangian path may be defined as follows, 
\[
\rs(\Psi):=\mu(\mbox{Graph}(\Psi),\Delta).
\]
This index is an integer and satisfies
\[
(-1)^{\mu(\Psi)-n}=\sign\det(\id - \Psi(b)).
\]
The above number is the parity of the Lagrangian frame $(\id, \Psi(b))$ for the graph of $\Psi(b)$.  This index can then be computed via quadratic forms defined at crossing numbers.

A number $t\in [0,T]$ is called a \textbf{crossing} if $\det(\Psi(t)-\id)=0.$  We denote the set of crossings by
\[
E_t:=\ker(\Psi(t)-\id). 
\]
For a crossing $t \in [0,T]$, the crossing form $\g(\Psi,t)$ is the quadratic form on $E_t$ defined by:
\[
\g(\Psi,t)(v):=d\A(v, \dot{\Psi} v) \ \ \ \mbox{ for } v \in E_t. 
\] 

A crossing $t$ is \textbf{regular} whenever the crossing form at $t$ is nonsingular.  Note that regular crossings are necessarily isolated.  Any path $\Psi$ is homotopic with fixed end points to a path having only regular crossings.   Recall that the \textbf{signature} of a nondegenerate quadratic form is the difference between the number of its positive eigenvalues and the number of its negative eigenvalues. 

Robbin and Salamon define the index $\rs(\Psi)$ of the path $\Psi$ having only regular crossings to be
\[
\rs(\Psi):=\frac{1}{2}\sign(\g(\Psi,0)(v)) + \sum_{0<\mbox{\tiny all crossings }t < T}\sign(\g(\Psi,t)(v)) + \frac{1}{2}\sign(\g(\Psi,T)(v)).
\]
In the case that we have taken the linearized flow of a nondegenerate Reeb orbit to obtain our path of symplectic matrices, i.e. $\Psi \in \Sigma^*(1)$, we obtain
\[
\rs(\Psi):=\frac{1}{2}\sign(\g(\Psi,0)(v)) + \sum_{0<\mbox{\tiny all crossings }t \leq T}\sign(\g(\Psi,t)(v)). 
\]
This is because $t=T$ is no longer a crossing as $\det(\Psi(t)-\id) \neq 0.$

If we are working in $(\R^{2n}, \om_0)$, we have the following expression of the crossing form.  Since any path in $\Sp(2n, \R)$ is a solution to a differential equation $\dot{\Psi}(t)=J_0S(t)\Psi(t)$, with $S(t)$ a symmetric matrix we can write the crossing form in $\R^{2n}$ as
\begin{equation}\label{symmcz}
\g_0(\Psi(t), t)(v) = \langle v, S(t)v \rangle
\end{equation}

  

The main features of the Robbin-Salamon index are the following.
\begin{proposition}
The Robbin-Salamon index has the following properties.
\begin{list}{\labelitemi}{\leftmargin=0em }
\item[{\em (i)}] The Robbin-Salamon index satisfies additivity under concatenations of paths, 
\[
\rs \left( \Psi \arrowvert_{[a, b]} \right) + \rs \left( \Psi \arrowvert_{[b, c]} \right) = \rs \left( \Psi \arrowvert_{[a, c]} \right)
\]
\item[{\em (ii)}] The Robbin-Salamon index characterizes paths up to homotopy with fixed end points.
\item[{\em (iii)}] The Robbin-Salamon index satisfies additivity under products,
\[
 \rs(\Psi' \oplus \Psi '') = \rs(\Psi ') + \rs(\Psi '') .
\]
\end{list}
\end{proposition}

As a preliminary example, we compute the Robbin-Salamon index for the symplectic path of matrices arising from the flow given by $\varphi_t(z) = e^{i t}z$ on $(\C,\omega_0)$.  If we take $t\in [0, 2\pi n]$ we do not obtain a path of symplectic matrices in $\Sigma^*(1)$ but we may still make use of  crossing forms to compute the Robbin-Salamon index for this path.
\begin{example}
\label{rssphere}
 \em
The linearization of $\varphi_t(z) = e^{i t}z$  is given by $d\varphi_t(z) \cdot v = e^{ i t} v$.  We denote $ \Psi(t)=e^{ i t}$ and obtain crossings for $t = 2\pi n$ for every $n \in \Z_{\geq 0}$. From (\ref{symmcz}) the crossing form may be written as
\[
\g_0(\Psi, t)(v)= \langle v, v \rangle
\]
For $t=2\pi n$ with $n \in \Z_{\geq0}$ we have that $\g_0$ is nondegenerate and 
\[
\g_0(\Psi, t)(v) = v \bar{v} = a^2 + b^2,
\]
where $v=a+ ib$.  This has signature +2, and thus on $[0, 2\pi n]$ with $n \in \Z_{>0}$ we have
\[
\rs(\Psi(t)) = 2n.
\]
If we take $\Psi(t)$ to be defined on the interval $[0, 2\pi n + \veps]$ with $0<\veps < 2\pi$ then this is a path of symplectic matrices in $\Sigma^*(1)$ and we obtain 
\[
\czm(\Psi(t))=\rs(\Psi(t))=2n
\]
\end{example}

Next we compute the Robbin-Salamon index of the linearization $\Psi$ of the time $\vepsilon$ flow near a critical point $p$ of a Morse function $H$ on $(\Sigma^{2n},\omega)$.

\begin{lemma}\label{RSHam}
Let $\Psi:=\{\Psi(t)\}_{t\in[0,\vepsilon)}$ be the path of symplectic matrices associated to the linearization of the Hamiltonian vector field $X_H$ of a Morse function $H$ at a critical point $p$ on $(\Sigma^{2n},\omega)$.  Then $\{\Psi(t)\}_{t\in[0,\vepsilon)}$ has an isolated crossing at 0 and
\[
\rs(\Psi) = \mbox{\em index}_pH - \frac{1}{2}\mbox{\em dim}\Sigma.
\]
\end{lemma}

\begin{proof}
We will use the convention that 
\[
\omega(X_H, \cdot) = dH.
\]
Let $p$ be a critical point of $H$.  After picking a Darboux ball around $p$ we have 
\[
X_H=-J_0\nabla H.
\]
The linearized flow $\Psi$ is a solution of the autonomous ODE 
\[
\dot{\Psi} = -J_0 \nabla^2 H \cdot \Psi.
\]
Thus
\[
\Psi(t) = \mbox{exp}(-J_0 \mbox{Hess}_p(H) t).
\]
Since $H$ is Morse its Hessian is nondegenerate at $p$.  The crossing form is given by
\[
\g_0(\Psi,0)(v) = v^T \mbox{Hess}_p(H) v,
\]
and for sufficiently small $\vepsilon$ the only crossing is at $t=0$.  
By a Morse shift lemma we obtain
\[
\rs(\Psi) = -\frac{1}{2} \sign \ \mbox{Hess}_pH= \indx_pH- \frac{1}{2}{\dim \Sigma}.
\]
\end{proof}

Also needed is the following computation of the Robbin-Salamon index associated to the linearized  Hopf flow.  

\begin{lemma}\label{rssphere}
For a closed Reeb orbit $\ga^k$ associated to the degenerate Reeb flow on $S^3$ generated by the standard contact form $\lo$, we have
\[
\rs(\ga^k)=4k.
\]
\end{lemma}
\begin{proof}
The standard contact form on $S^3$ is
\[
\lo=\frac{i}{2}(u d\bar{u} - \bar{u} du + v d\bar{v} - \bar{v} dv)|_{S^3},
\]
and
\begin{equation}
\label{reeb3sphere}
\begin{array}{ccl}
R & =& i \left( u \dfrac{\pa}{\pa u} -  \bar{u} \dfrac{\pa}{\pa \bar{u}} +  v \dfrac{\pa}{\pa v} -  \bar{v} \dfrac{\pa}{\pa \bar{v}}\right)\\
&= & (ix_1-y_1, ix_2-y_2) \\
&=& (iu, iv) \\
\end{array}
\end{equation}
 Recall that 
 \[
 \varphi_t(u,v)=(e^{it}u, e^{it}v).
 \]
gives the flow of the Reeb vector field of (\ref{reeb3sphere}).  It also gives rise to a symplectomorphism of $\C^2\setminus \{\bo \}$, thereby allowing us to obtain a global trivialization which extends the trivialization around the closed orbits  to the closed disks spanned by the orbits.  

There is the following natural splitting of $\C^2$, 
\[
\C^2 \cong \xi_p \oplus \xi_p^{\om}.
\]
Here $\xi_p^{\om}$ is the symplectic complement of $\xi_p$, defined as follows
\[
\xi_p^{\om} =\{ v \in T_pS^3 \ | \ \om(v,w)  = 0 \mbox{ for all } w\in \xi_p  \}.
\]
On $\C^2 \setminus \{ 0\}$ we use the symplectic form $d(e^\tau \lo)$ pulled back under the biholomorphism,
\[
\begin{array}{crcl}
\psi:& \co &\to& \R \times S^3 \\
& z &\mapsto& \left( \frac{1}{2} \ln |z|, \dfrac{z}{|z|} \right) \\
\end{array}
\] 
which we denote by
\[
\om_0= \om_{\co} = \psi^*(d(e^\tau \lo))
\]

We may write $\xi^{\om_0}_p$ as the span of the following vector fields evaluated at $p$:
\begin{equation}\label{XY}
\begin{array}{lcccr}
X&=&-i(u,v) &=& - i \left( u \dfrac{\pa}{\pa u} -  \bar{u} \dfrac{\pa}{\pa \bar{u}} +  v \dfrac{\pa}{\pa v} -  \bar{v} \dfrac{\pa}{\pa \bar{v}}\right),  \\
Y &=& (u,v) &=&  \left( u \dfrac{\pa}{\pa u} -  \bar{u} \dfrac{\pa}{\pa \bar{u}} +  v \dfrac{\pa}{\pa v} -  \bar{v} \dfrac{\pa}{\pa \bar{v}}\right). \\
\end{array}
\end{equation}

The vector fields $X$ and $Y$ defined in in (\ref{XY}) yield a standard symplectic or Darboux basis for the symplectic vector space $\xi_p^{\om_0}$ because
\[
\begin{array}{rcrcr}
\om_0(X,Y) &=& -\om_0(Y,X) &=& 1; \\
\om_0(X,X) &=& \om_0(Y,Y) &=&0. \\
\end{array}
\]
We have that $\om_0$ on $\xi_p^{\om_0}$ is given by
\[
 \left( \begin{array}{cc}
0 & 1 \\
 1 & 0 \\
\end{array} \right).
\]
Thus $\xi^{\om_0}$ is symplectically trivial and $\xi$ is symplectically trivial because
\[
T \C^2 \cong \xi \oplus \xi^{\om}.
\]
The linearized flow acts on $\xi_p^{\om_0}$ by
\[
\begin{array}{lcl}
d\varphi_t(X(p)) &=& X(\varphi_t(p)), \\
d\varphi_t(Y(p)) &=& Y(\varphi_t(p)). \\
\end{array}
\]
A trivialization of $\xi$ over any disc in $M$ followed by the above trivialization of  $\xi^{\om_0}$ gives a trivialization of $T_p(\co)$ which is homotopic to the standard one. 

As a result we may finally conclude that $d\varphi_t$ on  $T_p(\co)$ is given by the ``standard'' differential of $\varphi_t$ on $\C^2$, namely
\[
d\varphi_t = 
 \left( \begin{array}{cc}
e^{it} & 0 \\
 0 & e^{it} \\
\end{array} \right).
\]
We obtain
\[
\Phi_{\C^2}(t) := d\varphi_t \arrowvert_{\C^2}
\]
as the path of symplectic matrices associated to the linearized Reeb flow of $\ga$ extended to $\C^2 \setminus \{ \bo \}$ for $T \in [0,T]$.  Similarly, we denote $\Phi_{\xi^{\om_0}}(t) $
to be the path of symplectic matrices associated to the linearized Reeb flow of $\ga_p$ for  $T \in [0,T]$ restricted on the symplectic complement of $\xi$.  

Then the naturality, homotopy, and product properties of the Conley-Zehnder index yield  
\[
\czm(\ga_p(t)): =  \czm \left(d\varphi(t) \arrowvert_{\xi} \right) = \czm\left(\Phi_{\C^2}(t)\right) -  \czm\left(\Phi_{\xi^{\om_0}}(t)\right).
\]
Since 
\[
\begin{array}{lcr}
X(\varphi_t) & =& -i(e^{it}u, e^{it}v) \\
Y(\varphi_t) &=& (e^{it}u, e^{it}v) \\
\end{array}
\]
and
\[
\begin{array}{lcrcc}
d \varphi_{2k\pi}(X(p)) & =& -i(u, v) &= & X(p)\\
d \varphi_{2k\pi}(Y(p)) &=& (u,v) &=&Y(p) \\
\end{array}
\]
we obtain
\[
\Phi_{\xi^{\om_0}}(2 k\pi) = \id,
\]
Thus $\czm(\Phi_{\xi^{\om_0}}(2 k\pi))=0$.  With the help of Example \ref{rssphere} we obtain
\[
\czm(\ga_p(t)): =  \czm \left(d\varphi(t) \arrowvert_{\xi} \right) = \czm\left(\Phi_{\C^2}(t)\right) = 4 k.
\]
\end{proof}

By a covering trick we obtain Proposition \ref{spheremaslov}.

\begin{corollary}
Let $(V,\lambda)$ be the prequantization bundle over the monotone, simply connected closed symplectic manifold $(S^2, k \omega_0)$ for $k \in \Z_{>0}$.  Then $(V,\xi) = (L(k,1),\xi_{std}) $ and the $k$-fold cover of every simple orbit $\gamma$ is contractible and $\mu_{RS}^\Phi(\gamma^k) = 4$.  

\end{corollary}

\begin{proof}
The result follows from taking the $k$-fold cover of $V$ which reduces the setup of the proof of Proposition \ref{rssphere}.  
\end{proof}

Finally, using a convenient choice of constant trivialization as in \cite[\S 3.1, 4.2]{ggm1}, we compute the Robbin-Salamon index of fibers of prequantization bundles over surfaces $(\Sigma_g, \omega_{0})$  of genus $g\geq1$ wherein $[\omega_0]$ is primitive. 


\begin{lemma}\label{consttrivlem}
Let $(V,\lambda) \overset{\pi}{\rightarrow}(\Sigma_g, \omega_0)$ be a prequantization bundle over a surface of genus $g\geq 1$ with $[\omega_0]$ primitive.  Then for the constant trivialization $\Phi$ along the circle fiber $\gamma_p = \pi^{-1}(p) $, we obtain $\rs^\Phi(\gamma_p)=0$ and $\rs^\Phi(\gamma^k_p) =0$.
\end{lemma}

\begin{proof}
Let $\Gamma$ be a free homotopy class of $V$.  In order to define the Robbin-Salamon index of a simple (nondegenerate) Reeb orbit $\gamma_p$ with $[\gamma_p]=\Gamma$ we must fix a trivialization of $\xi|_{\gamma_p}$.  We can fix a trivialization up to homotopy of $\xi$ along a reference loop in the free homotopy class $\Gamma$.  Connecting $\gamma_p$ to the reference loop by a cylinder and extending the trivialization along the cylinder produces, up to homotopy, a well-defined trivialization of $\xi$ along $\gamma_p$.  When $c_1(\xi)$ is atoroidal we are able to guarantee that the resulting trivialization is independent of the cylinder. 

The choice of trivializations needs to be compatible under iteration.  However, for prequantization bundles over $\Sigma_g$ with $g \geq 1$, the classes $\Gamma^k$ for $k \in \Z_{\geq 1}$ are all distinct and nontrivial.     Thus we can fix a trivialization $\Phi$ of $\xi$ along a loop in the class $\Gamma$ then the trivialization for the class $\Gamma^k$ is obtained by taking the $k$-th iterate of $\Phi$.  

We take the reference loop to be the fiber $\gamma = \pi^{1}(p)$ over a point in the class $\Gamma$.   For any point  $q \in \pi^{-1}(p)$ a fixed trivialization of $T_p\Sigma_g$ allows us to trivialize $\xi_q$ as $\xi_q \cong T_p\Sigma$.  This trivialization is invariant under the linearized Reeb flow and can be thought of as a ``constant trivialization" over the orbit $\gamma_p$ because the linearized Reeb flow, with respect to this trivialization, is the identity map.  In regards to iterates, we use the $k$-th iteration of the fiber as the reference loop for $\Gamma^k$ and the reference trivialization associated to $\Gamma^k$ is still the constant trivialization.  

\end{proof}

\begin{remark}\em
The above choice of constant trivialization is compatible the regularity result of Propositions \ref{cylcoverssep} and \ref{cylcoverssep-par}.
\end{remark}

\subsection{Perturbed Reeb vector fields on prequantization bundles}\label{prequantdyn}
This section reviews the Reeb dynamics of prequantization bundles and completes the proof of Proposition \ref{czcontractible}, which gives us the formula for the Conley-Zehnder index of closed Reeb orbits of $R_\veps$ over critical points $p$ of $H$.  We begin with the following result.
\begin{proposition}
The Reeb vector field associated to $\lep=(1+\vepsilon \pi^*H)\lambda$ is given by
\begin{equation}
\label{perturbedreeb9}
R_{\vepsilon}=\frac{R}{1+\vepsilon \pi^*H} + \frac{\veps \widetilde{X}_H}{{(1+\vepsilon \pi^*H)}^{2}}.
\end{equation}
where $X_{ H}$ is a Hamiltonian vector field \footnote{We use the convention $\omega(X_H, \cdot) = dH.$} on $S^2$ and $\widetilde{X}_{H}$ its horizontal lift, i.e.
\[
 dh(q)\widetilde{X}_H(q) = X_{\vepsilon  H}(h(q)) \ \ \mbox{ and } \ \ \lo(\widetilde{X}_H)=0.
\]
\end{proposition}
\begin{proof}
We have the following splitting of $TM$ with respect to the contact form $\lo$,
\[
T_pM = \langle R(p) \rangle \oplus \xi_p.
\]
Thus we know that there exists $a, b \in \R$ and $Y$ where $\lo(Y)=0$ such that $
R_\veps = aR + bY.$ We will show that $a = \dfrac{1}{1+\vepsilon \pi^*H}$, $b= \dfrac{\veps}{{(1+\vepsilon \pi^*H)}^{2}}$ and $Y=\widetilde{X}_H$.  We know that $R_\veps$ is uniquely determined by the equations
\begin{equation}\label{reebyreeb}
\lep(R_\veps)=1,  \ \ \ 
d\lep(R_\veps, \cdot)=0. 
\end{equation}
That $a$ is of the desired form follows immediately from the first line of (\ref{reebyreeb}) as
\[
\lep(R_\veps) = (1+\veps \pi^*H)\lo(aR) + (1+\veps \pi^*H)\lo(bY) = (1+ \veps \pi^*H)\lo(aR) 
\]
We compute to find
\[
d\lep = (1+\veps \pi^*H)d\lep + \veps \pi^*dH \we \lo.
\]
Then
\[
\begin{array}{lcl}
d\lep(R_\veps, \cdot) &= & (1+ \veps \pi^*H)\left(d\lo(aR, \cdot) + d\lo(bY, \cdot) \right) \\
&&  \ \ \ + \  \veps \pi^*dH(aR) \lo(\cdot) - \veps \pi^*dH(\cdot)\lo(aR)  \\
&&  \ \ \ +  \  \veps \pi^*dH(bY) \lo(\cdot) - \veps \pi^*dH(\cdot)\lo(bY) , \\
\end{array}
\]
which reduces to
\begin{equation}\label{reebyreeb2}
\begin{array}{lcl}
d\lep(R_\veps, \cdot) &= & (1+ \veps \pi^*H) d\lo(bY, \cdot) + \veps \pi^*dH(aR)\lo(\cdot) \\
&& \ \ \ - \ \frac{ \veps}{(1+ \veps \pi^*H)} \pi^*dH(\cdot) +\veps \pi^*dH(bY) \lo(\cdot).  \\
\end{array}
\end{equation}
Lest we forget about the symplectic form downstairs, recall $d\lo = \pi^*\om$ and 
$\om(X_H, \cdot) = dH.$ Also we have that 
\[
\pi^*dH(\cdot) = \pi^*\om(X_H, \cdot) = d\lo( \widetilde{X}_H, \cdot), \ \ \ \ \
\pi^*dH(\cdot) \we \lo(\cdot) = d\lo(\widetilde{X}_H, \cdot) \we \lo(\cdot).
\]
Thus (\ref{reebyreeb2}) becomes
\[
\begin{array}{lcl}
d\lep(R_\veps, \cdot)  &= & (1+\veps \pi^*H) d\lo(bY, \cdot) + \veps d\lo(\widetilde{X}_H,aR)\lo(\cdot) \\
&& \ \ \ - \ \frac{ \veps}{(1+ \veps \pi^*H)} d\lo(\widetilde{X}_H,\cdot) +  \veps d\lo(\widetilde{X}_H,bY) \lo(\cdot)  \\
&=& (1+ \veps \pi^*H) d\lo(bY, \cdot)  -  \frac{ \veps}{(1+ \veps \pi^*H)} d\lo(\widetilde{X}_H,\cdot) + \veps d\lo(\widetilde{X}_H,bY) \lo(\cdot). \\
\end{array}
\]
So $d\lep(R_\veps, \cdot) =0$ precisely when $b=\dfrac{\veps}{(1+\pi^*H)^{2}}$ and $Y=\widetilde{X}_H$ as desired.
\end{proof}



\begin{lemma}
Fix a Morse function $H$ such that $|H|_{C^2}<1$.  For each $T>0$, there exists $\vepsilon >0$ such that all Reeb orbits with $\mathcal{A}(\gamma) < T$ are nondegenerate and project to critical points of $H$.  
\end{lemma}

\begin{proof}
We have 
\[
R_{\vepsilon}=\frac{R}{1+\vepsilon \pi^*H} + \frac{\veps \widetilde{X}_H}{{(1+\vepsilon \pi^*H)}^{2}}.
\]
The horizontal lift $\widetilde{X}_{H}$  is determined  by
\[
  dh(q)\widetilde{X}_H(q) = X_{\vepsilon  H}(h(q)) \ \ \mbox{ and } \ \ \lo(\widetilde{X}_H)=0.
\]
Thus those orbits which do not project to $p\in \mbox{Crit}(H)$ must project to $X_H$.  We have 
\[
\frac{\veps}{(1+\veps)^2} <  \frac{\veps}{{(1+\vepsilon \pi^*H)}^{2}} < \frac{\veps}{(1-\veps)^2}
\]
A Taylor series expansion shows that the $m$-periodic orbits of $X_H$ give rise to orbits of $ \frac{\veps \widetilde{X}_H}{{(1+\vepsilon \pi^*H)}^{2}}$ which are $\frac{C}{\veps}$-periodic for some $C$.  We note that $C$ and $m$ must be bounded away from 0 since $X_H$ is time autonomous.  Nondegeneracy of Reeb orbits $\gamma$ such that $\mathcal{A}(\gamma) <T$ follows from the proof of Theorem 13 in Appendix A of \cite{ABW}.

\end{proof}
 

\begin{remark}\label{acrem}\em
The action of a Reeb orbit $\gpk$ of $R_\veps$ over a critical point $p$ of $H$ is proportional to the length of the fiber, namely
\[
\cala(\gpk) = \int_{\gpk} \lep = 2k \pi (1+\vepsilon \pi^*H),
\]
because $\pi^*H$ is constant on critical points $p$ of $H$. 
\end{remark}



With these details in place we can finish the proof of Proposition \ref{czcontractible} in regard to the Conley-Zehnder indices of $\ga_p^k$, which we will prove are given by
\[
\czm(\gpk)=\mu_{RS}(\gamma^k)-1+\mbox{index}_p(H),
\]
where $\gamma$ is a degenerate Reeb orbit corresponding to the circle fiber of $V \to \Sigma$.  To do this we employ an argument similar to the one found in \cite{CFHW}, as follows.


\begin{proof}[Proof of Proposition \ref{czcontractible}]
First, we show that one can use $d\lambda$ instead of $d\left((1+\varepsilon \pi^*H)\lambda\right)$ in computing the Conley-Zehnder indices for closed Reeb orbits over critical points of $H$.  We have
\[
\begin{array}{lcl}
d\left( (1+\varepsilon \pi^*H)\lambda \right) \arrowvert_\xi &= & \left( d(\veps \pi^*H) \we \lambda + \veps \pi^*H d\lambda \right) \arrowvert_\xi \\
&=& \left( \veps \pi^*H d\lambda \right) \arrowvert_\xi. 
\end{array}
\]
This tells us that $\pi^*H$ is constant along Hopf fibers over critical points of $H$, which are precisely the nondegenerate Reeb orbits of interest to us.  

Consider the decomposition 
\[
T_{\widetilde{q}}(\R \times V) = \R \oplus \langle R_\veps(q) \rangle \oplus \xi_q,
\]
where $\widetilde{q}$ is the lift of $q$ under the projection map $pr: \R \times V\to V$.  Since $p=\pi(q)$ is a critical point of $H$ the linearization at $\widetilde{q}$ with respect to this decomposition is given by
\[
d\varphi^\veps_t(\widetilde{q}) = \left( \begin{array}{cc} \left( \begin{array}{cc}
1 &0 \\
0& 1 \\
\end{array} \right) & \\
& d\varphi^\veps_t \arrowvert_{\xi_{q}}  \\ \end{array} \right).
\]
We define
\[
\Phi_\veps(t) = d\varphi^\veps_t \arrowvert_{\xi_{q}},  
\]
to be the linearization of the perturbed flow $R_\veps$ restricted to $\xi_q$.  Note that when $h(q)=p$ is a critical point of $H$ then the Reeb orbits associated to $R_\veps$ are $2 k \pi (1 + \veps H(p))$-periodic.  Define
\[
T_k:=2 k \pi (1 + \veps H(p)).
\]
Let $\Phi(t)$ be the linearized flow of $R$ restricted to $\xi_q$ and
$\Psi_\veps(t)$  be the linearized flow of $\psi_t$ associated to $\widetilde{X}_{1+\vepsilon H}$.


The homotopy
\[
L(s,t)=\Phi_{s\veps}(t)\Psi_{(1-s)\veps}(t)
\]
connects with fixed end points the path $\Phi_\veps(t)$ to $\Phi(t)\Psi_\veps(t)$.  For small $\veps$ we know that these paths have ends in $\Sp^*(2)$, the set of $2 \times 2$ symplectic matrices with eigenvalues not equal to 1.  Using the homotopy
\[
K_0(s,t)= \left\{ \begin{array}{lc}
L(s, \frac{2t}{s+1}) & \mbox{ if } t \leq T_k \cdot \frac{s+1}{2} \\
L(2\frac{t}{T_k}-1,T_k)&   \mbox{ if } t \geq T_k \cdot \frac{s+1}{2} \\
\end{array} \right. 
\]
and the aforementioned properties of the Conley-Zehnder index and Robbin-Salamon index we obtain
\[
\rs(\Phi_\veps) = \rs(\Phi \Psi_\veps).
\]
Another homotopy,
\[
K_1(s,t)= \left\{ \begin{array}{lc}
\Phi(\frac{2t}{s+1})\Psi_\veps(st) & \mbox{ if } t \leq T_k \cdot \frac{s+1}{2} \\
\Phi(T_k)\Psi_\veps((s+2)t - (s+1))&   \mbox{ if } t \geq T_k \cdot \frac{s+1}{2} \\
\end{array} \right.
\]
for $(s,t) \in [0,1] \times [0, T_k]$ yields 
\[
\rs(\Phi \Psi_\veps) = \rs(\Phi) + \rs(\Phi(T_k)\Psi_\veps).
\]
As a result we obtain 
\[
\rs(\Phi_\veps)=\rs(\Phi) + \rs(\Phi(T_k)\Psi_\veps).  
\]

We also have 
\[
\pi_*(T_qV) = \pi_*(\xi_q) = T_p\Sigma.
\]
If we extend the flow $\psi_t$ of $X_H$ to the symplectization then $\ker d\psi_t = \{0 \}$ for $t>0$.  Thus the only contribution from (the lift of) $-X_H$ to the Robbin-Salamon index occurs at $t=0$.  From Lemma \ref{RSHam} we obtain
 \[
\rs(\Phi(T_k)\Psi_\vepsilon) = \mbox{ index}_pH - \frac{1}{2}\mbox{ dim}\Sigma = \mbox{ index}_pH -1,
\]
and the result follows.

\end{proof}


\section{Fun with filtrations} \label{filtration}
 In this section, we prove Theorem \ref{prequantch} which demonstrates that the cylindrical contact homology of  three dimensional prequantization bundles $S^1 \hookrightarrow (V,\lambda) \overset{\pi}{\rightarrow} (\Sigma,\omega)$ consists of infinitely many copies of $H_*(\Sigma; \Q)$. We first establish a correspondence between gradient flows of a Morse-Smale function $H$ on the base $\Sigma$ and pseudoholomorphic cylinders asymptotic to closed Reeb orbits associated to $\lep = (1+\vepsilon \pi^*H)\lo$ which are of the same multiplicity and project to critical points of $H$.  The second step in our proof uses the results of  Proposition \ref{cylcoverssep} as well as Sections \ref{invariance} and \ref{cz-section} to construct a chain complex filtered by the action and index.  Using direct limits and the established proportionality between the action and index with respect to a free homotopy class, we recover the cylindrical contact homology of 3-dimensional prequantization bundles.


\subsection{Correspondence between cylinders and gradient flow lines}

 Recall that the perturbed Reeb vector field associated to $\lep$ is given by
\[
R_{\vepsilon}=\frac{R}{1+\vepsilon h^*H} + \frac{ \vepsilon \widetilde{X}_H}{{(1+\vepsilon h^*H)}^{2}},
\]where $X_{ H}$ is a Hamiltonian vector field\footnote{We use the convention $\omega(X_H, \cdot) = dH.$} on $S^2$ and $\widetilde{X}_{ H}$ its horizontal lift.

From \cite[Theorems 7.3]{SZ} and \cite[Theorem 10.30]{wendl-sft} we have the following result, which is key to establishing the isomorphism between Hamiltonian Floer homology and singular homology \cite{F89}, \cite[Theorem 7.1]{SZ}.  

\begin{theorem}\label{morsetofloer}
Let $(\Sigma, \omega_0)$ be a closed oriented surface of genus $g$.  Suppose $H: \Sigma \to \R$ is a smooth Morse function such that $|H|_{C^2}<1,$ $J$ is an $\omega$-compatible almost complex structure, and the flow of $H$ with respect to $\omega(\cdot, J \cdot)$ is Morse-Smale.  Given $\varepsilon > 0$, let $H^{\vepsilon} : = \vepsilon H$ with Hamiltonian vector field $X_{H^\vepsilon}=\vepsilon X_H$, and consider the contact form   $\lep = (1+\vepsilon \pi^*H)\lo$ on the prequantiation bundle $S^1 \hookrightarrow (V,\lambda) \overset{\pi}{\rightarrow} (\Sigma,\omega)$.  Then for all $\vepsilon  > 0$ sufficiently small, the following hold.
\begin{enumerate}
\item[\em{(i)}] The simple $R_\vepsilon$-orbit $\gamma_p$ arising from any $p \in \mbox{\em Crit}(H)$ is nondegenerate and its Conley-Zehnder index relative to the constant trivialization $\Phi$ is given by
\[
\czm^\Phi(\gamma_p) = -1 + \mbox{\em index}_p(H).
\]
\item[\em{(ii)}] Any trajectory $x : \R \to \Sigma$ satisfying the negative gradient flow equation $\dot{x}=-\nabla H^{\vepsilon}$ gives rise to a Fredholm regular solution  
\[
\begin{array}{lrcl}
v: &\R \times S^1 &\to& \Sigma \\
& (s,t) &\to&  x(s) \\
\end{array}
\]
 of the time independent Floer equation
\begin{equation}\label{eqnfloer}
 \partial_s v + J(v) \left(\partial_t + X_{ H^\vepsilon}(v) \right) =0
\end{equation}
and the virtual dimensions of the spaces of Floer trajectories near $v$ and gradient flow trajectories near $x$ are the same.
\item[\em{(iii)}] Every 1-periodic orbit of $X_{H^\vepsilon}$ in $\Sigma$ is a constant loop at a critical point of $H$.
\item[\em{(iv)}] Every finite energy solution $v: \R \times S^1 \to \Sigma$ of \eqref{eqnfloer} is of the form $v(s,t)=x(s)$ for some negative gradient flow trajectory $x : \R \to \Sigma$.

\end{enumerate}
\end{theorem}

A couple of remarks are in order.

\begin{remark}\em
We have added to Chris Wendl's horror\footnote{In addition to \cite[Remark 10.32]{wendl-sft}, one should read his blog post, as well as the subsequent discussion on \emph{heretical} signs at \url{https://symplecticfieldtheorist.wordpress.com/2015/08/23/signs-or-how-to-annoy-a-symplectic-topologist/}} by using the convention $\omega(X_H, \cdot) = dH.$  As computed in Lemma \ref{RSHam}, this yields  Theorem \ref{morsetofloer}(i)  which has opposite sign from his  \cite[Theorem 10.30(i)]{wendl-sft}.
\end{remark}

\begin{remark}\em
Proofs of Theorem \ref{morsetofloer} (iv) typically impose the additional assumption that $|\mbox{index}_p(H)-\mbox{index}_q(H)| \leq 1$ for all pairs $p, q \in \mbox{Crit}(H)$ to avoid a discussion of gluing.  This assumption is sufficient for our needs in defining the filtered complex as we have independently shown for cylindrical contact homology that $\left(\partial_\pm^{EGH}\right)^2=0$.
\end{remark}

A fairly standard argument \cite[\S 7]{Sief2} yields the following correspondence between cylinders $u \in \M^{J_\vepsilon}(\ga_p^k, \ga_q^k)$ and gradient flow lines between the critical points $p$ and $q$.  Theorem \ref{morsetofloer} then implies that these cylinders correspond to  Floer trajectories  between the 1-periodic orbits of $H^\vepsilon$, which are the constant loops at the critical points $p$ and $q$.

 In our setting, we consider prequantization bundles over closed symplectic manifolds in lieu of trivial $S^1$-bundles.  However, as explained in \cite[\S 6.1]{moreno}, one can parametrize these $J_\vepsilon$-holomorphic cylinders as follows.  Let  $(\theta, y) \mapsto \theta^*y$ be the global $S^1$-action on the $S^1$ principal bundle $Y$ and define 
  \[
\begin{array}{llcl}
u_x^\vepsilon: & \R \times S^1 & \to & \R \times Y \\
&(s,t) & \mapsto & (a(s), t^*\widetilde{x}(s)) \\
\end{array}
\]
where $\dot{a}(s) = (1 + \vepsilon \pi^*H({x}(s)))$, $\widetilde{x}$ solves $\dot{\widetilde{x}} = \dfrac{ \widetilde{\nabla {H}^\vepsilon} (\widetilde{x}(s))}{1 + \vepsilon \pi^*H({x}(s))}$, where $\widetilde{\nabla {H}}$ is the lift of $\nabla H$ to $\xi$. This construction yields the following proposition.

\begin{proposition}\label{cyltofloer}
Assume the same hypotheses in Theorem \ref{morsetofloer} and that $p, q \in \mbox{\em Crit}(H)$ with $|\mbox{\em index}_p(H)-\mbox{\em index}_q(H)| \leq 1$.   Given $\vepsilon >0$ sufficiently small and any $\lep$-compatible almost complex structures $J_\vepsilon$, every negative gradient flow equation $\dot{x} = -\nabla H^\vepsilon(x)$ from $p$ to $q$ gives rise to a smooth Fredholm regular $J_\vepsilon$-holomorphic cylinder $u^\vepsilon_x \in {\M}^{J_\vepsilon}(\ga_p^k, \ga_q^k)$ in $\R \times V$ and the virtual dimensions of the spaces of pseudoholomorphic cylinders near $u_\vepsilon^x$ and gradient flow trajectories near $x$ are the same. 
\end{proposition}

\begin{proof}
When $H$ is small but not identically zero, the projection of the curve to $\Sigma$ is no longer holomorphic as in the proof of \cite[Theorem 3.1]{Sief2}.  However, one can appeal to the asymptotic behavior of holomorphic curves, along with intersection theory and the relationship between the Conley-Zehnder indices and extremal winding numbers as in the proof of \cite[Prop. 4.11, Thm. C.10]{moreno}

This is because of the $S^1$-invariance of the contact structure $\xi$ induced by $\lambda$ gives a natural trivialization $\Phi$ of $\xi$ along the Reeb orbits associated to $\lep$ which project to critical points of $H$.  That one can use the trivialization induced by $\lambda$ for $\lep$ is explained in the proof of Proposition \ref{czcontractible}. Finally, the formulas for the Conley-Zehnder indices established in Section \ref{cz-section} show that the virtual dimension of the moduli space of cylinders agrees with the virtual dimension of the space of gradient flow lines.  Regularity follows from Propositions \ref{aut-agree}, \ref{cylcoverssep}, and \ref{cylcoverssep-par}.
\end{proof}

Next, we prove uniqueness via the following result which we have adapted from \cite[Theorem 10.33]{wendl-sft}.  The original result is stated in terms of certain stable Hamiltonian structures on $\Sigma \times S^1$ rather for prequantization bundles over $\Sigma$, but it is still applicable.  

\begin{theorem}\label{noescape}
Assume the same hypotheses in Theorem \ref{morsetofloer} and that $p, q \in \mbox{\em Crit}(H)$ with $|\mbox{\em index}_p(H)-\mbox{\em index}_q(H)| \leq 1$.    Let $\vepsilon >0$ be sufficiently small and take any smooth family $J_\delta$ of $\lep$-compatible almost complex structures matching $J_\vepsilon$ at $\delta=0$. Then there exists $\delta_0 >0 $ such that every $J_\delta$-holomorphic cylinder for $\delta \in [0, \delta_0]$ with a positive end asymptotic to $\ga_p^k$ and a negative end asymptotic to $\ga_q^k$, both up to reparametrization, is in the same equivalence class of cylinders as $u_\delta^x$ in Proposition \ref{cyltofloer}.

\end{theorem}

\begin{proof}


If $\delta_n \to 0$ and $u_n$ is a sequence of $J_{\delta_n}$-holomorphic cylinders then first by the uniqueness of the asymptotic orbits, we can extract a subsequence for which all $u_n$ are asymptotic at both ends to closed Reeb orbits, up to reparametrization, of the form ${\ga_p^k}$ for $p \in \mbox{Crit}(H)$ as $n\to \infty$.  Remark \ref{acrem} shows that we have a uniform bound on the energies of our curves.

 Since all $R_\vepsilon$-orbits in a fixed homotopy class are nondegenerate, one can now conclude that $u_n$ has a subsequence convergent to a finite-energy $J_\vepsilon$-holomorphic building $\mathbf{u}_\infty$ consisting only of cylinders.  When $g(\Sigma) \geq1$ there are no contractible orbits and when $\Sigma=S^2$, the exclusion of noncylindrical levels follows from the lemmata of \S \ref{obs-chain}.
 
 The levels of the building are asymptotic to orbits of the form ${\ga_p^k}$ for $p \in \mbox{Crit}(H)$, projecting to solutions of the $H^\vepsilon$-Floer equation.  Since $\nabla H$ is Morse-Smale and indices of critical points can only differ by at most 1, the building $\mathbf{u}_\infty$ can have at most one nontrivial level $u_\infty$, which is in the same equivalence class as $u_\vepsilon^x$ by Proposition \ref{cyltofloer}.  That there are no other levels follows from the calculations in \S \ref{obs-chain}. This implies $u_n \to u_\infty$ and because $u_\infty$ is Fredholm regular and the implicit function theorem gets us the rest of the way.  



\end{proof}

\subsection{The filtered chain complex}\label{filteredchsec}

In this section we use the proportionality between the action and Conley-Zehnder index of the Reeb orbits with respect to a  fixed free homotopy class to compute the filtered cylindrical contact homology for prequantization bundles over closed oriented surfaces.  These Morse-Bott methods are in a similar spirit as those found in the realm of symplectic homology, see  \cite{CFHW, gutt2}, \cite[\S 4]{BOduke}, and \cite[Appendix B]{vkk}.

With respect to a fixed free homotopy class $\Gamma$, we first filter cylindrical contact homology by action and investigate a further filtration by the Conley-Zehnder index. We can filter the complex by action because the cylindrical contact homology differential is action decreasing \cite[Lemma 2.18]{jo1}.  We will further assume that the critical values of the Morse-Smale function $H$ on $\Sigma$ are negative and close to $0$.  

We fix a particular choice of coherent orientations on the (filtered) complex by requiring that cylinders which correspond to the Morse flow lines of the $C^2$ small Morse function $H$ on the base be counted in the same way as the version of Morse homology that is isomorphic to singular homology.  Further details on the choice of such coherent orientations may be found in \cite[Appendix B.0.2]{vkk}.  In particular the local coefficient system is trivial by \cite[Lemma B.7, Remark B.8]{vkk}

From Remark \ref{acrem} we saw that action of a Reeb orbit $\gpk$ of  $R_\veps$ over a critical point $p$ of $H$ was proportional to the length of the fiber, namely
\[
\cala(\gpk) = 2k \pi (1+\vepsilon \pi^*H).
\]
As a result, we can introduce a filtration on the whole complex.  We use    bold face $\mathbf{p}$ and $\mathbf{q}$ in place of the conventions $p$ and $q$ for the bigrading in \cite[\S 5.4]{weibel} to avoid confusion with critical points of $H$.  With respect to a free homotopy class $\Gamma$, the filtration is then,
\[
F_{\mathbf{p}}C_{\mathbf{q}}(V, \lep,\Gamma) = \left\{\gamma \in C_{\mathbf{q}}(M, \lep, \Gamma) \ \bigg \arrowvert \ \mathcal{A}(\gamma)< 2\mathbf{p} \pi \right\} 
\]
Implicit in this filtration is the choice of positive $\vepsilon$, which must be chosen to be sufficiently small such that all orbits $\gamma$ with $\mathcal{A}(\gamma) \leq 2\mathbf{p} \pi$ project to a critical point of $H$.  
Since we have restricted ourselves to a fixed homotopy class this filtration exhausts the complex in finitely many steps.  

Let $L_\vepsilon$ be the minimal value of $2\mathbf{p} \pi$ such that
$F_{\mathbf{p}}C_{\mathbf{q}}(M, \lep,\Gamma) = C_{\mathbf{q}}(M, \lep,\Gamma)$  for all $\mathbf{p}  > \frac{L_\vepsilon}{2\pi}$. The $E^0$-page of this spectral sequence is given by
\[
E_{\mathbf{pq}}^0 = F_{\mathbf{p}}C_{\mathbf{p+q}}(V, \lep,\Gamma)/ F_{\mathbf{p}-1}C_{\mathbf{p+q}}(V, \lep,\Gamma).
\]

\begin{lemma}\label{page0diff}
For $\vepsilon$ chosen sufficiently small so that the assumptions of Theorem \ref{morsetofloer} also apply to 
\[
F_{\mathbf{p}}C_{\mathbf{q}}(V, \lep,\Gamma) = C_{\mathbf{q}}(V, \lep,\Gamma)
\]
 for all $\mathbf{p}  > \frac{L_\vepsilon}{2\pi}$ hold. Then the differential for $\mathbf{p}>0$ on the the $E^0$-page 
 \[
E_{\mathbf{pq}}^0 = F_{\mathbf{p}}C_{\mathbf{p+q}}(V, \lep,\Gamma)/ F_{\mathbf{p}-1}C_{\mathbf{p+q}}(V, \lep,\Gamma).
\]
agrees with the Morse differential on the base $\Sigma$.
\end{lemma}

\begin{proof}
The differential on the $E^0$-page only counts cylinders that decrease the action level less then $2\pi$.  Moreover, this differential is well-defined and squares to 0 because it agrees with the usual cylindrical contact homology differential, so the results of \cite[Theorems 1.10, 1.12, Remark 1.13]{jo1} apply. 

The results of \S \ref{cz-section}, in particular Proposition \ref{czcontractible} and Lemmas \ref{rssphere}, \ref{consttrivlem}, show that with respect to a fixed homotopy class $\Gamma$, there is also a simultaneous filtration by index.  When the genus of $\Sigma$ is greater than $0$, the Puppe sequence shows there are no contractible Reeb orbits, and the action of any Reeb orbit with action less than $L_\vepsilon$ determines its covering multiplicity and hence its free homotopy class.   Proposition \ref{cyltofloer} and Theorem \ref{noescape} imply that  for $\mathbf{p}>0$, these cylinders are counted by the Morse differential on the base $\Sigma$.
\end{proof}

With respect to a fixed free homotopy class $\Gamma$, Lemma \ref{page0diff} permits us to conclude that the $E^1$-page is given by copies of the Morse homology of $\Sigma$ with appropriate degree shifts corresponding to the SFT-grading of $\gamma,$ $\ |\gamma| = \czm(\ga) -1.$
  The proportionality between the action and the SFT-grading  allows us to simultaneously filter by the SFT-grading.  In particular, for an appropriate choice of $\vepsilon$, 
\[
\begin{array}{lcl}
F_{\mathbf{p}}C_{\mathbf{q}}(V, \lep,\Gamma) &=&  \left\{\gamma \in C_{\mathbf{q}}(V, \lep,\Gamma) \ \bigg \arrowvert \ \mathcal{A}(\gamma) \leq 2\mathbf{p} \pi \right\} \\
& =& \left\{\gamma \in C_{\mathbf{q}}(V, \lep,\Gamma) \ \bigg \arrowvert \ |\ga|:= \czm(\gamma) -1 \leq 4\mathbf{p} +1 \right\} 
\end{array}
\]
As a result, the SFT-grading filtration with repsect to a free homotopy class excludes differentials other than the one coming from the action filtration.  Hence, the spectral sequence degenerates at the $E^2$-page.  Such an argument is also used to obtain a spectral sequence for positive $S^1$-equivariant symplectic homology in the proof of Proposition 3.7 Part (II) of \cite{BOSHCH}.

To take direct limits of the filtered groups, we need the following lemma regarding continuation maps.  

\begin{lemma}\label{continuationlem}
For $0< \vepsilon' < \vepsilon$ sufficiently small and under the assumptions of Theorem \ref{page0diff}, the continuation map $c_{\vepsilon \vepsilon'}$ is well-defined and induces a morphism of spectral sequences associated to a fixed free homotopy class
\[
c_{\vepsilon \vepsilon'}^r: E_{\mathbf{pq}}^r(\lep) \to E_{\mathbf{pq}}^r(\lambda_{\vepsilon'})
\]
for $r=0,1$.
\end{lemma}

\begin{proof}
That the map $c_{\vepsilon \vepsilon'}^r$ is well-defined follows from the same reasoning as to why the chain map is well-defined; see Proposition \ref{0compact} and \S \ref{chainmap-sec}.  To see that it induces a morphism between spectral sequences we note that the spectral sequence $E_{pq}^r(\lep)$ is bounded with maximal filtration degree $\frac{L_\vepsilon}{2\pi}$, so it converges not any later than on that page.  

By construction, the gap between the last non-zero column of $E^0_{\mathbf{pq}}(\lep)$ and the first new column of $E^0_{\mathbf{pq}}(\lambda_{\vepsilon'})$ is {larger than $\frac{L_\vepsilon}{2\pi}$.}   This is because  for $\mathbf{p} > \frac{L_{\vepsilon_j}}{2\pi} $   we can choose  $\vepsilon_i < \vepsilon_j$ such that $F_\mathbf{p}C_*^{EGH}(V,\lambda_{\vepsilon_i}) = C_*^{EGH}(V,\lambda_{\vepsilon_j})$ for $\mathbf{p} = \frac{L_{\vepsilon_j}}{2\pi} ,  \frac{L_{\vepsilon_j}}{2\pi}+1.$ Thus, the first new non-zero column, if any, of $E^0_{\mathbf{p}*}(\lambda_{\vepsilon_i}, J_{\vepsilon_i})$ appears for filtration degree $\mathbf{p}$ at least $\frac{L_{\vepsilon_j}}{2\pi}+2$.

For $\mathbf{p} \leq \frac{L_\vepsilon}{2\pi}$  we see that $c_{\vepsilon \vepsilon'}^r$ commutes with the differential because of Proposition \ref{cyltofloer} and Theorem \ref{noescape} as well as \S \ref{chainmap-sec}.  For $\mathbf{p} > \frac{L_\vepsilon}{2\pi}$ the continuation map $c_{\vepsilon \vepsilon'}^r$ vanishes, as does the differential $d^{r,\vepsilon}$.   
\end{proof}

Since  $CH^{EGH,L}_*(V,\lambda_\vepsilon, J, \Gamma)$ is the homology of the chain complex generated by copies of  $H_*^{\mbox{\tiny Morse}}(\Sigma, H; \Q)$ with $\partial_\pm = \partial_H^{\mbox{\tiny Morse}}$ on each copy, after taking the direct limit,
\[
CH^{EGH}(V,\ker \lambda,\Gamma) := \varinjlim_\vepsilon CH^{EGH,L_\vepsilon}(V, \lep, J_\vepsilon,\Gamma). 
\]
 we obtain
\[
CH^{EGH}(V,\ker \lambda, \Gamma)=E^1_{\mathbf{pq}}(CH^{EGH}(V,\ker \lambda, \Gamma)) = \left\{ 
\begin{array}{lr}
\displaystyle \bigoplus_{{\gamma \mbox{ \tiny good, } [\gamma]=\Gamma }} H_{\mathbf{p}+\mathbf{q}-|\gamma|}(\Sigma;\Q) & \mathbf{p}>0 \\
0 & \mathbf{p}\leq 0.\\
\end{array} \right.
\]
This yields the proof of Theorem \ref{prequantch}.  







\noindent \textsc{Jo Nelson \\  Rice University}\\
{\em email: }\texttt{jo.nelson@rice.edu}\\




\end{document}